\documentclass[10pt]{article}
\usepackage{amsthm}
\usepackage{amsmath}
\usepackage{amssymb}
\usepackage{makeidx}

\theoremstyle{definition}
\newtheorem{definition}{Definition}[section]

\newtheorem{example}[definition]{Example}

\newtheorem{remark}[definition]{Remark}

\theoremstyle{plain}
\newtheorem{theorem}[definition]{Theorem}
\newtheorem{lemma}[definition]{Lemma}
\newtheorem{proposition}[definition]{Proposition}
\newtheorem{corollary}[definition]{Corollary}

\numberwithin{equation}{section}

\def\N{{\mathbb N}}

\begin{document}
\title{Going Up and Lying Over in Congruence--modular Algebras}
\author{George GEORGESCU and Claudia MURE\c SAN\thanks{Corresponding author.}\\ \footnotesize University of Bucharest\\ \footnotesize Faculty of Mathematics and Computer Science\\ \footnotesize Academiei 14, RO 010014, Bucharest, Romania\\ \footnotesize Emails: georgescu.capreni@yahoo.com; c.muresan@yahoo.com, cmuresan@fmi.unibuc.ro}
\date{\today }
\maketitle

\begin{abstract} In this paper, we extend properties Going Up and Lying Over from ring theory to the general setting of congruence--modular equational classes, using the notion of prime congruence defined through the commutator. We show how these two properties relate to each other, prove that they are preserved by finite direct products and quotients and provide algebraic and topological characterizations for them. We also point out many kinds of varieties in which these properties always hold.\end{abstract}

\section{Introduction}

\label{introduction}

Properties Going Up (GU) and Lying Over (LO) reflect the behaviour of commutative ring extensions with respect to finite chains of prime ideals. An extension of commutative rings $A\subseteq B$ fulfills {\em GU} iff, for any prime ideals $P,Q$ of $A$ and $P^{\prime }$ of $B$, if $P\subseteq Q$ and $P^{\prime }\cap A=P$, then there exists a prime ideal $Q^{\prime }$ of $B$ such that $P^{\prime }\subseteq Q^{\prime }$ and $Q^{\prime }\cap A=Q$; the extension $A\subseteq B$ fulfills {\em LO} iff, for any prime ideal ${\cal P}$ of $A$, there exists a prime ideal $P^{\prime }$ of $B$ such that $P^{\prime }\cap A=P$. These two conditions have be generalized from ring embeddings to arbitrary morphisms of commutative rings: a morphism $f:A\rightarrow B$ between two commutative rings shall fulfill GU, respectively LO, iff the extension $f(A)\subseteq B$ fulfills GU, respectively LO. By the Cohen--Seidenberg Theorem \cite[Theorem 5.11]{amd}, integral ring extensions fulfill GU and LO.

GU and LO--type conditions have been studied for the prime ideals of some algebraic structures related to logic: bounded distributive lattices \cite[p. 773]{loqu}, MV--algebras \cite{bel} and BL--algebras \cite{rada}. By \cite{bel}, respectively \cite{rada}, any MV-algebra, respectively BL--algebra morphism fulfills GU and LO.

In the present paper, we study properties GU and LO for morphisms in certain kinds of varieties of universal algebras, relating them to congruences instead of ideals. In order to define GU and LO in this general setting, we need a notion of prime congruence; we have chosen the prime congruences introduced through the notion of commutator, which can be defined in congruence--modular varieties \cite[p. 82]{fremck}. \cite{agl} shows that the prime spectra of algebras in semi--degenerate congruence--modular varieties have rich enough properties for developping an interesting mathematical theory concerning GU and LO.

While the inverse images of prime ideals through morphisms of commutative rings, bounded distributive lattices, MV--algebras and BL--algebras are again prime ideals, the same does not go for prime congruences in algebras from congruence--modular varieties, in general, and, since this property makes the theory of conditions GU and LO work for these particular kinds of algebras, we have had to restrict our research to morphisms that fulfill this property for prime congruences, which we have called {\em admissible morphisms}. However, in many kinds of varieties, all morphisms are admissible; we list such varieties in the final section of this paper.

In Section \ref{preliminaries} of our paper, we recall some notions from universal algebra and commutator theory, including properties of prime congruences. The results in the following sections are new and original, excepting those that we cite from other papers.

Section \ref{admmorph} has an introductory purpose; here we present the notion of {\em admissible morphism}, along with some properties of this kind of morphisms, most of which we cite from \cite{euadm}. We also give some examples.

In Section \ref{guandlo}, we define properties GU and LO for admissible morphisms in congruence--modular varieties, we provide examples, along with some simple characterizations that we need in what follows, and obtain the main results on these properties, such as the fact that their study can be reduced to embeddings, that surjectivity implies GU, GU implies LO in semi--degenerate congruence--modular equational classes, but the converses of these implications do not hold, GU is preserved by composition, while LO needs enforcing surjectivity on one of the morphisms or injectivity on the other for it to be preserved in the composition of those morphisms.

In some particular cases, concerning the structures of the posets of the prime congruences of the algebras in question, GU always holds or LO implies GU. Such cases are pointed out in Section \ref{gulopartic}. Here we notice that any morphism in the class of bounded distributive lattices is admissible and fulfills GU, a result which we also generalize both in this section and in Section \ref{varieties}.

In Section \ref{prodsum}, we prove that GU and LO are preserved by finite direct products in semi--degenerate congruence--modular varieties and in congruence--distributive varieties, as well as by finite ordinal sums in the class of bounded lattices and in any congruence--modular variety of bounded orderred structures that fulfills certain conditions.

In Section \ref{charguandlo}, we prove that GU and LO are preserved by quotients, and obtain a series of algebraic and topological characterizations for GU and LO, which lead to further results on the relationships between these two properties.

In Section \ref{varieties}, we study admissibility, GU and LO in different kinds of congruence--modular equational classes. We prove that all morphisms are admissible in varieties having a system of congruence intersection terms without parameters, among which there are congruence--distributive varieties with the compact intersection property, which in turn include filtral varieties, discriminator varieties, bounded distributive lattices and residuated lattices. As for varieties in which all admissible morphisms fulfill GU and LO, it turns out that they include semi--degenerate varieties with equationally definable principal congruences, which in turn include semi--degenerate filtral varieties, semi--degenerate discriminator varieties, bounded distributive lattices and residuated lattices and many other varieties which are important in the algebra of logic. If we put together these results, we obtain a set of varieties in which all morphisms are admissible and fulfill GU and LO, a fact which generalizes the results on MV--algebras and BL--algebras from \cite{bel} and \cite{rada}, respectively, but also includes many other interesting cases, such as bounded distributive lattices, residuated lattices, semi--degenerate filtral varieties and semi--degenerate discriminator varieties.

\section{Preliminaries}
\label{preliminaries}

In this section, we recall some properties on congruences in universal algebras and a series of results from commutator theory; we shall provide short proofs for those which are least commonly used. For the notions on universal algebras that we use in the sequel, we refer the reader to \cite{agl}, \cite{bur}, \cite{gralgu}, \cite{koll}. For those on lattices, see \cite{bal}, \cite{blyth}, \cite{cwdw}, \cite{gratzer}, \cite{schmidt}. For a further study of commutator theory, we recommend \cite{agl}, \cite{fremck}, \cite{koll}, \cite{owe}.

We shall denote by $\N $ the set of the natural numbers and by $\N ^*=\N \setminus \{0\}$. For any set $M$, we shall denote by $|M|$ the cardinality of $M$, by ${\cal P}(M)$ the set of the subsets of $M$, by ${\rm Eq}(M)$ the set of the equivalence relations on $M$, by $\Delta _M=\{(x,x)\ |\ x\in M\}\in {\rm Eq}(M)$ and by $\nabla _M=M^2\in {\rm Eq}(M)$. If $M$ is non--empty and $\pi $ is a partition of $M$, then we shall denote by $eq(\pi )$ the equivalence on $M$ that corresponds to $\pi $; if $n\in \N ^*$ and $\{M_1,\ldots ,M_n\}$ is a finite partition of $M$, $eq(\{M_1,\ldots ,M_n\})$ shall also be denoted by $eq(M_1,\ldots ,M_n)$. For any $\theta \in {\rm Eq}(M)$, any $a\in M$, $V\subseteq M$ and $W\subseteq M^2$, $a/\theta $ will denote the equivalence class of $a$ with respect to $\theta $, $V/\theta =\{x/\theta \ |\ x\in V\}$, $W/\theta =\{(x/\theta ,y/\theta)\ |\ x,y\in M,(x,y)\in W\}$ and $p_{\theta }:M\rightarrow M/\theta $ shall be the canonical surjection.

Let $I$ be a non--empty set and $(X_i)_{i\in I}$ and $(Y_i)_{i\in I}$ be families of sets. If $X\subseteq \displaystyle \prod _{i\in I}X_i$, then by $a=(a_i)_{i\in I}\in X$ we mean $a_i\in X_i$ for all $i\in I$, such that $a\in X$. If $f_i:X_i\rightarrow Y_i$ for all $i\in I$, then $\displaystyle \prod _{i\in I}f_i:\prod _{i\in I}X_i\rightarrow \prod _{i\in I}Y_i$ shall have the usual componentwise definition and, in the particular case when $I=\overline{1,n}$ for some $n\in \N ^*$ and $f_1=\ldots =f_n=f$, then we denote $\displaystyle \prod _{i=1}^nf_i=f^n$. If $\theta _i\in {\rm Eq}(X_i)$ for all $i\in I$, then we denote by $\displaystyle \prod _{i\in I}\theta _i=\{((x_i)_{i\in I},(y_i)_{i\in I})\ |\ (\forall \, i\in I)\, ((x_i,y_i)\in \theta _i)\}$. If $M$ and $N$ are sets and $f:M\rightarrow N$, then the direct image of $f^2$ shall simply be denoted by $f$, and $(f^2)^{-1}$, the inverse image of $f^2$, shall be denoted by $f^*$. Clearly, if $f$ is injective, then so is $f^2$, thus $f^*:{\cal P}(N^2)\rightarrow {\cal P}(M^2)$ is surjective, and, if $f$ is surjective, then so is $f^2$, thus $f^*$ is injective. Trivially, $f^*(\nabla _N)=\nabla _M$. We shall denote by ${\rm Ker}(f)=f^*(\Delta _N)$: the {\em kernel} of $f$. Notice that, if $Q$ is a set and $g:N\rightarrow Q$, then $(g\circ f)^*=((g\circ f)^2)^{-1}=(g^2\circ f^2)^{-1}=(f^2)^{-1}\circ (g^2)^{-1}=f^*\circ g^*$.

Whenever there is no danger of confusion, any algebra shall be designated by its support set. All algebras shall be considerred non--empty; by {\em trivial algebra} we mean one--element algebra, and by {\em non--trivial algebra} we mean algebra with at least two distinct elements. Any quotient algebra and any direct product of algebras shall be considerred with the operations defined canonically. Sometimes, for brevity, we shall denote by $A\cong B$ the fact that two algebras $A$ and $B$ of the same type are isomorphic.

For any $n\in \N ^*$, we shall denote the $n$--element chain by ${\cal L}_n$. We shall denote by ${\cal D}$ the diamond and by ${\cal P}$ the pentagon. We shall abbreviate the {\em ascending chain condition} for lattices by {\em ACC}. For any lattice $L$ and any $x\in L$, we shall denote by $[x)$ the principal filter of $L$ generated by $x$: $[x)=\{y\in L\ |\ x\leq y\}$.

Throughout this paper, $\tau $ shall be a type of universal algebras, ${\cal C}$ shall be an equational class of algebras of type $\tau $ and ${\bf L}_{\tau }$ shall be the first order language associated to $\tau $. For any term $t$ in ${\bf L}_{\tau }$ and any member $M$ of ${\cal C}$, we shall denote by $t^M$ the derivative operation of $M$ associated to $t$.

Throughout the rest of this section, $A$ shall be an algebra from ${\cal C}$. We shall denote by ${\rm Con}(A)$ the set of the congruences of $A$; for each $X\subseteq A^2$, we shall denote by $Cg_A(X)$ the congruence of $A$ generated by $X$; for every $a,b\in A$, $Cg_A(\{(a,b)\})$ is also denoted by $Cg_A(a,b)$ and called the {\em principal congruence} of $A$ generated by $(a,b)$. Let $\phi \in {\rm Con}(A)$; $\phi $ is said to be {\em finitely generated} iff $\phi =Cg_A(X)$ for some finite subset $X$ of $A^2$; $\phi $ is called a {\em proper congruence} of $A$ iff $\phi \neq \nabla _A$. The {\em maximal congruences} of $A$ are the maximal elements of $({\rm Con}(A)\setminus \{\nabla _A\},\subseteq )$; the set of the maximal congruences of $A$ is denoted by ${\rm Max}(A)$ and called the {\em maximal spectrum} of $A$. $({\rm Con}(A),\vee ,\cup ,\Delta _A,\nabla _A)$ is a bounded lattice, orderred by set inclusion, where, for all $\phi ,\psi \in {\rm Con}(A)$, $\phi \vee \psi =Cg_A(\phi \cup \psi )$; moreover, ${\rm Con}(A)$ is a complete lattice, in which, for any family $(\phi 
_i)\subseteq {\rm Con}(A)$, $\displaystyle \bigvee _{i\in I}\phi _i=Cg_A(\bigcup _{i\in I}\phi _i)$. Obviously, $A$ is non--trivial iff $\Delta _A\neq \nabla _A$. The algebra $A$ is said to be {\em congruence--modular}, respectively {\em congruence--distributive}, iff the lattice ${\rm Con}(A)$ is modular, respectively distributive. The equational class ${\cal C}$ is said to be {\em congruence--modular}, respectively {\em congruence--distributive}, iff all algebras from ${\cal C}$ are congruence--modular, respectively congruence--distributive. The class of lattices and that of residuated lattices are congruence--distributive; that of commutative rings is congruence--modular and it is not congruence--distributive.

If $I$ is a non--empty set, $(A_i)_{i\in I}$ and $(B_i)_{i\in I}$ are families of algebras in ${\cal C}$ and, for all $i\in I$, $\theta _i\in {\rm Con}(A_i)$ and $f_i:A_i\rightarrow B_i$ is a morphism in ${\cal C}$, then, clearly: $\displaystyle \prod _{i\in I}\theta _i\in {\rm Con}(\prod _{i\in I}A_i)$ and $\displaystyle \prod _{i\in I}f_i$ is a morphism in ${\cal C}$.

Now let $B$ be an algebra from ${\cal C}$ and $f:A\rightarrow B$ be a morphism in ${\cal C}$ and let $\phi \in {\rm Con}(A)$ and $\psi \in {\rm Con}(B)$. Then $f^*(\psi )\in {\rm Con}(A)$, in particular ${\rm Ker}(f)\in {\rm Con}(A)$, and, clearly, $f(f^*(\psi ))=\psi \cap f(A^2)$, so, if $f$ is surjective, then $f(f^*(\psi ))=\psi $. Also, $f(\phi )\in {\rm Con}(f(A))$, thus, if $f$ is surjective, then $f(\phi )\in {\rm Con}(B)$ and $f^*(f(\phi ))=\phi $ since, in this case, $f^*$ is injective.

For any $\theta \in {\rm Con}(A)$, $p_{\theta }$ is a surjective morphism in ${\cal C}$ and the mapping $\gamma \mapsto p_{\theta }(\gamma )=\gamma /\theta $ sets a bounded lattice isomorphism from $[\theta )$ to ${\rm Con}(A/\theta )$, so ${\rm Con}(A/\theta )=\{\gamma /\theta \ | \gamma \in [\theta )\}$ and, for all $\gamma \in [\theta )$, $p_{\theta }^*(p_{\theta }(\gamma ))=p_{\theta }^*(\gamma /\theta )=\gamma $, thus ${\rm Ker}(p_{\theta })=p_{\theta }^*(\Delta _{A/\theta })=p_{\theta }^*(\theta /\theta )=\theta $. Notice that, for any $\gamma \in [\theta )$ and any $a,b\in A$, the following hold: $(a/\theta ,b/\theta )\in \gamma /\theta $ iff $(p_{\theta }(a),p_{\theta }(b))\in p_{\theta }(\gamma )$ iff $(a,b)\in p_{\theta }^*(p_{\theta }(\gamma ))$ iff $(a,b)\in \gamma $; from this or directly from the fact that the map above is a lattice isomorphism, it follows that, for any $\alpha ,\beta \in [\theta )$: $\alpha /\theta =\beta /\theta $ iff $\alpha =\beta $, and: $\alpha /\theta \subseteq \beta /\theta $ iff $\alpha \subseteq \beta $. From the above it follows that ${\rm Max}(A/\theta )=\{\mu /\theta \ |\ \mu \in {\rm Max}(A),\theta \subseteq \mu \}=p_{\theta }([\theta )\cap {\rm Max}(A))$.

\begin{theorem}{\rm \cite{fremck}} If ${\cal C}$ is congruence--modular, then, for each member $M$ of ${\cal C}$, there exists a unique binary operation $[\cdot ,\cdot ]_M$ on ${\rm Con}(M)$ such that, for all $\alpha ,\beta \in {\rm Con}(M)$, $[\alpha ,\beta ]_M=\min \{\mu \in {\rm Con}(M)\ |\ \mu \subseteq \alpha \cap \beta $ and, for any member $N$ of ${\cal C}$ and any surjective morphism $f:M\rightarrow N$ in ${\cal C}$, $\mu \vee {\rm Ker}(f)=f^*([f(\alpha \vee {\rm Ker}(f)),f(\beta \vee {\rm Ker}(f))]_N)\}$.\label{wow}\end{theorem}

\begin{remark} Notice that, since it refers to surjective functions $f$, the last equality from Theorem \ref{wow} implies: $f(\mu \vee {\rm Ker}(f))= [f(\alpha \vee {\rm Ker}(f)),f(\beta \vee {\rm Ker}(f))]_N$.\label{r2.5}\end{remark}

\begin{definition} If ${\cal C}$ is congruence--modular, then, for any member $M$ of ${\cal C}$, the operation $[\cdot ,\cdot ]_M:{\rm Con}(M)\times {\rm Con}(M)\rightarrow {\rm Con}(M)$ from Theorem \ref{wow} is called the {\em commutator of $M$}, and, for any $\alpha ,\beta \in {\rm Con}(M)$, $[\alpha ,\beta ]_M$ is called the {\em commutator of $\alpha $ and $\beta $}.\label{1.1}\end{definition}

\begin{theorem}{\rm \cite{bj}} If ${\cal C}$ is congruence--distributive, then, in each member of ${\cal C}$, the commutator coincides to the intersection of congruences.\label{distrib}\end{theorem}

Throughout the rest of this section, ${\cal C}$ shall be congruence--modular.

\begin{remark} For any $\alpha ,\beta \in {\rm Con}(A)$, we have:\begin{itemize}
\item $[\alpha ,\beta ]_A\subseteq \alpha \cap \beta $;
\item for any algebra $B$ from ${\cal C}$ and any surjective morphism $f:A\rightarrow B$, $[\alpha ,\beta ]_A\vee {\rm Ker}(f)=f^*([f(\alpha \vee {\rm Ker}(f)),f(\beta \vee {\rm Ker}(f))]_B)$, which implies: $f([\alpha ,\beta ]_A\vee {\rm Ker}(f))=[f(\alpha \vee {\rm Ker}(f)),f(\beta \vee {\rm Ker}(f))]_B$;
\item for all $\theta \in {\rm Con}(A)$, $p_{\theta }:A\rightarrow A/\theta $ is a surjective morphism, so we may take $B=A/\theta $ and $f=p_{\theta }$ in the above, and we obtain: $([\alpha ,\beta ]_A\vee \theta )/\theta =[(\alpha \vee \theta )/\theta,(\beta \vee \theta )/\theta ]_{A/\theta }$; in particular, if $\alpha ,\beta \in [\theta )$, then $([\alpha ,\beta ]_A\vee \theta )/\theta =[\alpha /\theta,\beta /\theta ]_{A/\theta }$.\end{itemize}\label{r2.8}\end{remark}

\begin{proposition}{\rm \cite{fremck}} The commutator is:\begin{itemize}
\item commutative, that is $[\alpha ,\beta ]_A=[\beta ,\alpha ]_A$ for all $\alpha ,\beta \in {\rm Con}(A)$;
\item increasing in both arguments, that is, for all $\alpha ,\beta ,\phi ,\psi \in {\rm Con}(A)$, if $\alpha \subseteq \beta $ and $\phi \subseteq \psi $, then $[\alpha ,\phi ]_A\subseteq [\beta ,\psi ]_A$;
\item distributive in both arguments with respect to arbitrary joins, that is, for any families $(\alpha _i)_{i\in I}$ and $(\beta _j)_{j\in J}$ of congruences of $A$, $\displaystyle [\bigvee _{i\in I}\alpha _i,\bigvee _{j\in J}\beta _j]_A=\bigvee _{i\in I}\bigvee _{j\in J}[\alpha _i,\beta _j]_A$.\end{itemize}\label{1.3}\end{proposition}

\begin{lemma}{\rm \cite{fremck}} If $B$ is a subalgebra of $A$, then, for any $\alpha ,\beta \in {\rm Con}(A)$, $[\alpha \cap B^2,\beta \cap B^2]_B\subseteq [\alpha ,\beta ]_A\cap B^2$.\label{1.8}\end{lemma}

\begin{proposition}{\rm \cite[Theorem 5.17, p. 48]{owe}} Let $n\in \N ^*$, $M_1,\ldots ,M_n$ be algebras from ${\cal C}$, $\displaystyle M=\prod _{i=1}^nM_i$ and, for all $i\in \overline{1,n}$, $\alpha _i,\beta _i\in {\rm Con}(M_i)$. Then: $\displaystyle [\prod _{i=1}^n\alpha _i,\prod _{i=1}^n\beta _i]_M=\prod _{i=1}^n[\alpha _i,\beta _i]_{M_i}$.\label{comutprod}\end{proposition}

\begin{definition}{\rm \cite{fremck}} A proper congruence $\phi $ of $A$ is said to be {\em prime} iff, for all $\alpha ,\beta \in {\rm Con}(A)$, $[\alpha ,\beta ]_A\subseteq \phi $ implies $\alpha \subseteq \phi $ or $\beta \subseteq \phi $.\label{1.4}\end{definition}

The set of the prime congruences of $A$ shall be denoted by ${\rm Spec}(A)$. ${\rm Spec}(A)$ is called the {\em (prime) spectrum} of $A$. Note that not every algebra in a congruence--modular equational class has prime congruences. We shall denote by ${\rm Min}(A)$ the set of the {\em minimal prime congruences} of $A$, that is the minimal elements of the poset $({\rm Spec}(A),\subseteq )$.

\begin{lemma}{\rm \cite[Theorem $5.3$]{agl}} If $\nabla _A$ is finitely generated, then:\begin{itemize}
\item any proper congruence of $A$ is included in a maximal congruence of $A$;
\item any maximal congruence of $A$ is prime.\end{itemize}\label{folclor}\end{lemma}

Following \cite{koll}, we say that ${\cal C}$ is {\em semi--degenerate} iff no non--trivial algebra in ${\cal C}$ has one--element subalgebras. For instance, obviously, the class of bounded lattices, that of residuated lattices and that of unitary rings are semi--degenerate.

\begin{proposition}{\rm \cite{koll}} The following are equivalent:\begin{enumerate}
\item\label{2.6(1)} ${\cal C}$ is semi--degenerate;
\item\label{2.6(2)} for all members $M$ of ${\cal C}$, the congruence $\nabla _M$ is finitely generated.\end{enumerate}\label{2.6}\end{proposition}

\begin{proposition}{\rm \cite[Theorem 8.5, p. 85]{fremck}} The following are equivalent:\begin{itemize}
\item for any algebra $M$ from ${\cal C}$, $[\nabla _M,\nabla _M]=\nabla _M$;
\item for any algebra $M$ from ${\cal C}$ and any $\theta \in {\rm Con}(M)$, $[\theta ,\nabla _M]=\theta $;
\item for any $n\in \N ^*$ and any algebras $M_1,\ldots ,M_n$ from ${\cal C}$, $\displaystyle {\rm Con}(\prod _{i=1}^nM_i)=\{\prod _{i=1}^n\theta _i\ |\ (\forall \, i\in \overline{1,n})\, (\theta _i\in {\rm Con}(M_i))$.\end{itemize}\label{prodcongr}\end{proposition}

\begin{lemma}\begin{enumerate}
\item\label{distribsemid1} If ${\cal C}$ is semi--degenerate, then ${\cal C}$ fulfills the equivalent conditions from Proposition \ref{prodcongr}.
\item\label{distribsemid2} If ${\cal C}$ is congruence--distributive, then ${\cal C}$ fulfills the equivalent conditions from Proposition \ref{prodcongr}.\end{enumerate}\label{distribsemid}\end{lemma}

\begin{proof} (\ref{distribsemid1}) This is exactly \cite[Lemma 5.2]{agl}.

\noindent (\ref{distribsemid2}) Clear, from Theorem \ref{distrib}.\end{proof}

\begin{definition}{\rm \cite{agl}} A non--empty subset $S\subseteq A^2$ is called an {\em m--system} iff, for any $(a,b),(c,d)\in S$, we have $[Cg_A(a,b),Cg_A(c,d)]_A\cap S\neq \emptyset $.\label{msist}\end{definition}

\begin{lemma} For any $\phi \in {\rm Spec}(A)$, $\nabla _A\setminus \phi $ is an m--system.\label{1.6}\end{lemma}

\begin{proof} Let $\phi \in {\rm Spec}(A)$ and $S=\nabla _A\setminus \phi $, so that $S\neq \emptyset $ since $\phi $ is a proper congruence of $A$. Let $(a,b),(c,d)\in S$, so that $(a,b),(c,d)\notin \phi $, thus $Cg_A(a,b)\nsubseteq \phi $ and $Cg_A(c,d)\nsubseteq \phi $, hence $[Cg_A(a,b),Cg_A(c,d)]_A\nsubseteq \phi $ since $\phi $ is a prime congruence. Thus $[Cg_A(a,b),Cg_A(c,d)]_A\cap S=[Cg_A(a,b),Cg_A(c,d)]_A\cap (\nabla _A\setminus \phi )\neq \emptyset $. Therefore $S$ is an m--system.\end{proof}

\begin{lemma}{\rm \cite{agl}} Assume that $\nabla _A$ is finitely generated. Let $\alpha \in {\rm Con}(A)$ and $S\subseteq A^2$ an m--system. If $\phi $ is a maximal element of the set $\{\theta \in {\rm Con}(A)\ |\ \alpha \subseteq \theta ,\theta \cap S=\emptyset \}$, then $\phi \in {\rm Spec}(A)$.\label{1.7}\end{lemma}

For all $\theta \in {\rm Con}(A)$, we shall denote by $V_A(\theta )=\{\phi \in {\rm Spec}(A)\ |\ \theta \subseteq \phi \}$ and $D_A(\theta )={\rm Spec}(A)\setminus V_A(\theta )$. For all $a,b\in A$, we denote $V_A(a,b)=V_A(Cg_A(a,b))$ and $D_A(a,b)=D_A(Cg_A(a,b))$. It is well known that, if $\nabla _A$ is finitely generated, then $\{D_A(\theta )\ |\ \theta \in {\rm Con}(A)\}$ is a topology on ${\rm Spec}(A)$, called {\em the Stone topology}, having $\{D_A(a,b)\ |\ a,b\in A\}$ as a basis, $\{V_A(\theta )\ |\ \theta \in {\rm Con}(A)\}$ as the set of closed sets and $\{V_A(a,b)\ |\ a,b\in A\}$ as a basis of closed sets. For any $M\subseteq {\rm Spec}(A)$, we shall denote by $\overline{M}$ the closure of $M$ in the topological space ${\rm Spec}(A)$ with the Stone topology. Clearly, for all $\phi \in {\rm Spec}(A)$, $\overline{\{\phi \}}=V_A(\phi )$.

\section{Admissible Morphisms}
\label{admmorph}

In this section, we study {\em admissible morphisms}, that is those morphisms $f$ with the property that $f^*$ takes prime congruences to prime congruences. We recall some of their properties from \cite{euadm}, among which: surjectivity implies admissibility, but the converse is not true, nor does admissibility always hold. Throughout this section, $A,B,C$ shall be algebras from ${\cal C}$ and $f:A\rightarrow B$, $g:B\rightarrow C$ shall be morphisms in ${\cal C}$.

\begin{remark} For any subalgebra $S$ of $A$, if $i:S\rightarrow A$ is the canonical embedding and $\alpha \in {\rm Con}(A)$, then $i^*(\alpha )=\alpha \cap S^2\in {\rm Con}(S)$.\label{rex1}\end{remark}

\begin{remark}{\rm \cite[Lemma $6$, p. $19$, and Lemma $7$, p. $20$]{gratzer}} Any class of a congruence of a lattice $L$ is a convex sublattice of $L$, thus it has a unique writing as an intersection between a filter and an ideal of $L$.\label{convex}\end{remark}

\begin{remark} For any $\beta \in {\rm Con}(B)$, we have $\Delta _B\subseteq \beta $, thus ${\rm Ker}(f)=f^*(\Delta _B)\subseteq f^*(\beta )$.\label{ker}\end{remark}

\begin{remark} If $\phi ,\theta \in {\rm Con}(A)$ and $\psi \in {\rm Con}(A/\theta )$ such that $p_{\theta }^*(\psi )=\phi $, then, by Remark \ref{ker}, $\theta ={\rm Ker}(p_{\theta })\subseteq \phi $.\label{surjcan}\end{remark}

\begin{remark} ${\rm Ker}(f)\subseteq {\rm Ker}(g\circ f)$, because ${\rm Ker}(f)=f^*(\Delta _B)\subseteq f^*(g^*(\Delta _C))=(g\circ f)^*(\Delta _C)={\rm Ker}(g\circ f)$.\label{kerker}\end{remark}

Following \cite{euadm}, for any algebra M, we shall denote by ${\rm Con}_2(M)$ the set of the two--class congruences of $M$: ${\rm Con}_2(M)=\{\theta \in {\rm Con}(M)\ |\ |M/\theta |=2\}$.

\begin{remark}{\rm \cite{euadm}} ${\rm Con}_2(A)\subseteq {\rm Max}(A)$.\label{2clsmax}\end{remark}

\begin{lemma}{\rm \cite{euadm}}\begin{enumerate}
\item\label{2cls0} If ${\cal C}$ is semi--degenerate, then $f^*(\{\nabla _A\})=\{\nabla _B\}$.

\item\label{2cls1} If $f^*(\{\nabla _A\})=\{\nabla _B\}$, then $f^*({\rm Con}_2(B))\subseteq {\rm Con}_2(A)$.
\item\label{2cls2} If ${\cal C}$ is semi--degenerate, then $f^*({\rm Con}_2(B))\subseteq {\rm Con}_2(A)$.\end{enumerate}\label{2cls}\end{lemma}

Throughout the rest of this section, ${\cal C}$ shall be congruence--modular.

\begin{remark}{\rm \cite{euadm}} Clearly, if ${\cal C}$ is congruence--distributive (see Theorem \ref{distrib}) or the commutator in $A$ equals the intersection of congruences, then the prime congruences of $A$ are exactly the prime elements of the lattice ${\rm Con}(A)$. If, additionally, ${\rm Con}(A)$ is finite, then the prime congruences of $A$ are exactly the elements of ${\rm Con}(A)$ that have exactly one successor in the lattice ${\rm Con}(A)$. Thus, if, moreover, ${\rm Con}(A)$ is a non--trivial finite chain, then ${\rm Spec}(A)={\rm Con}(A)\setminus \{\nabla _A\}$.\label{rex2}\end{remark}

Remarks \ref{rex1}, \ref{convex} and \ref{rex2} are easy to use for determining all congruences and the prime congruences in the examples we shall give. Sometimes, we shall use the remarks in this paper without referencing them.

\begin{lemma}{\rm \cite{euadm}}\begin{enumerate}
\item\label{specmax1} If ${\cal C}$ is semi--degenerate, then ${\rm Max}(A)\subseteq {\rm Spec}(A)$. 
\item\label{specmax0} If $A$ is congruence--distributive and the commutator in $A$ equals the intersection, then ${\rm Max}(A)\subseteq {\rm Spec}(A)$.
\item\label{specmax2} If ${\cal C}$ is congruence--distributive, then ${\rm Max}(A)\subseteq {\rm Spec}(A)$.
\item\label{specmax3} If the commutator in $A$ equals the intersection of congruences and ${\rm Con}(A)$ is a Boolean algebra, then ${\rm Spec}(A)={\rm Max}(A)$.
\item\label{specmax4} If ${\cal C}$ is congruence--distributive and ${\rm Con}(A)$ is a Boolean algebra, then ${\rm Spec}(A)={\rm Max}(A)$.
\end{enumerate}\label{specmax}\end{lemma}

\begin{remark} By Remark \ref{2clsmax} and Lemmas \ref{folclor} and \ref{specmax}, if $\nabla _A$ is finitely generated or $A$ is congruence--distributive and the commutator in $A$ equals the intersection, or ${\cal C}$ is semi--degenerate or congruence--distributive, then ${\rm Con}_2(A)\subseteq {\rm Max}(A)\subseteq {\rm Spec}(A)$.\label{2clsspecmax}\end{remark}

\begin{remark} If we consider the conditions:

$(c_1)\ $ ${\rm Spec}(A)={\rm Max}(A)$, $\quad (c_2)\ $ $({\rm Spec}(A),\subseteq )$ is unorderred, $\quad $ then: $(c_1)\Rightarrow (c_2)$ and, if ${\rm Max}(A)\subseteq {\rm Spec}(A)$, for instance if $\nabla _A$ is finitely generated or $A$ is congruence--distributive and the commutator in $A$ equals the intersection, or ${\cal C}$ is semi--degenerate or congruence--distributive (see Lemmas \ref{folclor} and \ref{specmax}), then $(c_1)\Leftrightarrow (c_2)$.\label{unord}\end{remark}

\begin{definition} We say that the morphism $f:A\rightarrow B$ is {\em admissible} iff, for all $\psi \in {\rm Spec}(B)$, we have $f^*(\psi )\in {\rm Spec}(A)$.\label{morfadm}\end{definition}

So $f$ is admissible iff $f^*({\rm Spec}(B))\subseteq {\rm Spec}(A)$, where by $f^*$ we denote the direct image of the function $f^*:{\rm Con}(B)\rightarrow {\rm Con}(A)$. If $f$ is admissible, then we may consider the restriction $f^*\mid _{{\rm Spec}(B)}:{\rm Spec}(B)\rightarrow {\rm Spec}(A)$.

\begin{remark} $f$ is admissible iff $f^*({\rm Spec}(B))\subseteq V_A({\rm Ker}(f))$. Indeed, $f^*({\rm Spec}(B))\subseteq f^*({\rm Con}(B))\subseteq [{\rm Ker}(f))$ by Remark \ref{ker}, and, if $f$ is admissible, then $f^*({\rm Spec}(B))\subseteq {\rm Spec}(A)$ by the above, hence $f^*({\rm Spec}(B))\subseteq {\rm Spec}(A)\cap [{\rm Ker}(f))=V_A({\rm Ker}(f))$. The converse implication is trivial.

If $f$ is admissible, then, for all $\theta \in {\rm Con}(B)$, $f^*(V_B(\theta ))\subseteq V_A(f^*(\theta ))$, because: $f^*(V_B(\theta ))\subseteq f^*({\rm Spec}(B))\subseteq {\rm Spec}(A)$ by the admissibility of $f$, and $f^*(V_B(\theta ))\subseteq f^*([\theta ))$, hence $f^*(V_B(\theta ))\subseteq [f^*(\theta ))\cap {\rm Spec}(A)=V_A(f^*(\theta ))$.\label{admv}\end{remark}

\begin{proposition}{\rm \cite{euadm}}\begin{enumerate}
\item\label{bdlat1} Any morphism in the class of bounded distributive lattices is admissible.
\item\label{bdlat2} Moreover: any bounded lattice morphism whose co--domain is distributive is admissible.\end{enumerate}\label{bdlat}\end{proposition}

\begin{example} Let ${\cal L}_2^2$, ${\cal D}$ and ${\cal P}$ have the elements denoted as in the following Hasse diagrams, $i:{\cal L}_2^2\rightarrow {\cal D}$ and $j:{\cal L}_2^2\rightarrow {\cal P}$ be the canonical bounded lattice embeddings and $h:{\cal P}\rightarrow {\cal D}$ and $k:{\cal P}\rightarrow {\cal D}$ be the bounded lattice morphisms given by the following table:\vspace*{-15pt}

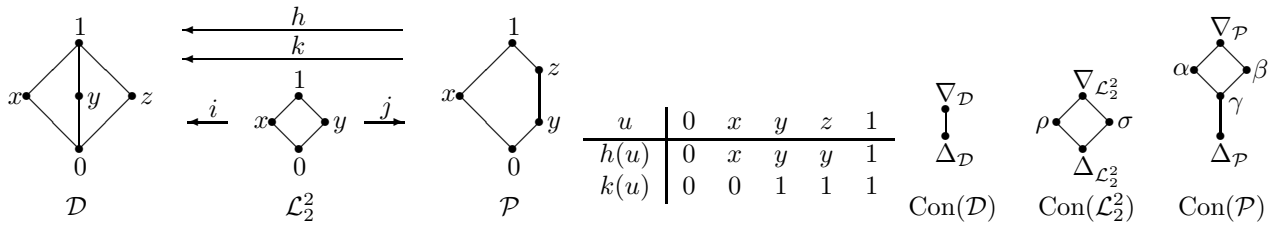
\begin{center}\begin{tabular}{ccccccc}
\begin{picture}(40,80)(0,0)

\put(30,25){\line(0,1){40}}
\put(30,25){\line(1,1){20}}
\put(30,25){\line(-1,1){20}}
\put(30,65){\line(1,-1){20}}
\put(30,65){\line(-1,-1){20}}
\put(30,25){\circle*{3}}
\put(10,45){\circle*{3}}
\put(3,42){$x$}
\put(30,45){\circle*{3}}
\put(33,42){$y$}
\put(50,45){\circle*{3}}
\put(53,42){$z$}
\put(30,65){\circle*{3}}
\put(28,15){$0$}
\put(28,68){$1$}
\put(25,0){${\cal D}$}
\end{picture}
&
\begin{picture}(100,80)(0,0)
\put(100,70){\vector(-1,0){83}}
\put(58,72){$h$}
\put(100,59){\vector(-1,0){83}}
\put(58,61){$k$}
\put(61,25){\line(1,1){10}}
\put(61,25){\line(-1,1){10}}
\put(61,45){\line(1,-1){10}}
\put(61,45){\line(-1,-1){10}}
\put(61,25){\circle*{3}}
\put(86,35){\vector(1,0){15}}
\put(34,35){\vector(-1,0){15}}
\put(27,37){$i$}
\put(92,38){$j$}
\put(44,33){$x$}
\put(74,33){$y$}
\put(51,35){\circle*{3}}
\put(71,35){\circle*{3}}

\put(61,45){\circle*{3}}
\put(59,15){$0$}
\put(59,48){$1$}
\put(56,0){${\cal L}_2^2$}
\end{picture}
&
\begin{picture}(40,80)(0,0)
\put(30,25){\line(1,1){10}}
\put(30,25){\line(-1,1){20}}
\put(30,65){\line(1,-1){10}}
\put(30,65){\line(-1,-1){20}}
\put(40,35){\line(0,1){20}}
\put(30,25){\circle*{3}}
\put(10,45){\circle*{3}}
\put(3,43){$x$}
\put(40,35){\circle*{3}}
\put(43,33){$y$}
\put(30,65){\circle*{3}}
\put(40,55){\circle*{3}}
\put(43,53){$z$}
\put(25,0){${\cal P}$}
\put(28,15){$0$}
\put(28,68){$1$}
\end{picture}
&\hspace*{1.5pt}
\begin{picture}(100,80)(0,0)
\put(0,20){\begin{tabular}{c|ccccc}
$u$ & $0$ & $x$ & $y$ & $z$ & $1$\\ \hline 
$h(u)$ & $0$ & $x$ & $y$ & $y$ & $1$\\ 
$k(u)$ & $0$ & $0$ & $1$ & $1$ & $1$\end{tabular}}

\end{picture}
&\hspace*{1.5pt}

\begin{picture}(40,80)(0,0)
\put(20,30){\line(0,1){10}}
\put(20,30){\circle*{3}}
\put(20,40){\circle*{3}}
\put(16,20){$\Delta _{\cal D}$}
\put(16,43){$\nabla _{\cal D}$}
\put(6,0){${\rm Con}({\cal D})$}
\end{picture}
&
\begin{picture}(40,80)(0,0)
\put(20,25){\line(1,1){10}}
\put(20,25){\line(-1,1){10}}
\put(20,45){\line(1,-1){10}}
\put(20,45){\line(-1,-1){10}}
\put(20,25){\circle*{3}}
\put(2,33){$\rho $}
\put(33,33){$\sigma $}
\put(10,35){\circle*{3}}
\put(30,35){\circle*{3}}
\put(20,45){\circle*{3}}
\put(16,15){$\Delta _{{\cal L}_2^2}$}
\put(16,48){$\nabla _{{\cal L}_2^2}$}
\put(3,0){${\rm Con}({\cal L}_2^2)$}
\end{picture}
&

\begin{picture}(40,80)(0,0)
\put(20,45){\line(1,1){10}}
\put(20,45){\line(-1,1){10}}
\put(20,65){\line(1,-1){10}}
\put(20,65){\line(-1,-1){10}}
\put(20,45){\line(0,-1){15}}
\put(16,20){$\Delta _{\cal P}$}
\put(16,68){$\nabla _{\cal P}$}
\put(20,30){\circle*{3}}
\put(23,40){$\gamma $}
\put(2,53){$\alpha $}
\put(32,52){$\beta $}
\put(20,45){\circle*{3}}
\put(10,55){\circle*{3}}
\put(30,55){\circle*{3}}
\put(20,65){\circle*{3}}
\put(4,0){${\rm Con}({\cal P})$}\end{picture}\end{tabular}\end{center}\vspace*{-4pt}

${\rm Con}({\cal D})=\{\Delta _{\cal D},\nabla _{\cal D}\}\cong {\cal L}_2$, so ${\rm Spec}({\cal D})=\{\Delta _{\cal D}\}$.

${\rm Con}({\cal L}_2^2)=\{\Delta _{{\cal L}_2^2},\rho ,\sigma ,\nabla _{{\cal L}_2^2}\}\cong {\cal L}_2^2$, as is the case for any finite Boolean algebra, where $\rho =eq(\{\{0,x\},\{y,1\}\})$ and $\sigma =eq(\{\{0,y\},\{x,1\}\})$, so ${\rm Spec}({\cal L}_2^2)=\{\rho ,\sigma \}$.

${\rm Con}({\cal P})=\{\Delta _{\cal P},\alpha ,\beta ,\gamma ,\nabla _{\cal P}\}$, with the Hasse diagram above, where $\alpha =eq(\{0,y,z\},\{x,1\})$, $\beta =eq(\{0,x\},$\linebreak $\{y,z,1\})$ and $\gamma =eq(\{0\},\{x\},\{y,z\},\{1\})$, thus ${\rm Spec}({\cal P})=\{\Delta _{\cal P},\alpha ,\beta \}$. $\gamma \notin {\rm Spec}({\cal P})$, because $\gamma =\alpha \cap \beta =[\alpha ,\beta ]_{\cal P}\supseteq [\alpha ,\beta ]_{\cal P}$, but $\alpha \nsubseteq \gamma $ and $\beta \nsubseteq \gamma $.

$\Delta _{\cal D}\in {\rm Spec}({\cal D})$ and $i^*(\Delta _{\cal D})=\Delta _{{\cal L}_2^2}\notin {\rm Spec}({\cal L}_2^2)$, thus $i$ is not admissible.

$\Delta _{\cal P}\in {\rm Spec}({\cal P})$ and $j^*(\Delta _{\cal P})=\Delta _{{\cal L}_2^2}\notin {\rm Spec}({\cal L}_2^2)$, thus $j$ is not admissible.

$\Delta _{\cal D}\in {\rm Spec}({\cal D})$ and $h^*(\Delta _{\cal D})=\gamma \notin {\rm Spec}({\cal P})$, therefore $h$ is not admissible.

For all $\theta \in {\rm Con}({\cal P})\setminus \{\Delta _{\cal P}\}$, $k^*(\theta )=\Delta _{\cal D}$, thus $ k^*(\Delta _{\cal P})=k^*(\alpha )=k^*(\beta )=\Delta _{\cal D}\in {\rm Spec}({\cal D})$, hence $k$ is admissible.\label{exadm}\end{example}

\begin{lemma} Any surjective morphism in ${\cal C}$ is admissible, but the converse is not true, not even when ${\cal C}$ is semi--degenerate and congruence--distributive.\label{surjadm}\end{lemma}

\begin{proof} \cite[Proposition $2.1$, (i)]{agl} shows that every surjective morphism is admissible. Proposition \ref{bdlat} provides us with an infinity of examples of admissible morphisms that are not surjective. Also, the admissible morphism $k$ from Example \ref{exadm} is not surjective. These are examples of morphisms in the semi--degenerate congruence--distributive class of bounded lattices.\end{proof}

\begin{proposition} Not all morphisms are admissible, not even when ${\cal C}$ is semi--degenerate and congruence--distributive.\label{notalladm}\end{proposition}

\begin{proof} In Example \ref{exadm}, the morphisms $i$, $j$ and $h$ in the semi--degenerate congruence--distributive class of bounded lattices are not admissible.\end{proof}

\begin{remark} For any $\theta \in {\rm Con}(A)$, the morphism $p_{\theta }:A\rightarrow A/\theta $ is surjective, and thus admissible by Lemma \ref{surjadm}.\label{surjcanadm}\end{remark}

\begin{remark} Any composition of admissible morphisms is admissible. Indeed, assume that $f$ and $g$ are admissible; then, for all $\chi \in {\rm Spec}(C)$, it follows that $g^*(\chi )\in {\rm Spec}(B)$, hence $(g\circ f)^*(\chi )=(f^*\circ g^*)(\chi )=f^*(g^*(\chi ))\in {\rm Spec}(A)$, therefore $g\circ f$ is admissible.\label{compadm}\end{remark}

\begin{remark} Clearly, if $g$ is an isomorphism, then: $f$ is admissible iff $g\circ f$ is admissible. Similarly, if $f$ is an isomorphism, then: $g$ is admissible iff $g\circ f$ is admissible.\label{admcompizom}\end{remark}

\begin{lemma}{\rm \cite{euadm}}\begin{enumerate}
\item\label{2clsadm1} If ${\cal C}$ is semi--degenerate and ${\rm Spec}(B)={\rm Con}_2(B)$, then $f$ is admissible.
\item\label{2clsadm2} If $\nabla _A$ is finitely generated, $f^*(\{\nabla _A\})=\{\nabla _B\}$ and ${\rm Spec}(B)={\rm Con}_2(B)$, then $f$ is admissible. 
\item\label{2clsadm3} If $A$ is congruence--distributive and the commutator in $A$ equals the intersection, $f^*(\{\nabla _A\})=\{\nabla _B\}$ and ${\rm Spec}(B)={\rm Con}_2(B)$, then $f$ is admissible.
\item\label{2clsadm4} If ${\cal C}$ is congruence--distributive, $f^*(\{\nabla _A\})=\{\nabla _B\}$ and ${\rm Spec}(B)={\rm Con}_2(B)$, then $f$ is admissible.\end{enumerate}\label{2clsadm}\end{lemma}

\begin{remark}{\rm \cite{euadm}}\begin{itemize}
\item If $L$ and $M$ are bounded lattices and $h:L\rightarrow M$ is a bounded lattice morphism with $h(L)=\{0,1\}$, then $h^*({\rm Con}(M)\setminus \{\nabla _M\})\subseteq {\rm Con}_2(L)$, thus $h$ is admissible by Lemma \ref{2cls}, (\ref{2cls2}). Notice, also, that, if $M$ is non--trivial, then $h^*({\rm Con}(M)\setminus \{\nabla _M\})=h^*({\rm Spec}(M))=h^*(\{\Delta _M\})=\{{\rm Ker}(h)\}$; in particular, ${\rm Ker}(h)\in {\rm Con}_2(L)\subseteq {\rm Spec}(L)$. This is the case for the bounded lattice morphism $k$ from Example \ref{exadm}.
\item The above actually holds for any equational class of bounded orderred structures.\end{itemize}\label{2cls01}\end{remark}

\begin{lemma}{\rm \cite{euadm}} Let $L$ be a bounded lattice.\begin{enumerate}

\item\label{speclat1} If $L$ is distributive, then ${\rm Spec}(L)={\rm Max}(L)={\rm Con}_2(L)$.
\item\label{speclat2} If $L$ can be obtained through finite direct products and/or finite ordinal sums from bounded distributive lattices and/or finite modular lattices and/or relatively complemented bounded lattices with ACC, then ${\rm Spec}(L)={\rm Max}(L)$.\end{enumerate}\label{speclat}\end{lemma}

Note that Proposition \ref{bdlat} above follows from Lemma \ref{2clsadm}, (\ref{2clsadm1}), and Lemma \ref{speclat}, (\ref{speclat1}).

\begin{lemma} Assume that ${\rm Spec}(A)={\rm Con}(A)\setminus \{\nabla _A\}$.

\begin{enumerate}
\item\label{specall1} If $(f^*)^{-1}(\{\nabla _A\})=\{\nabla _B\}$, then $f$ is admissible. 
\item\label{specall2} If ${\cal C}$ is semi--degenerate, then $f$ is admissible.\end{enumerate}\label{specall}\end{lemma}

\begin{proof} (\ref{specall1}) $(f^*)^{-1}(\{\nabla _A\})=\{\nabla _B\}$ implies $f^*({\rm Con}(B)\setminus \{\nabla _B\})\subseteq {\rm Con}(A)\setminus \{\nabla _A\}$, hence $f^*({\rm Spec}(B))\subseteq f^*({\rm Con}(B)\setminus \{\nabla _B\})\subseteq {\rm Con}(A)\setminus \{\nabla _A\}={\rm Spec}(A)$, so $f$ is admissible.

\noindent (\ref{specall2}) By (\ref{specall1}) and Lemma \ref{2cls}, (\ref{2cls0}).\end{proof}

\begin{lemma} Assume that ${\cal C}$ is congruence--distributive or the commutator in $A$ equals the intersection. If ${\rm Con}(A)$ is a non--trivial finite chain, then:\begin{enumerate}
\item\label{conlantfin1} If $(f^*)^{-1}(\{\nabla _A\})=\{\nabla _B\}$, then $f$ is admissible. 
\item\label{conlantfin2} If ${\cal C}$ is semi--degenerate, then $f$ is admissible.\end{enumerate}\label{conlantfin}\end{lemma}

\begin{proof} By Lemma \ref{specall} and Remark \ref{rex2}.\end{proof}

\section{Properties Going Up and Lying Over}
\label{guandlo}

In this section we define conditions Going Up and Lying Over on admissible morphisms and start investigating their properties. We prove that they are non--trivial, that their study can be reduced to embeddings, that surjectivity implies Going Up and Going Up implies Lying Over, but the converses of these implications do not hold. Throughout this section, ${\cal C}$ shall be congruence--modular, $A,B,C$ shall be members of ${\cal C}$ and $f:A\rightarrow B$, $g:B\rightarrow C$ shall be admissible morphisms in ${\cal C}$. Then $g\circ f$ is admissible by Remark \ref{compadm}. Also, $M,N$ shall be members of ${\cal C}$ and $h:M\rightarrow N$ shall be a morphism in ${\cal C}$, not necessarily admissible.

\begin{definition} We say that $f$ fulfills property {\em Going Up} (abbreviated {\em GU}) iff, for any $\phi ,\psi \in {\rm Spec}(A)$ and any $\phi _1\in {\rm Spec}(B)$ such that $\phi \subseteq \psi $ and $f^*(\phi _1)=\phi $, there exists a $\psi _1\in {\rm Spec}(B)$ such that $\phi _1\subseteq \psi _1$ and $f^*(\psi _1)=\psi $.

We say that $f$ fulfills property {\em Lying Over} (abbreviated {\em LO}) iff, for any $\phi \in {\rm Spec}(A)$ such that ${\rm Ker}(f)\subseteq \phi $, there exists a $\phi _1\in {\rm Spec}(B)$ such that $f^*(\phi _1)=\phi $.\label{3.1}\end{definition}

\begin{remark} Trivially, any isomorphism is admissible (see also Lemma \ref{surjadm}) and fulfills GU and LO.

Clearly, if $g$ is an isomorphism, then: $f$ fulfills GU, respectively LO, iff $g\circ f$ fulfills GU, respectively LO. Similarly, if $f$ is an isomorphism, then: $g$ fulfills GU, respectively LO, iff $g\circ f$ fulfills GU, respectively LO.

Hence, if all canonical embeddings of $A$ into other algebras from ${\cal C}$ are admissible and fulfill GU, respectively LO, then all embeddings of $A$ into other algebras from ${\cal C}$ are admissible and fulfill GU, respectively LO.\label{compizom}\end{remark}

\begin{remark} If $A$ is a subalgebra of $B$ and $i:A\rightarrow B$ is the canonical embedding, then, for all $\beta \in {\rm Con}(B)$, $i^*(\beta )=\beta \cap A^2$, thus ${\rm Ker}(i)=i^*(\Delta _B)=\Delta _B\cap A^2=\Delta _A$, therefore, if $i$ is admissible:\begin{itemize}
\item $i$ fulfills GU iff, for any $\phi ,\psi \in {\rm Spec}(A)$ and any $\phi _1\in {\rm Spec}(B)$ such that $\phi \subseteq \psi $ and $\phi _1\cap A^2=\phi $, there exists a $\psi _1\in {\rm Spec}(B)$ such that $\phi _1\subseteq \psi _1$ and $\psi _1\cap A^2=\psi $;
\item $i$ fulfills LO iff, for any $\phi \in {\rm Spec}(A)$, there exists a $\phi _1\in {\rm Spec}(B)$ such that $\phi _1\cap A^2=\phi $.\end{itemize}\label{3.2}\end{remark}

\begin{lemma}\begin{enumerate}
\item\label{cargulo1} $f$ fulfills GU iff, for all $\psi \in {\rm Spec}(B)$, $V_A(f^*(\psi ))\subseteq f^*(V_B(\psi ))$.
\item\label{cargulo2} $f$ fulfills LO iff $V_A({\rm Ker}(f))\subseteq f^*({\rm Spec}(B))$ iff $V_A({\rm Ker}(f))=f^*({\rm Spec}(B))$.
\item\label{cargulo3} $h$ is admissible and fulfills LO iff $V_M({\rm Ker}(h))=h^*({\rm Spec}(N))$.\end{enumerate}\label{cargulo}\end{lemma}

\begin{proof} (\ref{cargulo1}) By the definition of GU.

\noindent (\ref{cargulo2}) Since $f$ is admissible, these equivalences follow from the definition of LO and Remark \ref{admv}.

\noindent (\ref{cargulo3}) By (\ref{cargulo2}) and the definition of admissibility.\end{proof}

\begin{proposition}\begin{enumerate}
\item\label{gulogu1} If $g\circ f$ fulfills GU, $g$ fulfills LO and ${\rm Spec}(B)\subset [{\rm Ker}(g))$, then $f$ fulfills GU.
\item\label{gulogu2} If $g\circ f$ fulfills GU and $g$ is injective and fulfills LO, then $f$ fulfills GU.\end{enumerate}\label{gulogu}\end{proposition}

\begin{proof} (\ref{gulogu1}) Let $\phi ,\psi \in {\rm Spec}(A)$ and $\phi _1\in {\rm Spec}(B)$ such that $\phi \subseteq \psi $ and $f^*(\phi _1)=\phi $. Then $\phi _1\supseteq {\rm Ker}(g)$, so, since $g$ fulfills LO, it follows that there exists a $\phi _2\in {\rm Spec}(C)$ such that $g^*(\phi _2)=\phi _1$, hence $(g\circ f)^*(\phi _2)=f^*(g^*(\phi _2))=f^*(\phi _1)=\phi $. $g\circ f$ fulfills GU, hence there exists a $\psi _2\in {\rm Spec}(C)$ such that $\phi _2\subseteq \psi _2$ and $\psi =(g\circ f)^*(\psi _2)=f^*(g^*(\psi _2))$. If we denote by $\psi _1=g^*(\psi _2)$, then $\psi _1\in {\rm Spec}(B)$, $f^*(\psi _1)=\psi $ and $\psi _1=g^*(\psi _2)\supseteq g^*(\phi _2)=\phi _1$. Hence $f$ fulfills GU.

\noindent (\ref{gulogu2}) By (\ref{gulogu1}) and the fact that, if $g$ is injective, then ${\rm Ker}(g)=g^*(\Delta _C)=\Delta _B$.\end{proof}

\begin{proposition}\begin{enumerate}
\item\label{carloinj1} If ${\rm Spec}(A)\subset [{\rm Ker}(f))$, then: $f$ fulfills LO iff the map $f^*\mid _{{\rm Spec}(B)}:{\rm Spec}(B)\rightarrow {\rm Spec}(A)$ is surjective.
\item\label{carloinj2} If $f$ is injective, then: $f$ fulfills LO iff the map $f^*\mid _{{\rm Spec}(B)}:{\rm Spec}(B)\rightarrow {\rm Spec}(A)$ is surjective.\end{enumerate}\label{carloinj}\end{proposition}

\begin{proof} (\ref{carloinj1}) By Lemma \ref{cargulo}, (\ref{cargulo2}).

\noindent (\ref{carloinj2}) By (\ref{carloinj1}) and the fact that, if $f$ is injective, then ${\rm Ker}(f)=f^*(\Delta _B)=\Delta _A$.\end{proof}

\begin{example} The non--surjective bounded lattice morphism $k$ in Example \ref{exadm} is admissible and fulfills GU and LO (in a trivial way, because ${\rm Spec}({\cal D})=\{\Delta _{\cal D}\}$).

Here is an admissible bounded lattice morphism which does not fulfill GU, nor does it fulfill LO: let $H$ and $K$ be the bounded lattices given by the following Hasse diagrams, with $H$ a bounded sublattice of $K$, and $i:H\rightarrow K$ be the canonical bounded lattice embedding:\vspace*{-12pt}

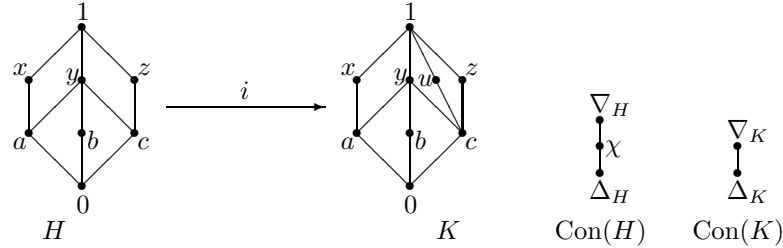
\begin{center}\begin{tabular}{cccccc}
\begin{picture}(80,100)(0,0)
\put(40,20){\circle*{3}}
\put(38,10){$0$}
\put(40,20){\line(1,1){20}}
\put(40,20){\line(-1,1){20}}
\put(20,40){\circle*{3}}
\put(60,40){\circle*{3}}
\put(20,60){\circle*{3}}
\put(40,60){\circle*{3}}
\put(60,60){\circle*{3}}
\put(40,80){\circle*{3}}
\put(40,40){\circle*{3}}
\put(20,40){\line(0,1){20}}
\put(20,40){\line(1,1){20}}

\put(60,40){\line(-1,1){20}}

\put(60,40){\line(0,1){20}}
\put(40,80){\line(0,-1){60}}
\put(40,80){\line(-1,-1){20}}
\put(40,80){\line(1,-1){20}}
\put(14,34){$a$}
\put(42,34){$b$}
\put(61,34){$c$}
\put(14,62){$x$}
\put(34,60){$y$}
\put(61,62){$z$}
\put(38,83){$1$}
\put(25,0){$H$}
\end{picture}
&
\begin{picture}(20,100)(0,0)
\put(8,53){$i$}
\put(-20,50){\vector(1,0){60}}
\end{picture}
&
\begin{picture}(80,100)(0,0)
\put(40,20){\circle*{3}}
\put(38,10){$0$}
\put(40,20){\line(1,1){20}}
\put(40,20){\line(-1,1){20}}
\put(20,40){\circle*{3}}
\put(60,40){\circle*{3}}
\put(20,60){\circle*{3}}
\put(40,60){\circle*{3}}
\put(60,60){\circle*{3}}
\put(40,80){\circle*{3}}

\put(40,80){\line(1,-2){20}}
\put(50,60){\circle*{3}}
\put(40,40){\circle*{3}}
\put(20,40){\line(0,1){20}}
\put(20,40){\line(1,1){20}}
\put(60,40){\line(-1,1){20}}
\put(60,40){\line(0,1){20}}
\put(40,80){\line(0,-1){60}}
\put(40,80){\line(-1,-1){20}}
\put(40,80){\line(1,-1){20}}
\put(14,34){$a$}
\put(42,34){$b$}
\put(61,34){$c$}
\put(14,62){$x$}
\put(34,60){$y$}
\put(61,62){$z$}

\put(43,58){$u$}
\put(38,83){$1$}
\put(50,0){$K$}
\end{picture}
&
\begin{picture}(40,100)(0,0)
\put(20,25){\line(0,1){20}}
\put(20,25){\circle*{3}}
\put(20,35){\circle*{3}}
\put(20,45){\circle*{3}}
\put(16,15){$\Delta _H$}
\put(16,48){$\nabla _H$}
\put(22,33){$\chi $}
\put(3,0){${\rm Con}(H)$}
\end{picture}
&
\begin{picture}(40,100)(0,0)
\put(20,25){\line(0,1){10}}
\put(20,25){\circle*{3}}
\put(20,35){\circle*{3}}
\put(16,15){$\Delta _K$}
\put(16,38){$\nabla _K$}
\put(3,0){${\rm Con}(K)$}
\end{picture}
\end{tabular}
\end{center}\vspace*{-3pt}

Notice that ${\rm Con}(K)=\{\Delta _K,\nabla _K\}\cong {\cal L}_2$, so ${\rm Spec}(K)=\{\Delta _K\}$.

${\rm Con}(H)=\{\Delta _H,\chi ,\nabla _H\}\cong {\cal L}_3$, where $\chi =eq(\{0\},\{a,x\},\{b\},\{c,z\},\{y,1\})$, so ${\rm Spec}(H)=\{\Delta _H,\chi \}$.

$i^*(\Delta _K)=\Delta _H$, thus $i$ is admissible. ${\rm Ker}(i)=i^*(\Delta _K)=\Delta _H\subseteq \chi \in {\rm Spec}(H)$ and there exists no $\phi \in {\rm Spec}(K)=\{\Delta _K\}$ such that $i^*(\phi )=\chi $, therefore $i$ does not fulfill GU and it does not fulfill LO.\label{meproud}\end{example}

\begin{proposition} Not all admissible morphisms fulfill GU or LO, not even when ${\cal C}$ is congruence--distributive and semi--degenerate.\end{proposition}

\begin{proof} The bounded lattice embedding $i$ in Example \ref{meproud} is admissible and does not fulfill GU or LO.\end{proof}

\begin{remark} Clearly, by Remark \ref{rex2}:\begin{itemize}
\item If ${\cal C}$ is congruence--distributive and $h^*\mid _{{\rm Con}(N)}:{\rm Con}(N)\rightarrow {\rm Con}(M)$ is a bounded lattice isomorphism, then $h$ is admissible, $h^*\mid _{{\rm Spec}(N)}:{\rm Spec}(N)\rightarrow {\rm Spec}(M)$ is an order isomorphism and $h$ fulfills GU and LO.
\item If the commutator in $M$ and $N$ equals the intersection and $h^*\mid _{{\rm Con}(N)}:{\rm Con}(N)\rightarrow {\rm Con}(M)$ is a bounded lattice isomorphism, then $h$ is admissible, $h^*\mid _{{\rm Spec}(N)}:{\rm Spec}(N)\rightarrow {\rm Spec}(M)$ is an order isomorphism and $h$ fulfills GU and LO.
\end{itemize}\label{izomvsgu}\end{remark}

\begin{example} Let us see that the converses of the implications in Remark \ref{izomvsgu} do not hold, and let us see some more examples, which illustrate different cases that can appear, regarding admissibility, GU and LO. The following are examples of morphisms in the semi--degenerate congruence--distributive class of bounded lattices.

Let $H$, $K$ and the canonical embedding $i$ be as in Example \ref{meproud}. We have seen that $i$ is admissible and does not fulfill GU or LO. Note that ${\rm Con}(H)\ncong {\rm Con}(K)$, $|{\rm Spec}(H)|\neq |{\rm Spec}(K)|$ and $i^*\mid _{{\rm Con}(K)}:{\rm Con}(K)\rightarrow {\rm Con}(H)$ and $i^*\mid _{{\rm Spec}(K)}:{\rm Spec}(K)\rightarrow {\rm Spec}(H)$ are injective and they are not surjective.

Let us also consider the following bounded lattices:\vspace*{-12pt}

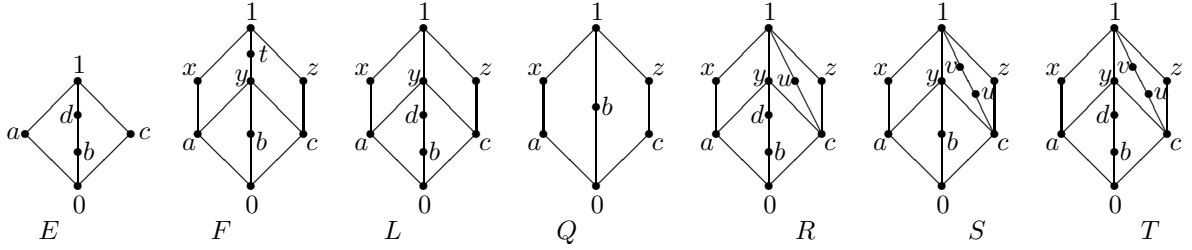
\begin{center}\begin{tabular}{ccccccc}
\begin{picture}(80,100)(0,0)
\put(40,20){\circle*{3}}
\put(38,10){$0$}
\put(40,20){\line(1,1){20}}
\put(40,20){\line(-1,1){20}}
\put(20,40){\circle*{3}}
\put(60,40){\circle*{3}}
\put(40,60){\circle*{3}}
\put(40,33){\circle*{3}}
\put(40,47){\circle*{3}}
\put(20,40){\line(1,1){20}}
\put(60,40){\line(-1,1){20}}
\put(40,60){\line(0,-1){40}}
\put(13,38){$a$}
\put(42,30){$b$}
\put(63,38){$c$}
\put(33,44){$d$}
\put(38,63){$1$}
\put(25,0){$E$}
\end{picture}
&\hspace*{-30pt}
\begin{picture}(80,100)(0,0)
\put(40,20){\circle*{3}}
\put(38,10){$0$}
\put(40,20){\line(1,1){20}}
\put(40,20){\line(-1,1){20}}
\put(20,40){\circle*{3}}
\put(60,40){\circle*{3}}
\put(20,60){\circle*{3}}
\put(40,60){\circle*{3}}
\put(40,70){\circle*{3}}
\put(60,60){\circle*{3}}
\put(40,80){\circle*{3}}
\put(40,40){\circle*{3}}
\put(20,40){\line(0,1){20}}
\put(20,40){\line(1,1){20}}
\put(60,40){\line(-1,1){20}}
\put(60,40){\line(0,1){20}}
\put(40,80){\line(0,-1){60}}
\put(40,80){\line(-1,-1){20}}
\put(40,80){\line(1,-1){20}}
\put(14,34){$a$}
\put(42,34){$b$}
\put(61,34){$c$}
\put(14,62){$x$}
\put(34,60){$y$}
\put(61,62){$z$}
\put(43,67){$t$}
\put(38,83){$1$}
\put(25,0){$F$}
\end{picture}
&\hspace*{-30pt}
\begin{picture}(80,100)(0,0)
\put(40,20){\circle*{3}}
\put(38,10){$0$}
\put(40,20){\line(1,1){20}}
\put(40,20){\line(-1,1){20}}
\put(20,40){\circle*{3}}
\put(60,40){\circle*{3}}
\put(20,60){\circle*{3}}
\put(40,60){\circle*{3}}
\put(60,60){\circle*{3}}
\put(40,80){\circle*{3}}
\put(40,33){\circle*{3}}
\put(40,47){\circle*{3}}
\put(20,40){\line(0,1){20}}
\put(20,40){\line(1,1){20}}
\put(60,40){\line(-1,1){20}}
\put(60,40){\line(0,1){20}}
\put(40,80){\line(0,-1){60}}
\put(40,80){\line(-1,-1){20}}
\put(40,80){\line(1,-1){20}}
\put(14,34){$a$}
\put(42,30){$b$}
\put(61,34){$c$}
\put(33,44){$d$}
\put(14,62){$x$}
\put(34,60){$y$}
\put(61,62){$z$}
\put(38,83){$1$}
\put(25,0){$L$}
\end{picture}
&\hspace*{-30pt}
\begin{picture}(80,100)(0,0)
\put(40,20){\circle*{3}}

\put(38,10){$0$}
\put(40,20){\line(1,1){20}}

\put(40,20){\line(-1,1){20}}
\put(20,40){\circle*{3}}
\put(60,40){\circle*{3}}
\put(20,60){\circle*{3}}
\put(60,60){\circle*{3}}
\put(40,80){\circle*{3}}
\put(40,50){\circle*{3}}
\put(20,40){\line(0,1){20}}
\put(60,40){\line(0,1){20}}
\put(40,80){\line(0,-1){60}}
\put(40,80){\line(-1,-1){20}}
\put(40,80){\line(1,-1){20}}
\put(14,34){$a$}
\put(42,47){$b$}
\put(61,34){$c$}
\put(14,62){$x$}
\put(61,62){$z$}
\put(38,83){$1$}
\put(25,0){$Q$}
\end{picture}
&\hspace*{-30pt}
\begin{picture}(80,100)(0,0)
\put(40,20){\circle*{3}}
\put(38,10){$0$}
\put(40,20){\line(1,1){20}}
\put(40,20){\line(-1,1){20}}
\put(20,40){\circle*{3}}
\put(60,40){\circle*{3}}
\put(20,60){\circle*{3}}
\put(40,60){\circle*{3}}
\put(60,60){\circle*{3}}
\put(40,80){\circle*{3}}
\put(40,33){\circle*{3}}
\put(40,47){\circle*{3}}
\put(50,60){\circle*{3}}
\put(20,40){\line(0,1){20}}
\put(20,40){\line(1,1){20}}
\put(60,40){\line(-1,1){20}}
\put(60,40){\line(0,1){20}}
\put(40,80){\line(0,-1){60}}
\put(40,80){\line(-1,-1){20}}
\put(40,80){\line(1,-1){20}}
\put(40,80){\line(1,-2){20}}
\put(43,58){$u$}
\put(14,34){$a$}
\put(42,30){$b$}
\put(61,34){$c$}
\put(33,44){$d$}
\put(14,62){$x$}
\put(34,60){$y$}
\put(61,62){$z$}
\put(38,83){$1$}
\put(50,0){$R$}
\end{picture}
&\hspace*{-30pt}
\begin{picture}(80,100)(0,0)
\put(40,20){\circle*{3}}
\put(38,10){$0$}
\put(40,20){\line(1,1){20}}
\put(40,20){\line(-1,1){20}}
\put(20,40){\circle*{3}}
\put(60,40){\circle*{3}}
\put(20,60){\circle*{3}}
\put(40,60){\circle*{3}}
\put(60,60){\circle*{3}}
\put(40,80){\circle*{3}}
\put(40,80){\line(1,-2){20}}
\put(53,55){\circle*{3}}
\put(47,65){\circle*{3}}
\put(40,40){\circle*{3}}
\put(20,40){\line(0,1){20}}
\put(20,40){\line(1,1){20}}
\put(60,40){\line(-1,1){20}}
\put(60,40){\line(0,1){20}}
\put(40,80){\line(0,-1){60}}
\put(40,80){\line(-1,-1){20}}
\put(40,80){\line(1,-1){20}}
\put(14,34){$a$}
\put(42,34){$b$}
\put(61,34){$c$}
\put(14,62){$x$}
\put(34,60){$y$}
\put(61,62){$z$}

\put(55,53){$u$}
\put(41,63){$v$}
\put(38,83){$1$}
\put(50,0){$S$}
\end{picture}
&\hspace*{-30pt}
\begin{picture}(80,100)(0,0)
\put(40,20){\circle*{3}}
\put(38,10){$0$}
\put(40,20){\line(1,1){20}}
\put(40,20){\line(-1,1){20}}
\put(40,33){\circle*{3}}
\put(40,47){\circle*{3}}
\put(20,40){\circle*{3}}
\put(60,40){\circle*{3}}
\put(20,60){\circle*{3}}
\put(40,60){\circle*{3}}
\put(60,60){\circle*{3}}
\put(40,80){\circle*{3}}
\put(40,80){\line(1,-2){20}}
\put(53,55){\circle*{3}}
\put(47,65){\circle*{3}}

\put(20,40){\line(0,1){20}}
\put(20,40){\line(1,1){20}}
\put(60,40){\line(-1,1){20}}
\put(60,40){\line(0,1){20}}
\put(40,80){\line(0,-1){60}}
\put(40,80){\line(-1,-1){20}}
\put(40,80){\line(1,-1){20}}
\put(14,34){$a$}
\put(42,30){$b$}
\put(61,34){$c$}
\put(33,44){$d$}
\put(14,62){$x$}
\put(34,60){$y$}
\put(61,62){$z$}
\put(55,53){$u$}
\put(41,63){$v$}
\put(38,83){$1$}
\put(50,0){$T$}
\end{picture}\end{tabular}\end{center}\vspace*{-2pt}

${\rm Con}(E)=\{\Delta _E,\varepsilon ,\nabla _E\}\cong {\cal L}_3$, where $\varepsilon =eq(\{0\},\{a\},\{b,d\},\{c\},\{1\})$, so ${\rm Spec}(E)=\{\Delta _E,\varepsilon \}$.

${\rm Con}(F)=\{\Delta _F,\phi _1,\phi _2,\nabla _F\}\cong {\cal L}_4$, where $\phi _1=eq(\{0\},\{a\},\{b\},\{c\},\{x\},\{y,t\},\{z\},\{1\})$ and $\phi _2=eq(\{0\},$\linebreak $\{a,x\},\{b\},\{c,z\},\{y,t,1\})$, so ${\rm Spec}(F)=\{\Delta _F,\phi _1,\phi _2\}$.

${\rm Con}(L)=\{\Delta _L,\lambda _1,\lambda _2,\lambda _3,\nabla _L\}$, with the Hasse diagram below, where $\lambda _1=eq(\{0\},\{a\},\{b,d\},\{c\},\{x\},\{y\},$\linebreak $\{z\},\{1\})$, $\lambda _2=eq(\{0\},\{a,x\},\{b\},\{d\},\{c,z\},\{y,1\}$ and $\lambda _3=eq(\{0\},\{a,x\},\{b,d\},\{c,z\},\{y,1\}$, so ${\rm Spec}(L)=\{\lambda _1,\lambda _2,\lambda _3\}$.

${\rm Con}(Q)=\{\Delta _Q,\gamma _1,\gamma _2,\gamma _3,\nabla _Q\}$, with the Hasse diagram below, where $\gamma _1=eq(\{0\},\{a,x\},\{b\},\{c\},\{z\},\{1\})$, $\gamma _2=eq(\{0\},\{a\},\{x\},\{b\},\{c,z\},\{1\})$ and $\gamma _3=eq(\{0\},\{a,x\},\{b\},\{c,z\},\{1\})$, so ${\rm Spec}(Q)=\{\gamma _1,\gamma _2,\gamma _3\}$.

${\rm Con}(R)=\{\Delta _R,\rho ,\nabla _R\}\cong {\cal L}_3$, where $\rho =eq(\{0\},\{a\},\{b,d\},\{c\},\{x\},\{y\},\{u\},\{z\},\{1\})$, so ${\rm Spec}(R)=\{\Delta _R,\rho \}$.

${\rm Con}(S)=\{\Delta _S,\sigma ,\nabla _S\}\cong {\cal L}_3$, where $\sigma =eq(\{0\},\{a\},\{b\},\{c\},\{x\},\{y\},\{u,v\},\{z\},\{1\})$, so ${\rm Spec}(S)=\{\Delta _S,\sigma \}$.

${\rm Con}(T)=\{\Delta _T,\tau _1,\tau _2,\tau _3,\nabla _T\}$, with the Hasse diagram below, where $\tau _1 =eq(\{0\},\{a\},\{b\},\{c\},\{d\},\{x\},\{y\},$\linebreak $\{u,v\},\{z\},\{1\})$, $\tau _2 =eq(\{0\},\{a\},\{b,d\},\{c\},\{x\},\{y\},\{u\},\{v\},\{z\},\{1\})$ and $\tau _3 =eq(\{0\},\{a\},\{b,d\},\{c\},\{x\},$\linebreak $\{y\},\{u,v\},\{z\},\{1\})$, so ${\rm Spec}(T)=\{\tau _1,\tau _2,\tau _3\}$.\vspace*{-5pt}

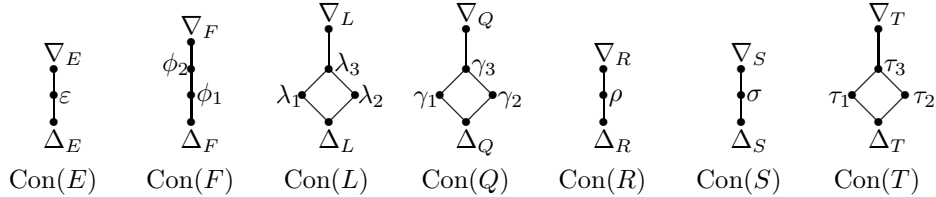
\begin{center}\begin{tabular}{ccccccc}
\begin{picture}(40,75)(0,0)
\put(20,25){\line(0,1){20}}
\put(20,25){\circle*{3}}
\put(20,35){\circle*{3}}
\put(20,45){\circle*{3}}
\put(16,15){$\Delta _E$}
\put(16,48){$\nabla _E$}
\put(22,32){$\varepsilon $}
\put(3,0){${\rm Con}(E)$}
\end{picture}
&
\begin{picture}(40,75)(0,0)
\put(20,25){\line(0,1){30}}
\put(20,25){\circle*{3}}
\put(20,35){\circle*{3}}

\put(20,45){\circle*{3}}
\put(20,55){\circle*{3}}
\put(16,15){$\Delta _F$}
\put(16,58){$\nabla _F$}
\put(9,44){$\phi _2$}
\put(22,33){$\phi _1$}
\put(3,0){${\rm Con}(F)$}
\end{picture}
&
\begin{picture}(40,75)(0,0) 
\put(20,25){\line(1,1){10}}
\put(20,25){\line(-1,1){10}}
\put(20,45){\line(1,-1){10}}
\put(20,45){\line(-1,-1){10}}
\put(20,45){\line(0,1){15}}
\put(20,25){\circle*{3}}
\put(20,60){\circle*{3}}
\put(10,35){\circle*{3}}
\put(30,35){\circle*{3}}
\put(20,45){\circle*{3}}
\put(16,15){$\Delta _L$}
\put(16,63){$\nabla _L$}
\put(22,44){$\lambda _3$}
\put(0,32){$\lambda _1$}
\put(31,32){$\lambda _2$}
\put(3,0){${\rm Con}(L)$}
\end{picture}
&
\begin{picture}(40,75)(0,0) 
\put(20,25){\line(1,1){10}}
\put(20,25){\line(-1,1){10}}
\put(20,45){\line(1,-1){10}}
\put(20,45){\line(-1,-1){10}}
\put(20,45){\line(0,1){15}}
\put(20,25){\circle*{3}}
\put(20,60){\circle*{3}}
\put(10,35){\circle*{3}}
\put(30,35){\circle*{3}}
\put(20,45){\circle*{3}}
\put(16,15){$\Delta _Q$}
\put(16,63){$\nabla _Q$}
\put(22,44){$\gamma _3$}
\put(0,32){$\gamma _1$}
\put(32,32){$\gamma _2$}
\put(3,0){${\rm Con}(Q)$}
\end{picture}
&
\begin{picture}(40,75)(0,0)
\put(20,25){\line(0,1){20}}
\put(20,25){\circle*{3}}
\put(20,35){\circle*{3}}
\put(20,45){\circle*{3}}
\put(16,15){$\Delta _R$}

\put(16,48){$\nabla _R$}
\put(22,33){$\rho $}
\put(3,0){${\rm Con}(R)$}
\end{picture}
&
\begin{picture}(40,75)(0,0)
\put(20,25){\line(0,1){20}}
\put(20,25){\circle*{3}}
\put(20,35){\circle*{3}}
\put(20,45){\circle*{3}}
\put(16,15){$\Delta _S$}
\put(16,48){$\nabla _S$}
\put(22,33){$\sigma $}
\put(3,0){${\rm Con}(S)$}
\end{picture}
&
\begin{picture}(40,75)(0,0) 
\put(20,25){\line(1,1){10}}
\put(20,25){\line(-1,1){10}}
\put(20,45){\line(1,-1){10}}
\put(20,45){\line(-1,-1){10}}
\put(20,45){\line(0,1){15}}
\put(20,25){\circle*{3}}
\put(20,60){\circle*{3}}
\put(10,35){\circle*{3}}
\put(30,35){\circle*{3}}
\put(20,45){\circle*{3}}
\put(16,15){$\Delta _T$}
\put(16,63){$\nabla _T$}
\put(22,44){$\tau _3$}
\put(1,32){$\tau _1$}
\put(33,32){$\tau _2$}
\put(3,0){${\rm Con}(T)$}
\end{picture}\end{tabular}\end{center}

Let $j:H\rightarrow R$, $k:L\rightarrow T$, $l:H\rightarrow L$ be the canonical bounded lattice embeddings and $m:R\rightarrow S$, $q:F\rightarrow L$, $r:E\rightarrow Q$ be the bounded lattice morphisms defined by:

$m(d)=b$ and $m(w)=w$ for all $w\in R\setminus \{d\}$,

$q(t)=y$ and $q(w)=w$ for all $q\in F\setminus \{t\}$,

$r(d)=b$ and $r(w)=w$ for all $r\in E\setminus \{d\}$.

None of these morphisms is surjective.

${\rm Con}(L)\cong {\rm Con}(T)$. $k^*(\tau _1)=\Delta _L\notin {\rm Spec}(L)$, so $k$ is not admissible.

${\rm Con}(H)\cong {\rm Con}(R)$. ${\rm Ker}(j)=j^*(\Delta _R)=j^*(\rho )=\Delta _H$, so, as in the case of $i$, $j$ is admissible and does not fulfill GU or LO. $j^*\mid _{{\rm Con}(R)}:{\rm Con}(R)\rightarrow {\rm Con}(H)$ and $j^*\mid _{{\rm Spec}(R)}:{\rm Spec}(R)\rightarrow {\rm Spec}(H)$ are neither injective, nor surjective.

${\rm Con}(H)\ncong {\rm Con}(L)$ and $|{\rm Spec}(H)|\neq |{\rm Spec}(L)|$. ${\rm Ker}(l)=l^*(\Delta _L)=l^*(\lambda _1)=\Delta _H$ and $l^*(\lambda _2)=l^*(\lambda _3)=\chi $, thus $l$ is admissible and fulfills GU and LO. $l^*\mid _{{\rm Con}(L)}:{\rm Con}(L)\rightarrow {\rm Con}(H)$ and $l^*\mid _{{\rm Spec}(L)}:{\rm Spec}(L)\rightarrow {\rm Spec}(H)$ are surjective and they are not injective.

${\rm Con}(E)\ncong {\rm Con}(Q)$ and $|{\rm Spec}(E)|\neq |{\rm Spec}(Q)|$. ${\rm Ker}(r)=r^*(\Delta _Q)=r^*(\gamma _1)=r^*(\gamma _2)=r^*(\gamma _3)=\varepsilon $, thus $r$ is admissible and fulfills GU and LO. $r^*\mid _{{\rm Con}(Q)}:{\rm Con}(Q)\rightarrow {\rm Con}(E)$ and $r^*\mid _{{\rm Spec}(Q)}:{\rm Spec}(Q)\rightarrow {\rm Spec}(E)$ are neither injective, nor surjective.

${\rm Con}(F)\ncong {\rm Con}(L)$; $|{\rm Spec}(F)|=|{\rm Spec}(L)|$, but the posets $({\rm Spec}(F),\subseteq )$ and $({\rm Spec}(L),\subseteq )$ are not isomorphic. ${\rm Ker}(q)=q^*(\Delta _L)=q^*(\lambda _1)=\phi _1$ and $q^*(\lambda _2)=q^*(\lambda _3)=\phi _2$, thus $q$ is admissible and fulfills GU and LO. $q^*\mid _{{\rm Con}(L)}:{\rm Con}(L)\rightarrow {\rm Con}(F)$ and $q^*\mid _{{\rm Spec}(L)}:{\rm Spec}(L)\rightarrow {\rm Spec}(F)$ are neither injective, nor surjective. 

${\rm Con}(R)\cong {\rm Con}(S)$. ${\rm Ker}(m)=m^*(\Delta _S)=m^*(\sigma )=\rho $, thus $m$ is admissible and fulfills GU and LO. $m^*\mid _{{\rm Con}(S)}:{\rm Con}(S)\rightarrow {\rm Con}(R)$ and $m^*\mid _{{\rm Spec}(S)}:{\rm Spec}(S)\rightarrow {\rm Spec}(R)$ are neither injective, nor surjective.\label{exadmgulo}\end{example}

\begin{lemma}\begin{enumerate}
\item\label{speccat1} For all $\theta \in {\rm Con}(M)$, ${\rm Spec}(M/\theta )=p_{\theta }(V_M(\theta ))$ and the mapping $\gamma \mapsto \gamma /\theta $ sets an order isomorphism from $(V_M(\theta ),\subseteq )$ to $({\rm Spec}(M/\theta ),\subseteq )$.
\item\label{speccat2} ${\rm Con}(h(M))=h([{\rm Ker}(h)))$.
\item\label{speccat3} ${\rm Spec}(h(M))=h(V_M({\rm Ker}(h)))$.
\item\label{speccat0} If $h$ is surjective, then ${\rm Spec}(N)=h(V_M({\rm Ker}(h)))$ and $V_M({\rm Ker}(h))=h^*({\rm Spec}(N))$.
\item\label{speccat4} For all $\alpha ,\beta \in {\rm Con}(M)$, $h(\alpha )\subseteq h(\beta )$ iff $\alpha \subseteq \beta $.
\item\label{speccat5} For all $\theta \in [{\rm Ker}(h))$, $V_{h(M)}(h(\theta ))=h(V_M({\rm Ker}(h)))$.\end{enumerate}\label{speccat}\end{lemma}

\begin{proof} (\ref{speccat1}) Let $\theta \in {\rm Con}(M)$, and recall that the mapping $\gamma \mapsto p_{\theta }(\gamma )=\gamma /\theta $ sets a bounded lattice isomorphism from $[\theta )$ to ${\rm Con}(M/\theta )$. Let $\psi ,\gamma ,\delta \in {\rm Con}(M/\theta )$, so that $\psi =\phi /\theta $, $\gamma =\alpha /\theta $ and $\delta =\beta /\theta $ for some $\phi ,\alpha ,\beta \in [\theta )$. We have the following equivalences, according to Remark \ref{r2.8}: $[\gamma ,\delta ]_{M/\theta }\subseteq \psi $ iff $[\alpha /\theta ,\beta /\theta ]_{M/\theta }\subseteq \phi /\theta $ iff $([\alpha,\beta ]_M\vee \theta )/\theta \subseteq \phi /\theta $ iff $[\alpha,\beta ]_M\vee \theta \subseteq \phi $ iff $[\alpha,\beta ]_M\subseteq \phi $, since $\phi \supseteq \theta $. We also have: $\gamma \subseteq \psi $ iff $\alpha /\theta \subseteq \phi /\theta $ iff $\alpha \subseteq \phi $, and: $\delta \subseteq \psi $ iff $\beta /\theta \subseteq \phi /\theta $ iff $\beta \subseteq \phi $. Therefore: $\phi /\theta =\psi \in {\rm Spec}(M/\theta )$ iff $[\gamma ,\delta ]_{M/\theta }\subseteq \psi $ implies $\gamma \subseteq \psi $ or $\delta \subseteq \psi $, iff $[\alpha ,\beta ]_M\subseteq \phi $ implies $\alpha \subseteq \phi $ or $\beta \subseteq \phi $ iff $\phi \in {\rm Spec}(M)$. Hence ${\rm Spec}(M/\theta )=\{\phi /\theta \ |\ \phi \in [\theta )\cap {\rm Spec}(M)\}=\{\phi /\theta \ |\ \phi \in V_M(\theta )\}=p_{\theta }(V_M(\theta ))$. Thus the mapping above sets a surjection from $V_M(\theta )$ to ${\rm Spec}(M/\theta )$; since it sets a bounded lattice isomorphism, thus an order isomorphism, from $[\theta )$ to ${\rm Con}(M/\theta )$, it follows this map is also injective, thus it is a bijection from $V_M(\theta )$ to ${\rm Spec}(M/\theta )$, hence it is an order isomorphism between these orderred sets.

\noindent (\ref{speccat2}) By the Fundamental Isomorphism Theorem, the map $\varphi :M/{\rm Ker}(h)\rightarrow h(M)$, defined by $\varphi (a/{\rm Ker}(h))=h(a)$ for all $a\in A$, is well defined and it is an isomorphism in ${\cal C}$. Hence ${\rm Con}(h(M))=\{\varphi (\gamma )\ |\ \gamma \in {\rm Con}(M/{\rm Ker}(h))\}=\{\varphi (\theta /{\rm Ker}(h))\ |\ \theta \in [{\rm Ker}(h))\}=\{h(\theta )\ |\ \theta \in [{\rm Ker}(h))\}=h([{\rm Ker}(h)))$.

\noindent (\ref{speccat3}) By (\ref{speccat1}) and (\ref{speccat2}) and its proof, ${\rm Spec}(h(M))=\{\varphi (\psi )\ |\ \psi \in {\rm Spec}(M/{\rm Ker}(h))\}=\{\varphi (\phi /{\rm Ker}(h))\ |\ \phi \in V_M({\rm Ker}(h))\}=\{h(\phi )\ |\ \phi \in V_M({\rm Ker}(h))\}=h(V_M({\rm Ker}(h)))$.

\noindent (\ref{speccat0}) By (\ref{speccat3}) and the surjectivity of $h$, ${\rm Spec}(N)={\rm Spec}(h(M))=h(V_M({\rm Ker}(h)))$ and, for all $\theta \in V_M({\rm Ker}(h))$, $h^*(h(\theta ))=\theta $, thus $V_M({\rm Ker}(h))=h^*(h(V_M({\rm Ker}(h))))=h^*({\rm Spec}(N))$.

\noindent (\ref{speccat4}) By the proof of (\ref{speccat2}), $h(\alpha )\subseteq h(\beta )$ iff $\varphi (\alpha /{\rm Ker}(h))\subseteq \varphi (\beta /{\rm Ker}(h))$ iff $\alpha /{\rm Ker}(h)\subseteq \beta /{\rm Ker}(h)$ iff $\alpha \subseteq \beta $.

\noindent (\ref{speccat5}) By (\ref{speccat2}) and its proof, along with (\ref{speccat3}) and (\ref{speccat4}), for all $\theta \in [{\rm Ker}(h))$, $V_{h(M)}(h(\theta ))=[h(\theta ))\cap {\rm Spec}(h(M))=[h(\theta ))\cap \{h(\phi )\ |\ \phi \in V_M({\rm Ker}(h))\}=\{h(\phi )\ |\ \phi \in V_M({\rm Ker}(h)),h(\theta )\subseteq h(\phi )\}=\{h(\phi )\ |\ \phi \in V_M({\rm Ker}(h)),\theta \subseteq \phi \}=\{h(\phi )\ |\ \phi \in {\rm Spec}(M)\cap [{\rm Ker}(h))\cap [\theta )\}=\{h(\phi )\ |\ \phi \in {\rm Spec}(M)\cap [\theta )\}=\{h(\phi )\ |\ \phi \in V_M(\theta )\}=h(V_M({\rm Ker}(h)))$, since ${\rm Ker}(h)\subseteq \theta $.\end{proof}

\begin{proposition} If $h$ is surjective, then $h$ is admissible and fulfills GU and LO. The converse is not true, not even when ${\cal C}$ is congruence--distributive and semi--degenerate.\label{surjgulo}\end{proposition}

\begin{proof} By Lemma \ref{surjadm}, $h$ is admissible. By Lemma \ref{speccat}, (\ref{speccat0}), and Lemma \ref{cargulo}, (\ref{cargulo2}), $V_M({\rm Ker}(h))=h^*({\rm Spec}(N))$, thus $h$ fulfills LO. Now let let $\phi ,\psi \in {\rm Spec}(M)$ and $\phi _1\in {\rm Spec}(N)$ such that $h^*(\phi _1)=\phi $ and $\phi \subseteq \psi $. Then, by Remark \ref{ker}, ${\rm Ker}(h)\subseteq \phi $, thus ${\rm Ker}(h)\subseteq \psi $, so, since $h$ fulfills LO, it follows that $h^*(\psi _1)=\psi $ for some $\psi _1\in {\rm Spec}(N)$. We have $h^*(\phi _1)=\phi \subseteq \psi =h^*(\psi _1)$, hence, by the surjectivity of $h$, $\phi _1=h(h^*(\phi _1))\subseteq h(h^*(\psi _1))=\psi _1$. Thus $h$ fulfills GU.

The bounded lattice morphisms $l$, $r$, $q$ and $m$ from Example \ref{exadmgulo} are admissible and fulfill GU and LO, but they are not surjective.\end{proof}

\begin{corollary} For every $\theta \in {\rm Con}(A)$, $p_{\theta }$ is admissible and fulfills GU and LO.\label{3.5}\end{corollary}

\begin{lemma}\begin{enumerate}
\item\label{gulocomp1} If $f$ and $g$ fulfill GU, then $g\circ f$ fulfills GU. 
\item\label{gulocomp2} If $f$ is surjective and $g$ fulfills LO, then $g\circ f$ fulfills LO.
\item\label{gulocomp3} If $f$ and $g$ fulfill LO and $g$ is injective, then $g\circ f$ fulfills LO.\end{enumerate}\label{gulocomp}\end{lemma}

\begin{proof} (\ref{gulocomp1}) Let $\phi ,\psi \in {\rm Spec}(A)$ and $\phi _2\in {\rm Spec}(C)$ such that $\phi \subseteq \psi $ and $\phi =(g\circ f)^*(\phi _2)=f^*(g^*(\phi _2))$. Denote $\phi _1=g^*(\phi _2)\in {\rm Spec}(B)$, so that $f^*(\phi _1)=\phi $. Since $f$ fulfills GU, it follows that there exists a $\psi _1\in {\rm Spec}(B)$ such that $\phi _1\subseteq \psi _1$ and $f^*(\psi _1)=\psi $. Since $g^*(\phi _2)=\phi _1$ and $g$ fulfills GU , it follows that there exists a $\psi _2\in {\rm Spec}(C)$ such that $\phi _2\subseteq \psi _2$ and $g^*(\psi _2)=\psi _1$. Then $(g\circ f)^*(\psi _2)=f^*(g^*(\psi _2))=f^*(\psi _1)=\psi $. Therefore $g\circ f$ fulfills GU. 

\noindent (\ref{gulocomp2}) Since $f$ is surjective, it follows that $f$ fulfills LO by Proposition \ref{surjgulo}. Let $\phi \in {\rm Spec}(A)$ such that ${\rm Ker}(g\circ f)\subseteq \phi $. Then, by Remark \ref{kerker}, the fact that $f$ fulfills LO and the surjectivity of $f$, ${\rm Ker}(f)\subseteq \phi $, there exists a $\phi _1\in {\rm Spec}(B)$ such that $f^*(\phi _1)=\phi \supseteq {\rm Ker}(g\circ f)=(g\circ f)^*(\Delta _C)=(f^*\circ g^*)(\Delta _C)=f^*(g^*(\Delta _C))=f^*({\rm Ker}(g))$, therefore $\phi _1=f(f^*(\phi _1))\supseteq f(f^*({\rm Ker}(g)))={\rm Ker}(g)$, hence there exists a $\phi _2\in {\rm Spec}(C)$ such that $g^*(\phi _2)=\phi _1$, so $(g\circ f)^*(\phi _2)=(f^*\circ g^*)(\phi _2)=f^*(g^*(\phi _2))=f^*(\phi _1)=\phi $. Thus $g\circ f$ fulfills LO.

\noindent (\ref{gulocomp3}) Let $\phi \in {\rm Spec}(A)$ such that ${\rm Ker}(g\circ f)\subseteq \phi $. Then, by Remark \ref{kerker} and the fact that $f$ fulfills LO, ${\rm Ker}(f)\subseteq \phi $, hence there exists a $\phi _1\in {\rm Spec}(B)$ such that $f^*(\phi _1)=\phi $. Since $g$ is injective and fulfills LO, ${\rm Ker}(g)=g^*(\Delta _C)=\Delta _B\subseteq \phi _1$, hence there exists a $\phi _2\in {\rm Spec}(C)$ such that $g^*(\phi _2)=\phi _1$, so $(g\circ f)^*(\phi _2)=(f^*\circ g^*)(\phi _2)=f^*(g^*(\phi _2))=f^*(\phi _1)=\phi $. Thus $g\circ f$ fulfills LO.\end{proof}

\begin{proposition} Let $i:h(M)\rightarrow N$ be the canonical embedding. Then:\begin{enumerate}
\item\label{embed1} $h$ is admissible iff $i$ is admissible;
\item\label{embed2} if $h$ is admissible, then: $h$ fulfills GU iff $i$ fulfills GU;
\item\label{embed3} if $h$ is admissible, then: $h$ fulfills LO iff $i$ fulfills LO.\end{enumerate}\label{embed}\end{proposition}

\begin{proof} Let $s:M\rightarrow h(M)$, for all $x\in M$, $s(x)=h(x)$. Then $s$ is surjective, thus $s$ is admissible and fulfills GU and LO by Proposition \ref{surjgulo}. We have: $h=i\circ s$, so $h^*=s^*\circ i^*$. Since $s$ is surjective, it follows that $s^*$ is injective. For all $\theta \in {\rm Con}(N)$, $i^*(\theta )=\theta \cap h(M^2)=h(h^*(\theta ))=s(h^*(\theta ))$.\vspace*{-15pt}

\begin{center}\begin{picture}(120,53)(0,0)
\put(5,30){$M$}
\put(15,33){\vector (1,0){92}}
\put(13,28){\vector (3,-2){35}}
\put(73,4){\vector (3,2){36}}
\put(56,34){$h$}
\put(23,12){$s$}
\put(92,10){$i$}
\put(108,30){$N$}

\put(48,0){$h(M)$}
\end{picture}\end{center}\vspace*{-5pt}

\noindent (\ref{embed1}) $s$ is admissible, thus, by Remark \ref{compadm}, if $i$ is admissible, then $h=s\circ i$ is admissible. Now assume that $h$ is admissible, and let $\chi \in {\rm Spec}(N)$, so that $h^*(\chi )\in {\rm Spec}(M)$ and, by Remark \ref{ker}, ${\rm Ker}(h)\subseteq h^*(\chi )$, thus $h^*(\chi )\in V_M({\rm Ker}(h))$, so that $i^*(\chi )=h(h^*(\chi ))\in {\rm Spec}(h(M))$ by Lemma \ref{speccat}, (\ref{speccat3}), hence $i$ is admissible. 

From now until the end of this proof, $h$ shall be admissible, so that, by (\ref{embed1}), $i$ shall be admissible, too.

\noindent (\ref{embed2}) $s$ fulfills GU, thus, by Lemma \ref{gulocomp}, (\ref{gulocomp1}), if $i$ fulfills GU, then $h=s\circ i$ fulfills GU. Now assume that $h$ fulfills GU, and let $\phi _1,\psi _1\in {\rm Spec}(h(M))$ and $\phi _2\in {\rm Spec}(N)$ such that $\phi _1\subseteq \psi _1$ and $\phi _1=i^*(\phi _2)$. Let $\phi =h^*(\phi _2)\in {\rm Spec}(M)$, since $h$ is admissible, and $\psi =s^*(\psi _1)\in {\rm Spec}(M)$, since $s$ is admissible. Then $\phi _1=i^*(\phi _2)=h(h^*(\phi _2))=h(\phi )$ and, since $s$ is surjective, $\psi _1=s(s^*(\psi _1))=s(\psi )=h(\psi )$. We have $h(\phi )=\phi _1\subseteq \psi _1=h(\psi )$, hence $\phi \subseteq \psi $ by Lemma \ref{speccat}, (\ref{speccat4}), so, since $h$ fulfills GU, it follows that there exists a $\psi _2\in {\rm Spec}(N)$ such that $\phi _2\subseteq \psi _2$ and $\psi =h^*(\psi _2)$, so that $i^*(\psi _2)=h(h^*(\psi _2))=h(\psi )=\psi _1$. Therefore $i$ fulfills GU.

\noindent (\ref{embed3}) $s$ is surjective, thus, by Lemma \ref{gulocomp}, (\ref{gulocomp2}), if $i$ fulfills LO, then $h=s\circ i$ fulfills LO. Now assume that $h$ fulfills LO, and let $\psi \in {\rm Spec}(h(M))$. Trivially, $\psi \supseteq \Delta _{h(M)}={\rm Ker}(i)$. Since $s$ is admissible, $s^*(\psi )\in {\rm Spec}(M)$ and, by Remark \ref{ker}, $s^*(\psi )\supseteq {\rm Ker}(s)=s^*(\Delta _{h(M)})=s^*(i^*(\Delta _N))=h^*(\Delta _N)={\rm Ker}(h)$. Since $h$ fulfills LO, it follows that there exists a $\chi \in {\rm Spec}(N)$ such that $s^*(\psi )=h^*(\chi )$, so $i^*(\chi )=h(h^*(\chi ))=h(s^*(\psi ))=s(s^*(\psi ))=\psi $, by the surjectivity of $\psi $. Hence $i$ fulfills LO.\end{proof}

\begin{corollary}\begin{enumerate} 
\item\label{sufembed1} The following are equivalent:\begin{itemize}
\item any morphism in ${\cal C}$ is admissible;
\item any canonical embedding in ${\cal C}$ is admissible.\end{itemize}
\item\label{sufembed2} The following are equivalent:\begin{itemize}
\item any admissible morphism in ${\cal C}$ fulfills GU;
\item any admissible canonical embedding in ${\cal C}$ fulfills GU.\end{itemize}
\item\label{sufembed3} The following are equivalent:\begin{itemize}
\item any admissible morphism in ${\cal C}$ fulfills LO;
\item any admissible canonical embedding in ${\cal C}$ fulfills LO.\end{itemize}\end{enumerate}\label{sufembed}\end{corollary}

\begin{lemma} Let $\alpha \in {\rm Con}(A)$ and $a,b\in A$. If $f$ is surjective, then:\begin{enumerate}
\item\label{lo1(1)} {\rm \cite[Proposition 1.2,1,v]{urs}} $Cg_B(f(\alpha ))=f(\alpha \vee {\rm Ker}(f))$;
\item\label{lo1(2)} {\rm \cite[Proposition 1.2,2]{urs}} $Cg_B(f(Cg_A(a,b)))=Cg_B(f(a),f(b))$.\end{enumerate}\label{lo1}\end{lemma}

\begin{lemma} For all $a,b,c,d\in A$, $f([Cg_A(a,b),Cg_A(c,d)]_A)\subseteq [Cg_B(f(a),f(b)),Cg_B(f(c),f(d))]_B$.\label{lo2}\end{lemma}

\begin{proof} $f:A\rightarrow f(A)$ is a surjective morphism. Let $a,b,c,d\in A$. By Remark \ref{r2.5}, $f([Cg_A(a,b),Cg_A(c,d)]_A\vee {\rm Ker}(f))=[f(Cg_A(a,b)\vee {\rm Ker}(f)),f(Cg_A(c,d)\vee {\rm Ker}(f))]_{f(A)}$. By Lemma \ref{lo1}, $f(Cg_A(a,b)\vee {\rm Ker}(f))=$\linebreak $Cg_{f(A)}(f(Cg_A(a,b))=Cg_{f(A)}(f(a),f(b))\subseteq Cg_B(f(a),f(b))\cap f(A^2)$ and, analogously, $f(Cg_A(c,d)\vee {\rm Ker}(f))\subseteq Cg_B(f(c),f(d))\cap f(A^2)$. By Proposition \ref{1.3} and Lemma \ref{1.8}, it follows that $[f(Cg_A(a,b)\vee {\rm Ker}(f)),f(Cg_A(c,d)\vee {\rm Ker}(f))]_{f(A)}\subseteq [Cg_B(f(a),f(b))\cap f(A^2),Cg_B(f(c),f(d))\cap f(A^2)]_{f(A)}\subseteq [Cg_B(f(a),f(b)),Cg_B(f(c),f(d))]_B$. Trivially, $f([Cg_A(a,b),Cg_A(c,d)]_A)\subseteq f([Cg_A(a,b),Cg_A(c,d)]_A\vee {\rm Ker}(f))$. From all the above, it follows that $f([Cg_A(a,b),Cg_A(c,d)]_A)\subseteq [Cg_B(f(a),f(b)),Cg_B(f(c),f(d))]_B$.\end{proof}

\begin{lemma}\begin{enumerate}
\item\label{fmsist1} If $S$ is an m--system in $M$, then $h(S)$ is an m--system in $N$.
\item\label{fmsist2} If $M\subseteq N$ and $S$ is an m--system in $M$, then $S$ is an m--system in $N$.\end{enumerate}\label{fmsist}\end{lemma}

\begin{proof} (\ref{fmsist1}) Let $(x,y),(z,u)\in h(S)$, so that $x=h(a),y=h(b),z=h(c),u=h(d)$ for some $(a,b),(c,d)\in S\subseteq M^2$. Since $S$ is an m--system, it follows that $[Cg_M(a,b),Cg_M(c,d)]_M\cap S\neq \emptyset $, thus, by Lemma \ref{lo2}, $\emptyset \neq h([Cg_M(a,b),Cg_M(c,d)]_M\cap S)\subseteq h([Cg_M(a,b),Cg_M(c,d)]_M)\cap h(S)\subseteq [Cg_N(h(a),h(b)),Cg_N(h(c),h(d))]_N\cap h(S)=[Cg_N(x,y),Cg_N(z,u)]_N\cap h(S)$, thus $[Cg_N(x,y),Cg_N(z,u)]_N\cap h(S)\neq \emptyset $. Therefore $h(S)$ is an m--system.

\noindent (\ref{fmsist2}) Let $i:M\rightarrow N$ be the canonical embedding. Then $i$ is a morphism, so, by (\ref{fmsist1}), $i(S)=S$ is an m--system in $N$.\end{proof}

\begin{lemma} If $A\subseteq B$, $\phi \in {\rm Spec}(A)$, $\nabla _B$ is finitely generated and $\psi $ is a maximal element of the set $\{\theta \in {\rm Con}(B)\setminus \{\nabla _B\}\ |\ \theta \cap (\nabla _A\setminus \phi )=\emptyset \}$, then $\psi \in {\rm Spec}(B)$.\label{specmsist}\end{lemma}

\begin{proof} Since $\phi \in {\rm Spec}(A)$, by Lemma \ref{1.6} it follows that $\nabla _A\setminus \phi $ is an m--system in $A$. Then, by Lemma \ref{fmsist}, (\ref{fmsist2}), $\nabla _A\setminus \phi $ is an m--system in $B$. By Lemma \ref{1.7}, it follows that $\psi \in {\rm Spec}(B)$.\end{proof}

\begin{proposition} Assume that $\nabla _B$ is finitely generated, $A\subseteq B$ and the canonical embedding $i:A\rightarrow B$ is admissible. Then the following are equivalent:\begin{enumerate}
\item\label{altacargu1} $i$ fulfills GU;
\item\label{altacargu2} for all $\phi \in {\rm Spec}(A)$, if $\psi $ is a maximal element of the set $\{\theta \in {\rm Con}(B)\setminus \{\nabla _B\}\ |\ \theta \cap (\nabla _A\setminus \phi )=\emptyset \}$, then $\psi \cap A^2=\phi $.\end{enumerate}\label{altacargu}\end{proposition}

\begin{proof} (\ref{altacargu1})$\Rightarrow $(\ref{altacargu2}): Let $\phi $ and $\psi $ be as in the enunciation. Then $\psi \in {\rm Spec}(B)$ by Lemma \ref{specmsist}. Let $\phi _0=\psi \cap A^2=i^*(\psi )\in {\rm Spec}(A)$, because $i$ is admissible. We have: $\emptyset =\psi \cap (\nabla _A\setminus \phi)=(\psi \cap \nabla _A)\setminus (\psi \cap \phi)=(\psi \cap A^2)\setminus (\psi \cap \phi)$, thus $\psi \cap A^2\subseteq \psi \cap \phi \subseteq \psi \cap A^2$, because $\phi \subseteq A^2$. Hence $\phi _0=\psi \cap A^2=\psi \cap \phi \subseteq \phi $. Since $i$ fulfills GU, it follows that there exists a $\psi _1\in {\rm Spec}(B)$ such that $\psi _1\cap A^2=i^*(\psi _1)=\phi $ and $\psi \subseteq \psi _1$. Then $\psi _1\neq \nabla _B$ and $\psi _1\cap (\nabla _A\setminus \phi)=\psi _1\cap \nabla _A\cap (\nabla _A\setminus \phi)=\phi \cap (\nabla _A\setminus \phi)=\emptyset $. By the maximality of $\psi $, it follows that $\psi =\psi _1$, hence $\phi =\psi _1\cap A^2=\psi \cap A^2$.

\noindent (\ref{altacargu2})$\Rightarrow $(\ref{altacargu1}): Let $\phi _0,\phi \in {\rm Spec}(A)$ and $\psi _0\in {\rm Spec}(B)$ such that $\phi _0\subseteq \phi $ and $\psi _0\cap A^2=i^*(\psi _0)=\phi _0$. $\psi _0\cap (\nabla _A\setminus \phi )=\psi _0\cap \nabla _A\cap (\nabla _A\setminus \phi )=\phi _0\cap (\nabla _A\setminus \phi )\subseteq \phi \cap (\nabla _A\setminus \phi )=\emptyset $, because $\phi _0\subseteq \phi $; so $\psi _0\cap (\nabla _A\setminus \phi )=\emptyset $. Let $\psi $ be a maximal element of the set $\{\theta \in {\rm Con}(A)\setminus \{\nabla _A\}\ |\ \psi _0\subseteq \theta ,\theta \cap (\nabla _A\setminus \phi )=\emptyset \}$. Then it is straightforward that $\psi $ is a maximal element of the set $\{\theta \in {\rm Con}(A)\setminus \{\nabla _A\}\ |\ \theta \cap (\nabla _A\setminus \phi )=\emptyset \}$, so $\psi \in {\rm Spec}(B)$ by Lemma \ref{specmsist}, and $i^*(\psi )=\psi \cap A^2=\phi $ by the hypothesis of this implication. Thus $i$ fulfills GU.\end{proof}

\begin{proposition} Assume that $\nabla _B$ is finitely generated, $A\subseteq B$ and the canonical embedding $i:A\rightarrow B$ is admissible. Then: if $i$ fulfills GU, then $i$ fulfills LO.\label{embedgulo}\end{proposition}

\begin{proof} Assume that $i$ fulfills GU, and let $\phi \in {\rm Spec}(A)$; of course, $\nabla _A=\nabla _B\cap A^2=i^*(\nabla _B)\subseteq \phi $. Then, by Proposition \ref{altacargu} and Lemma \ref{specmsist}, there exists a $\psi \in {\rm Spec}(B)$ such that $i^*(\psi )=\psi \cap A^2=\phi $, therefore $i$ fulfills LO.\end{proof}

\begin{proposition} If $\nabla _B$ is finitely generated and $f$ fulfills GU, then $f$ fulfills LO, but the converse is not true.\label{gulo}\end{proposition}

\begin{proof} By Propositions \ref{embed} and \ref{embedgulo}, if $f$ fulfills GU, then $f$ fulfills LO. 

See in \cite[Exercise 3, p. 41]{kap} a type of ring extension which proves that not all admissible morphisms fulfilling LO from a semi--degenerate congruence--modular equational class also fulfill GU.\end{proof}

\begin{corollary} If ${\cal C}$ is semi--degenerate, then, in ${\cal C}$, GU implies LO, but the converse is not true.\label{classgulo}\end{corollary}

\begin{proof} By Propositions \ref{2.6} and \ref{gulo}.\end{proof}

\section{Going Up and Lying Over in Particular Cases}
\label{gulopartic}

In this section we list some cases in which admissibility and GU hold and we show how admissibility, GU and LO relate to each other in some particular cases. Throughout this section, ${\cal C}$ shall be congruence--modular, $A,B,M,N$ shall be members of ${\cal C}$, $f:A\rightarrow B$ shall be an admissible morphism in ${\cal C}$ and $h:M\rightarrow N$ shall be a morphism in ${\cal C}$, not necessarily admissible.

\begin{lemma}\begin{enumerate}
\item\label{specmaxgu1} If ${\rm Spec}(A)={\rm Max}(A)$, then $f$ fulfills GU.
\item\label{specmaxgu2} If the commutator in $A$ equals the intersection of congruences and ${\rm Con}(A)$ is a Boolean algebra, then $f$ fulfills GU.
\item\label{specmaxgu3} If ${\cal C}$ is congruence--distributive and ${\rm Con}(A)$ is a Boolean algebra, then $f$ fulfills GU.\end{enumerate}\label{specmaxgu}\end{lemma}

\begin{proof} (\ref{specmaxgu1}) Let $\phi ,\psi \in {\rm Spec}(A)$ and $\phi _1\in {\rm Spec}(B)$ such that $f^*(\phi _1)=\phi $ and $\phi \subseteq \psi $. Then, by Remark \ref{unord}, $\phi =\psi $, so we may take $\psi _1=\phi _1\in {\rm Spec}(B)$ and we have: $\phi _1=\psi _1\subseteq \psi _1$ and $f^*(\psi _1)=f^*(\phi _1)=\phi =\psi $, hence $f$ fulfills GU.

\noindent (\ref{specmaxgu2}) By (\ref{specmaxgu1}) and Lemma \ref{specmax}, (\ref{specmax3}).

\noindent (\ref{specmaxgu3}) By (\ref{specmaxgu1}) and Lemma \ref{specmax}, (\ref{specmax4}).\end{proof}

\begin{proposition}\begin{enumerate}
\item\label{specmaxadmgu1} If ${\cal C}$ is semi--degenerate, ${\rm Spec}(M)={\rm Max}(M)$ and ${\rm Spec}(N)={\rm Con}_2(N)$, then $h$ is admissible and fulfills GU.
\item\label{specmaxadmgu2} If ${\cal C}$ is semi--degenerate, the commutator in $M$ equals the intersection of congruences, ${\rm Con}(M)$ is a Boolean algebra and ${\rm Spec}(N)={\rm Con}_2(N)$, then $h$ is admissible and fulfills GU.
\item\label{specmaxadmgu3} If ${\cal C}$ is semi--degenerate and congruence--distributive, ${\rm Con}(M)$ is a Boolean algebra and ${\rm Spec}(N)={\rm Con}_2(N)$, then $h$ is admissible and fulfills GU.
\item\label{specmaxadmgu4} If $h^*(\{\nabla _M\})=\{\nabla _N\}$, ${\rm Spec}(M)={\rm Max}(M)$ and ${\rm Spec}(N)={\rm Con}_2(N)$, then $h$ is admissible and fulfills GU. 
\item\label{specmaxadmgu5} If $h^*(\{\nabla _M\})=\{\nabla _N\}$, the commutator in $M$ equals the intersection, ${\rm Con}(M)$ is a Boolean algebra and ${\rm Spec}(N)={\rm Con}_2(N)$, then $h$ is admissible and fulfills GU. 
\item\label{specmaxadmgu6} If $h^*(\{\nabla _M\})=\{\nabla _N\}$, ${\cal C}$ is congruence--distributive, ${\rm Con}(M)$ is a Boolean algebra and ${\rm Spec}(N)={\rm Con}_2(N)$, then $h$ is admissible and fulfills GU.\end{enumerate}\label{specmaxadmgu}\end{proposition}

\begin{proof} By Lemmas \ref{2clsadm} and \ref{specmaxgu}.\end{proof}

\begin{proposition} Any morphism in the class of bounded distributive lattices is admissible and fulfills GU.\label{distriblatgu}\end{proposition}

\begin{proof} By Proposition \ref{bdlat}, (\ref{bdlat1}), Lemma \ref{speclat}, (\ref{speclat1}), and Lemma \ref{specmaxgu}, (\ref{specmaxgu1}).\end{proof}

\begin{proposition} Let $L^{\prime }$ be a bounded lattice, $L$ be a bounded lattice that can be obtained through finite direct products and/or finite ordinal sums from bounded distributive lattices and/or finite modular lattices and/or relatively complemented bounded lattices with ACC and $m:L\rightarrow L^{\prime }$ be a bounded lattice morphism.

\begin{enumerate}
\item\label{latgu1} If $m$ is admissible, then $m$ fulfills GU.
\item\label{latgu2} If $L^{\prime }$ is distributive, then $m$ is admissible and fulfills GU.
\item\label{latgu3} If $m(L)=\{0,1\}$, then $m$ is admissible and fulfills GU.\end{enumerate}\label{latgu}\end{proposition}

\begin{proof} (\ref{latgu1}) By Lemma \ref{speclat}, (\ref{speclat2}), and Lemma \ref{specmaxgu}.

\noindent (\ref{latgu2}) By (\ref{latgu1}) and Proposition \ref{bdlat}, (\ref{bdlat2}).

\noindent (\ref{latgu3}) By (\ref{latgu2}) and Remark \ref{2cls01}.\end{proof}

\begin{remark} Of course, Lemma \ref{speclat}, (\ref{speclat1}), and Lemma \ref{specmaxgu}, (\ref{specmaxgu1}), show that, if $L$ is a bounded lattice with ${\rm Spec}(L)={\rm Max}(L)$ and $L^{\prime }$ is a bounded distributive lattice, then any bounded lattice morphism $m:L\rightarrow L^{\prime }$ is admissible and fulfills GU. See in \cite{euadm} examples of finite lattices whose lattice of congruences is Boolean, thus whose prime congruences coincide to their maximal ones, and which can not be obtained through direct products and/or ordinal sums from modular lattices and relatively complemented lattices.\end{remark}

\begin{proposition}\begin{enumerate}
\item\label{conlant1} If $f$ fulfills LO and $({\rm Spec}(B),\subseteq )$ is a chain, then $f$ fulfills GU.
\item\label{conlant2} If $f$ fulfills LO and ${\rm Con}(B)$ is a chain, then $f$ fulfills GU.\end{enumerate}\label{conlant}\end{proposition}

\begin{proof} (\ref{conlant1}) Let $\phi ,\psi \in {\rm Spec}(A)$ and $\phi _1\in {\rm Spec}(B)$ such that $f^*(\phi _1)=\phi $ and $\phi \subseteq \psi $. If $\phi =\psi $, then we may take $\psi _1=\phi _1$, as in the proof of Lemma \ref{specmaxgu}. Now assume that $\phi \neq \psi $. By Remark \ref{ker}, we have $\psi \supseteq \phi \supseteq {\rm Ker}(f)$, thus, since $f$ fulfills LO, there exists a $\psi _1\in {\rm Spec}(B)$ such that $f^*(\psi _1)=\psi $. Assume by absurdum that $\phi _1\nsubseteq \psi _1$, so that $\psi _1\subset \phi _1$ since $({\rm Spec}(B),\subseteq )$ is totally orderred. Then $\psi =f^*(\psi _1)\subseteq f^*(\phi _1)=\phi $, thus, since $\phi \subseteq \psi $, it follows that $\phi =\psi $, and we have a contradiction. Hence $\phi _1\subseteq \psi _1$, therefore $f$ fulfills GU.

\noindent (\ref{conlant2}) By (\ref{conlant1}).\end{proof}

\begin{proposition}\begin{enumerate}
\item\label{contriv1} If ${\rm Con}(M)=\{\Delta _M,\nabla _M\}$, $M$ is non--trivial and $h^*(\{\nabla _M\})=\{\nabla _N\}$, then $h$ is admissible and fulfills GU.

\item\label{contriv2} If ${\rm Con}(M)=\{\Delta _M,\nabla _M\}$ and ${\cal C}$ is semi--degenerate, then $h$ is admissible and fulfills GU.\end{enumerate}\label{contriv}\end{proposition}

\begin{proof} (\ref{contriv1}) Since $M$ is non--trivial, we have $\Delta _M\neq \nabla _M$, hence ${\rm Con}(M)=\{\Delta _M,\nabla _M\}\cong {\cal L}_2$, thus ${\rm Spec}(M)={\rm Max}(M)=\{\Delta _M\}={\rm Con}(M)\setminus \{\nabla _M\}$. By Lemma \ref{specall}, (\ref{specall1}), and Lemma \ref{specmaxgu}, (\ref{specmaxgu1}), it follows that $h$ is admissible and fulfills GU.

\noindent (\ref{contriv2}) If $M$ is the trivial algebra, then so is $h(M)$, thus so is $N$, because $h(M)$ is a subalgebra of $N$ and ${\cal C}$ is semi--degenerate. In this case, $h$ is an isomorphism, thus $h$ is admissible and fulfills GU. Now assume that $M$ is non--trivial. Then $h$ is admissible and fulfills GU by (\ref{contriv1}) and Lemma \ref{2cls}, (\ref{2cls0}).\end{proof}

\begin{example} ${\rm Con}({\cal L}_2)\cong {\cal L}_2$, because ${\cal L}_2$ is a finite Boolean algebra, thus ${\rm Con}({\cal L}_2)=\{\Delta _{{\cal L}_2},\nabla _{{\cal L}_2}\}$. We have seen in Example \ref{exadm} that ${\rm Con}({\cal D})=\{\Delta _{{\cal D}},\nabla _{{\cal D}}\}$. Therefore, by Proposition \ref{contriv}, (\ref{contriv2}), any bounded lattice morphism whose domain is ${\cal L}_2$ or ${\cal D}$ is admissible and fulfills GU. Many examples of such bounded lattices can be given. See some in \cite{euadm}, including one that is finite and can not be obtained through direct products and/or ordinal sums from modular lattices and relatively complemented lattices.\label{dcontriv}\end{example}

\section{Going Up and Lying Over in Direct Products of Algebras and Ordinal Sums of Bounded Orderred Structures}
\label{prodsum}

In this section, we prove that admissibility, GU and LO are preserved by finite direct products and, in the class of bounded lattices, also by finite ordinal sums; actually, the latter holds in any congruence--modular equational class of bounded orderred structures that fulfills a certain condition on congruences. Throughout this section, ${\cal C}$ shall be congruence--modular, $n\in \N ^*$, $A_1,\ldots ,A_n,B_1,\ldots ,B_n$ shall be algebras from ${\cal C}$, $f_i:A_i\rightarrow B_i$ shall be a morphism in ${\cal C}$ for all $i\in \overline{1,n}$, $\displaystyle A=\prod _{i=1}^nA_i$, $\displaystyle B=\prod _{i=1}^nB_i$ and $\displaystyle f=\prod _{i=1}^nf_i:A\rightarrow B$. We shall also assume that ${\cal C}$ fulfills the equivalent conditions from Proposition \ref{prodcongr}. Recall from Lemma \ref{distribsemid} that this is the case if ${\cal C}$ is semi--degenerate or congruence--distributive.

\begin{remark} Under the assumptions above, every $\beta \in {\rm Con}(B)$ is of the form $\displaystyle \beta =\prod _{i=1}^n\beta _i$ for some $\beta _1\in {\rm Con}(A_1),\ldots ,\beta _n\in {\rm Con}(A_n)$, so that $\displaystyle f^*(\beta )=(\prod _{i=1}^nf_i)^*(\prod _{i=1}^n\beta _i)=\prod _{i=1}^nf_i^*(\beta _i)$. Therefore $\displaystyle {\rm Ker}(f)=f^*(\Delta _B)=f^*(\prod _{i=1}^n\Delta _{B_i})=\prod _{i=1}^nf_i^*(\Delta _{B_i})=\prod _{i=1}^n{\rm Ker}(f_i)$.\label{iminv}\end{remark}

\begin{lemma}{\rm \cite{euadm}} $\displaystyle {\rm Spec}(A)=\bigcup _{i=1}^n\{\phi \times \prod _{j\in \overline{1,n}\setminus \{i\}}\nabla _{A_j}\ |\ \phi \in {\rm Spec}(A_i)\}$.\label{specprodfin}\end{lemma}

\begin{proposition}\begin{enumerate}
\item\label{vprod1} For any $i\in \overline{1,n}$ and any $\theta \in {\rm Con}(A_i)$, $\displaystyle V_A(\theta \times \prod _{j\in \overline{1,n}\setminus \{i\}}\nabla _{A_j})=\{\phi \times \prod _{j\in \overline{1,n}\setminus \{i\}}\nabla _{A_j}\ |\ \phi \in V_{A_i}(\theta )\}$.
\item\label{vprod2} If $\theta _i\in {\rm Con}(A_i)$ for all $i\in \overline{1,n}$, then $\displaystyle V_A(\prod _{i=1}^n\theta _i)=\bigcup _{i=1}^n\{\phi \times \prod _{j\in \overline{1,n}\setminus \{i\}}\nabla _{A_j}\ |\ \phi \in V_{A_i}(\theta _i)\}$.\end{enumerate}\label{vprod}\end{proposition}

\begin{proof} (\ref{vprod1}) Let $\alpha \in {\rm Con}(A)$. By Lemma \ref{specprodfin}, $\displaystyle \alpha \in V_A(\theta \times \prod _{j\in \overline{1,n}\setminus \{i\}}\nabla _{A_j})$ iff $\alpha \in {\rm Spec}(A)$ and $\displaystyle \alpha \supseteq \theta \times \prod _{j\in \overline{1,n}\setminus \{i\}}\nabla _{A_j}$ iff $\displaystyle \alpha =\phi \times \prod _{j\in \overline{1,n}\setminus \{i\}}\nabla _{A_j}$ for some $\phi \in V_{A_i}(\theta )$.

\noindent (\ref{vprod2}) By (\ref{vprod1}), Lemma \ref{specprodfin} and the fact that $\displaystyle \prod _{i=1}^n\theta _i=\bigcap _{i=1}^n(\theta _i\times \prod _{j\in \overline{1,n}\setminus \{i\}}\nabla _{A_j})$.\end{proof}

\begin{corollary}\begin{enumerate}
\item\label{admguloprodfin1} $f$ is admissible iff $f_1,\ldots ,f_n$ are admissible;
\item\label{admguloprodfin2} if $f$ is admissible, then: $f$ fulfills GU iff $f_1,\ldots ,f_n$ fulfill GU;
\item\label{admguloprodfin3} if $f$ is admissible, then: $f$ fulfills LO iff $f_1,\ldots ,f_n$ fulfill LO.\end{enumerate}\label{admguloprodfin}\end{corollary}

\begin{proof} (\ref{admguloprodfin1}) This is a result in \cite{euadm}, which follows immediately from Remark \ref{iminv}, Lemma \ref{specprodfin} and Proposition \ref{vprod}.

\noindent (\ref{admguloprodfin2}) By Lemma \ref{cargulo}, (\ref{cargulo1}), Remark \ref{iminv}, Proposition \ref{vprod}, (\ref{vprod1}), and the fact that $f_i^*(\nabla _{B_i})=\nabla _{A_i}$ for all $i\in \overline{1,n}$, $f$ fulfills GU iff, for all $\psi \in {\rm Spec}(B)$, $V_A(f^*(\psi ))\subseteq f^*(V_B(\psi ))$ iff, for all $i\in \overline{1,n}$ and all $\phi \in {\rm Spec}(B_i)$, $\displaystyle V_A(f^*(\phi \times \prod _{j\in \overline{1,n}\setminus \{i\}}\nabla _{B_j}))\subseteq f^*(V_B(\phi \times \prod _{j\in \overline{1,n}\setminus \{i\}}\nabla _{B_j}))$ iff, for all $i\in \overline{1,n}$ and all $\phi \in {\rm Spec}(B_i)$, $\displaystyle V_{A_i}(f_i^*(\phi ))\times \prod _{j\in \overline{1,n}\setminus \{i\}}\nabla _{A_j}\subseteq f_i^*(V_{B_i}(\phi ))\times \prod _{j\in \overline{1,n}\setminus \{i\}}\nabla _{A_j}$ iff, for all $i\in \overline{1,n}$ and all $\phi \in {\rm Spec}(B_i)$, $V_{A_i}(f_i^*(\phi ))\subseteq f_i^*(V_{B_i}(\phi ))$ iff, for all $i\in \overline{1,n}$, $f_i$ fulfills GU.

\noindent (\ref{admguloprodfin3}) By Lemma \ref{cargulo}, (\ref{cargulo2}), Remark \ref{iminv}, Lemma \ref{specprodfin}, Proposition \ref{vprod}, (\ref{vprod1}) and (\ref{vprod2}), and the fact that $f_i^*(\nabla _{B_i})=\nabla _{A_i}$ for all $i\in \overline{1,n}$, $f$ fulfills LO iff $V_A({\rm Ker}(f))\subseteq f^*({\rm Spec}(B))$ iff $V_A({\rm Ker}(f))=f^*({\rm Spec}(B))$ iff $\displaystyle V_A(\prod _{i=1}^n{\rm Ker}(f_i))=f^*(\bigcup _{i=1}^n\{\phi \times \prod _{j\in \overline{1,n}\setminus \{i\}}\nabla _{B_j}\ |\ \phi \in {\rm Spec}(B_i)\})$ iff $\displaystyle \bigcup _{i=1}^n\{\chi \times \prod _{j\in \overline{1,n}\setminus \{i\}}\nabla _{A_j}\ |\ \chi \in V_{A_i}({\rm Ker}(f_i))\}=\bigcup _{i=1}^n\{f_i^*(\phi )\times \prod _{j\in \overline{1,n}\setminus \{i\}}\nabla _{A_j}\ |$\linebreak $\ \phi \in {\rm Spec}(B_i)\})$ iff, for all $i\in \overline{1,n}$, $\displaystyle \{\chi \times \prod _{j\in \overline{1,n}\setminus \{i\}}\nabla _{A_j}\ |\ \chi \in V_{A_i}({\rm Ker}(f_i))\}=\{f_i^*(\phi )\times \prod _{j\in \overline{1,n}\setminus \{i\}}\nabla _{A_j}\ |\ \phi \in {\rm Spec}(B_i)\})$ iff, for all $i\in \overline{1,n}$, $V_{A_i}({\rm Ker}(f_i))=f_i^*({\rm Spec}(B_i))$ iff, for all $i\in \overline{1,n}$, $f_i$ fulfills LO.\end{proof}

For any bounded lattices $L$ and $M$, we shall denote by $L\oplus M$ the ordinal sum of $L$ with $M$ and, if $\alpha \in {\rm Con}(L)$ and $\beta \in {\rm Con}(M)$, then we denote by $\alpha \oplus \beta =eq((L/\alpha \setminus c/\alpha )\cup (M/\beta \setminus c/\beta )\cup \{c/\alpha \cup c/\beta \})$, where $c$ is the common element of $L$ and $M$ in $L\oplus M$. If $L^{\prime }$ and $M^{\prime }$ are bounded lattices and $h:L\rightarrow M$ and $h^{\prime }:L^{\prime }\rightarrow M^{\prime }$ are bounded lattice morphisms, then we define $h\oplus h^{\prime }:L\oplus L^{\prime }\rightarrow M\oplus M^{\prime }$ by: for all $x\in L\oplus L^{\prime }$, $(h\oplus h^{\prime })(x)=\begin{cases}h(x) & x\in L,\\ h^{\prime }(x) & x\in L^{\prime }.\end{cases}$ Then, clearly, $\alpha \oplus \beta \in {\rm Con}(L\oplus M)$ and $h\oplus h^{\prime }$ is a bounded lattice morphism.

Throughout the rest of this section, for all $i\in \overline{1,n}$, $L_i,M_i$ shall be bounded lattices and $h_i:L_i\rightarrow M_i$ shall be a bounded lattice morphism. We shall denote by $\displaystyle L=\bigoplus _{i=1}^nL_i$, $\displaystyle M=\bigoplus _{i=1}^nM_i$ and $\displaystyle h=\bigoplus _{i=1}^nh_i:L\rightarrow M$.

\begin{remark} Let $\beta _i\in {\rm Con}(M_i)$ for all $i\in \overline{1,n}$ and $\displaystyle \beta =\bigoplus _{i=1}^n\beta _i\in {\rm Con}(M)$. Then $\displaystyle h^*(\beta )=(\bigoplus _{i=1}^nh_i)^*(\bigoplus _{i=1}^n\beta _i)=\bigoplus _{i=1}^nh_i^*(\beta _i)$. Therefore $\displaystyle {\rm Ker}(h)=h^*(\Delta _M)=h^*(\bigoplus _{i=1}^n\Delta _{M_i})=\bigoplus _{i=1}^nh_i^*(\Delta _{M_i})=\bigoplus _{i=1}^n{\rm Ker}(h_i)$.\label{iminvsumord}\end{remark}

\begin{lemma}{\rm \cite{euadm}}\begin{enumerate}
\item\label{sumord1} $\displaystyle {\rm Con}(L)=\{\bigoplus _{i=1}^n\theta _i\ |\ (\forall \, i\in \overline{1,n})\, (\theta _i\in {\rm Con}(L_i))\}\cong \prod _{i=1}^n{\rm Con}(L_i)$.
\item\label{sumord2} $\displaystyle {\rm Spec}(L)=\bigcup _{i=1}^n\{\nabla _{L_1}\oplus \ldots \nabla _{L_{i-1}}\oplus \phi \oplus \nabla _{L_{i+1}}\oplus \ldots \nabla _{L_n}\ |\ \phi \in {\rm Spec}(L_i)\}$.\end{enumerate}\label{sumord}\end{lemma}

\begin{proposition}\begin{enumerate}
\item\label{vsumord1} For all $i\in \overline{1,n}$ and all $\theta \in {\rm Con}(L_i)$, $V_L(\nabla _{L_1}\oplus \ldots \nabla _{L_{i-1}}\oplus \theta \oplus \nabla _{L_{i+1}}\oplus \ldots \nabla _{L_n})=\{\nabla _{L_1}\oplus \ldots \nabla _{L_{i-1}}\oplus \phi \oplus \nabla _{L_{i+1}}\oplus \ldots \nabla _{L_n}\ |\ \phi \in V_{L_i}(\theta )\}$. 
\item\label{vsumord2} If $\theta _i\in {\rm Con}(L_i)$ for all $i\in \overline{1,n}$, $\displaystyle V_L(\bigoplus _{i=1}^n\theta _i)=\bigcup _{i=1}^n\{\nabla _{L_1}\oplus \ldots \nabla _{L_{i-1}}\oplus \phi \oplus \nabla _{L_{i+1}}\oplus \ldots \nabla _{L_n}\ |\ \phi \in V_{L_i}(\theta _i)\}$.\end{enumerate}\label{vsumord}\end{proposition}

\begin{proof} Similar to that of Proposition \ref{vprod}, but using Lemma \ref{sumord} instead of Lemma \ref{specprodfin}.\end{proof}

\begin{corollary}\begin{enumerate}
\item\label{admgulosumord1} $h$ is admissible iff $h_1,\ldots ,h_n$ are admissible;
\item\label{admgulosumord2} if $h$ is admissible, then: $h$ fulfills GU iff $h_1,\ldots ,h_n$ fulfill GU; 
\item\label{admgulosumord3} if $h$ is admissible, then: $h$ fulfills LO iff $h_1,\ldots ,h_n$ fulfill LO.\end{enumerate}\label{admgulosumord}\end{corollary}

\begin{proof} (\ref{admgulosumord1}) is a result in \cite{euadm}, which follows immediately from Remark \ref{iminvsumord}, Lemma \ref{sumord} and Proposition \ref{vsumord}.

\noindent (\ref{admgulosumord2}) and (\ref{admgulosumord3}) follow from the results mentioned above, along with Lemma \ref{cargulo}, through a straightforward proof similar to that of (\ref{admguloprodfin2}) and (\ref{admguloprodfin3}) from Corollary \ref{admguloprodfin}.\end{proof}

\begin{remark} Lemma \ref{sumord}, (\ref{sumord2}), and Proposition \ref{vsumord} hold for any congruence--modular equational class of bounded orderred structures whose finite ordinal sums have the congruences of the form in Lemma \ref{sumord}, (\ref{sumord1}). If in such a class finite ordinal sums of morphisms give morphisms, then that class fulfills Corollary \ref{admgulosumord}, as well.\end{remark}

\section{Characterizations for Properties Going Up and Lying Over}
\label{charguandlo}

In this section, we obtain several characterizations for properties GU and LO, including topological ones, and prove that GU and LO are preserved by quotients. Throughout this section, ${\cal C}$ shall be congruence--modular, $A$ and $B$ shall be algebras from ${\cal C}$ and $f:A\rightarrow B$ shall be a morphism in ${\cal C}$.

For every $\beta \in {\rm Con}(B)$, we define $f_{\beta }:A/f^*({\beta })\rightarrow B/\beta $, for all $a\in A$, $f_{\beta }(a/f^*({\beta }))=f(a)/\beta $.

\begin{remark} For each $\beta \in {\rm Con}(B)$, $f_{\beta }$ is well defined and injective, because, for all $a,b\in A$: $a/f^*({\beta })=b/f^*({\beta })$ iff $(a,b)\in f^*({\beta })$ iff $(f(a),f(b))\in \beta $ iff $f(a)/\beta =f(b)/\beta $ iff $f_{\beta }(a/f^*({\beta }))=f_{\beta }(b/f^*({\beta }))$. Clearly, $f_{\beta }$ is a morphism in ${\cal C}$ and, if $f$ is surjective, then $f_{\beta }$ is surjective. Also, the following diagram is commutative:\vspace*{-18pt}

\begin{center}\begin{picture}(180,60)(0,0)
\put(31,35){$A$}
\put(15,5){$A/f^*(\beta )$}
\put(135,35){$B$}
\put(129,5){$B/\beta $}
\put(9,24){$p_{f^*(\beta )}$}
\put(34,33){\vector(0,-1){19}}
\put(140,24){$p_{\beta }$}
\put(138,33){\vector(0,-1){19}}
\put(82,42){$f$}
\put(39,39){\vector(1,0){96}}
\put(82,12){$f_{\beta }$}
\put(53,9){\vector(1,0){75}}\end{picture}\end{center}\vspace*{-7pt}

For all $\psi \in [\beta )$, $f_{\beta }^*(\psi /\beta )=f^*(\psi )/f^*(\beta )$, because: $f_{\beta }^*(\psi /\beta )=\{(a/f^*(\beta ),b/f^*(\beta ))\ |\ a,b\in A,(f(a)/\beta ,f(b)/\beta )$\linebreak $\in \psi /\beta \}=\{(a/f^*(\beta ),b/f^*(\beta ))\ |\ a,b\in A,(f(a),f(b))\in \psi \}=\{(a/f^*(\beta ),b/f^*(\beta ))\ |\ (a,b)\in f^*(\psi )\}=f^*(\psi )/f^*(\beta )$.\label{fbetainj}\end{remark}

\begin{lemma} $f$ is admissible iff, for each $\beta \in {\rm Con}(B)$, $f_{\beta }$ is admissible.\label{fbadm}\end{lemma}

\begin{proof} For the converse implication, take $\beta =\Delta _B$, so that $p_{\beta }=p_{\Delta _B}:B\rightarrow B/\Delta _B$ is an isomorphism. By Lemma \ref{surjadm} and Remarks \ref{compadm} and \ref{admcompizom}, $p_{f^*(\Delta _B)}=p_{{\rm Ker}(f)}$ is surjective and thus admissible, thus, since $f_{\Delta _B}$ is admissible, $f\circ p_{\Delta _B}=f_{\Delta _B}\circ p_{{\rm Ker}(f)}$ is admissible, hence $f$ is admissible.

Now assume that $f$ is admissible and let $\beta \in {\rm Con}(B)$. By Lemma \ref{speccat}, (\ref{speccat2}), ${\rm Spec}(A/f^*(\beta ))=\{\phi /f^*(\beta )\ |\ \phi \in V_A(f^*(\beta ))\}$, ${\rm Spec}(B/\beta )=\{\chi /\beta \ |\ \chi \in V_B(\beta )\}$, and, by Remark \ref{admv}, $f^*(V_B(\beta ))\subseteq  V_A(f^*(\beta ))$, hence, for all $\chi \in V_B(\beta )$, $f_{\beta }^*(\chi /\beta )=\{(a/f^*(\beta ),b/f^*(\beta ))\ |\ (a,b)\in A^2, (f_{\beta }(a/f^*(\beta )),f_{\beta }(b/f^*(\beta )))\in \chi /\beta \}=\{(a/f^*(\beta ),b/f^*(\beta ))\ |$\linebreak $(a,b)\in A^2, (f(a)/\beta ,f(b)/\beta)\in \chi /\beta \}=\{(a/f^*(\beta ),b/f^*(\beta ))\ |\ (a,b)\in A^2, (f(a),f(b))\in \chi \}=\{(a/f^*(\beta ),$\linebreak $b/f^*(\beta ))\ |\ (a,b)\in f^*(\chi )\}=f^*(\chi )/f^*(\beta )\in {\rm Spec}(A/f^*(\beta ))$, therefore $f_{\beta }^*({\rm Spec}(B/\beta ))\subseteq {\rm Spec}(A/f^*(\beta ))$, that is $f_{\beta }$ is admissible.\end{proof}

\begin{proposition} If $\nabla _B$ is finitely generated and $f$ is admissible, then the following are equivalent:\begin{enumerate}
\item\label{fbeta1} $f$ fulfills GU;
\item\label{fbeta2} for all $\beta \in {\rm Spec}(B)$, the map $f_{\beta }^*\mid _{{\rm Spec}(B/\beta )}:{\rm Spec}(B/\beta )\rightarrow {\rm Spec}(A/f^*(\beta ))$ is surjective;
\item\label{fbeta3} the map $f^*\mid _{{\rm Spec}(B)}:{\rm Spec}(B)\rightarrow {\rm Spec}(A)$ is closed with respect to the Stone topologies;
\item\label{fbeta4} for all $\beta \in {\rm Con}(B)$, $f_{\beta }$ fulfills GU;
\item\label{fbeta5} for all $\beta \in {\rm Con}(B)$, $f_{\beta }$ fulfills LO;
\item\label{fbeta6} for all $\beta \in {\rm Spec}(B)$, $f_{\beta }$ fulfills GU;
\item\label{fbeta7} for all $\beta \in {\rm Spec}(B)$, $f_{\beta }$ fulfills LO.\end{enumerate}

Moreover, (\ref{fbeta1}), (\ref{fbeta2}), (\ref{fbeta4}) and (\ref{fbeta7}) are equivalent even if $\nabla _B$ is not finitely generated.\label{fbeta}\end{proposition}

\begin{proof} Since $f$ is admissible, the map $f^*\mid _{{\rm Spec}(B)}:{\rm Spec}(B)\rightarrow {\rm Spec}(A)$ is well defined and, by Remark \ref{admv}, so is the map $f^*\mid _{V_B(\beta )}:V_B(\beta )\rightarrow V_A(f^*(\beta ))$, for any $\beta \in {\rm Con}(B)$. By Lemma \ref{fbadm}, for all $\beta \in {\rm Con}(B)$, $f_{\beta }$ is admissible, so that the map $f_{\beta }^*\mid _{{\rm Spec}(B/\beta )}:{\rm Spec}(B/\beta )\rightarrow {\rm Spec}(A/f^*(\beta ))$ is well defined.

\noindent (\ref{fbeta1})$\Leftrightarrow $(\ref{fbeta2}): Let $\beta \in {\rm Spec}(B)$. Let $g_A:V_A(f^*(\beta ))\rightarrow {\rm Spec}(A/f^*(\beta ))$ and $g_B:V_B(\beta )\rightarrow {\rm Spec}(B/\beta )$ be the bijections established in Lemma \ref{speccat}, (\ref{speccat1}): for all $\phi \in V_A(f^*(\beta ))$ and all $\psi \in V_B(\beta )$, $g_A(\phi )=\phi /f^*(\beta )$ and $g_B(\psi )=\psi /\beta $. Then  the following diagram is commutative:\vspace*{-15pt}

\begin{center}\begin{picture}(180,60)(0,0)
\put(22,35){$V_B(\beta )$}
\put(16,5){${\rm Spec}(B/\beta )$}
\put(135,35){$V_A(f^*(\beta ))$}
\put(129,5){${\rm Spec}(A/f^*(\beta ))$}
\put(21,22){$g_B$}
\put(34,33){\vector(0,-1){19}}
\put(150,22){$g_A$}
\put(147,33){\vector(0,-1){19}}
\put(73,42){$f^*\mid _{V_B(\beta )}$}
\put(49,38){\vector(1,0){86}}
\put(68,12){$f^*\mid _{{\rm Spec}(B/\beta )}$}

\put(64,8){\vector(1,0){64}}\end{picture}\end{center}\vspace*{-4pt}

Indeed, for all $\chi \in V_B(\beta )$, $g_A(f^*(\chi ))=f^*(\chi )/f^*(\beta )=f^*(\chi /\beta )=f_{\beta }^*(g_B(\chi ))$ by Remark \ref{fbetainj}, thus $f_{\beta }^*\circ g_B=g_A\circ f^*$. Since $g_A$ and $g_B$ are bijections, it follows that: $f^*\mid _{V_B(\beta )}:V_B(\beta )\rightarrow V_A(f^*(\beta ))$ is surjective iff $f^*\mid _{{\rm Spec}(B/\beta )}:{\rm Spec}(B/\beta )\rightarrow {\rm Spec}(A/f^*(\beta ))$ is surjective, that is: $f^*(V_B(\beta ))=V_A(f^*(\beta ))$ iff $f^*({\rm Spec}(B/\beta ))={\rm Spec}(A/f^*(\beta ))$. By Lemma \ref{cargulo}, (\ref{cargulo1}), it follows that: $f$ fulfills GU iff, for all $\beta \in {\rm Spec}(B)$, $f^*(V_B(\beta ))=V_A(f^*(\beta ))$ iff, for all $\beta \in {\rm Spec}(B)$, $f^*({\rm Spec}(B/\beta ))={\rm Spec}(A/f^*(\beta ))$ iff, for all $\beta \in {\rm Spec}(B)$, the map $f_{\beta }^*\mid _{{\rm Spec}(B/\beta )}:{\rm Spec}(B/\beta )\rightarrow {\rm Spec}(A/f^*(\beta ))$ is surjective. 

\noindent (\ref{fbeta2})$\Leftrightarrow $(\ref{fbeta4}): Let $\beta \in {\rm Con}(B)$ and $\psi \in V_B(\beta )$, arbitrary, so that $\psi /\beta \in {\rm Spec}(B/\beta )$, arbitrary. Then $(f_{\beta })_{(\psi /\beta )}:(A/f^*(\beta ))/f_{\beta }^*(\psi /\beta )\rightarrow (B/\beta )/(\psi /\beta )$ is defined by: for all $a\in A$, $(f_{\beta })_{(\psi /\beta )}((a/f^*(\beta ))/f_{\beta }^*(\psi /\beta ))=f_{\beta }(a/f^*(\beta ))/$\linebreak $f_{\beta }^*(\psi /\beta )=(f(a)/\beta )/(\psi /\beta )$. By Remark \ref{fbetainj} and Lemma \ref{fbadm}, $(f_{\beta })_{(\psi /\beta )}$ is a well--defined admissible injective morphism in ${\cal C}$. By Remark \ref{fbetainj}, $f_{\beta }^*(\psi /\beta )=f^*(\psi )/f^*(\beta )$. Let $g:A/f^*(\psi )\rightarrow (A/f^*(\beta ))/f_{\beta }^*(\psi /\beta )=(A/f^*(\beta ))/(f^*(\psi )/f^*(\beta ))$ and $h:B/\psi \rightarrow (B/\beta )/(\psi /\beta )$ be the isomorphisms given by the Second Isomorphism Theorem: for all $a\in A$, $g(a/f^*(\psi ))=(a/f^*(\beta ))/(f^*(\psi )/f^*(\beta ))=(a/f^*(\beta ))/f_{\beta }^*(\psi /\beta )$, and, for all $b\in B$, $h(b/\psi )=(b/\beta )/(\psi /\beta )$. Then, clearly, $g^*:{\rm Spec}((A/f^*(\beta )/f_{\beta }^*(\psi /\beta ))\rightarrow {\rm Spec}(A/f^*(\psi ))$ and $h^*:{\rm Spec}((B/\beta )/(\psi /\beta ))\rightarrow {\rm Spec}(B/\psi )$ are bijections, and the following diagram is commutative:\vspace*{-17pt}

\begin{center}\begin{picture}(180,60)(0,0)
\put(18,35){$A/f^*(\psi )$}
\put(-10,5){$(A/f^*(\beta ))/f_{\beta }^*(\psi /\beta )$}
\put(135,35){$B/\psi $}
\put(129,5){$(B/\beta )/(\psi /\beta )$}
\put(28,22){$g$}
\put(34,33){\vector(0,-1){19}}
\put(148,22){$h$}
\put(147,33){\vector(0,-1){19}}
\put(90,42){$f_{\psi }$}
\put(56,38){\vector(1,0){79}}
\put(82,12){$(f_{\beta })_{(\psi /\beta )}$}
\put(78,8){\vector(1,0){50}}\end{picture}\end{center}\vspace*{-6pt}

From this, it follows that the following diagram is commutative, that is $g^*\circ (f_{\beta })_{(\psi /\beta )}^*=f_{\psi }^*\circ h^*$, therefore, since $g^*$ and $h^*$ are bijective: $f_{\psi }^*\mid _{{\rm Spec}(B/\psi )}:{\rm Spec}(B/\psi )\rightarrow {\rm Spec}(A/f^*(\psi ))$ is surjective iff $(f_{\beta })_{(\psi /\beta )}^*\mid _{{\rm Spec}((B/\beta )/(\psi /\beta ))}:{\rm Spec}((B/\beta )/(\psi /\beta ))\rightarrow {\rm Spec}((A/f^*(\beta ))/f_{\beta }^*(\psi /\beta ))$ is surjective.\vspace*{-18pt}

\begin{center}\begin{picture}(180,60)(0,0)
\put(18,35){${\rm Spec}(B/\psi )$}
\put(-12,5){${\rm Spec}((B/\beta )/(\psi /\beta ))$} 
\put(135,35){${\rm Spec}(A/f^*(\psi ))$}
\put(127,5){${\rm Spec}((A/f^*(\beta ))/f_{\beta }^*(\psi /\beta ))$}
\put(24,20){$h^*$}
\put(34,14){\vector(0,1){19}}
\put(162,20){$g^*$}
\put(160,14){\vector(0,1){19}}
\put(94,42){$f_{\psi }^*$}
\put(67,38){\vector(1,0){67}}
\put(79,12){$(f_{\beta })_{(\psi /\beta )}^*$}
\put(75,7){\vector(1,0){50}}\end{picture}\end{center}\vspace*{-4pt}

Since $V_B(\Delta _B)={\rm Spec}(B)$, by the equivalence (\ref{fbeta1})$\Leftrightarrow $(\ref{fbeta2}) proven above it follows that: $f$ fulfills GU iff, for all $\beta \in {\rm Con}(B)$ and all $\psi \in V_B(\beta )$, $f_{\psi }^*:{\rm Spec}(B/\psi )\rightarrow {\rm Spec}(A/f^*(\psi ))$ is surjective, iff, for all $\beta \in {\rm Con}(B)$ and all $\psi \in V_B(\beta )$, $(f_{\beta })_{(\psi /\beta )}^*:{\rm Spec}((B/\beta )/(\psi /\beta ))\rightarrow {\rm Spec}((A/f^*(\beta ))/f_{\beta }^*(\psi /\beta ))$ is surjective, iff, for all $\beta \in {\rm Con}(B)$, $f_{\beta }$ fulfills GU.

\noindent (\ref{fbeta4})$\Rightarrow $(\ref{fbeta5})$\Rightarrow $(\ref{fbeta7}) and (\ref{fbeta4})$\Rightarrow $(\ref{fbeta6})$\Rightarrow $(\ref{fbeta7}): By Proposition \ref{gulo}.

\noindent (\ref{fbeta2})$\Leftrightarrow $(\ref{fbeta7}): By Proposition \ref{carloinj}, (\ref{carloinj2}), and Remark \ref{fbetainj}.

\noindent (\ref{fbeta3})$\Rightarrow $(\ref{fbeta1}): Let $\beta \in {\rm Spec}(B)$, so that $f^*(\beta )\in {\rm Spec}(A)$. Then $\overline{\{f^*(\beta )\}}=V_A(f^*(\beta ))$ and $\beta \in V_B(\beta )$, so $f^*(\beta )\in f^*(V_B(\beta ))$, that is $\{f^*(\beta )\}\subseteq f^*(V_B(\beta ))$, which is a closed set in ${\rm Spec}(A)$ with respect to the Stone topology, since $V_B(\beta )$ is closed in ${\rm Spec}(B)$ and $f^*\mid _{{\rm Spec}(B)}:{\rm Spec}(B)\rightarrow {\rm Spec}(A)$ is a closed function. Hence $V_A(f^*(\beta ))=\overline{\{f^*(\beta )\}}\subseteq f^*(V_B(\beta ))$. Therefore $f$ fulfills GU by Lemma \ref{cargulo}, (\ref{cargulo1}). 

\noindent (\ref{fbeta1})$\Rightarrow $(\ref{fbeta3}): Let $\beta \in {\rm Con}(B)$, so that $V_B(\beta )$ is an arbitrary closed set in ${\rm Spec}(B)$ with the Stone topology. The equivalence (\ref{fbeta1})$\Leftrightarrow $(\ref{fbeta4}) follows from the above, so, since $f$ fulfills GU, $f_{\beta }$ fulfills GU, hence $f_{\beta }$ fulfills LO by Proposition \ref{gulo}. By Remark \ref{fbetainj}, $f_{\beta }$ is injective. By Proposition \ref{carloinj}, (\ref{carloinj2}), and again Remark \ref{fbetainj}, it follows that $f_{\beta }^*({\rm Spec}(B/\beta ))={\rm Spec}(A/f^*(\beta ))$, that is $\{f^*(\psi )/f^*(\beta )\ |\ \psi \in V_B(\beta )\}=\{f_{\beta }^*(\psi /\beta )\ |\ \psi \in V_B(\beta )\}=\{\phi /f^*(\beta )\ |\ \phi \in V_A(f^*(\beta ))\}$, hence $f^*(V_B(\beta ))=\{f^*(\psi )\ |\ \psi \in V_B(\beta )\}=V_A(f^*(\beta ))$, which is a closed set in ${\rm Spec}(A)$. Therefore the map $f^*:{\rm Spec}(B)\rightarrow {\rm Spec}(A)$ is closed with respect to the Stone topologies.\end{proof}

Let us define $\varphi _f:{\rm Con}(B)\rightarrow {\rm Con}(A/{\rm Ker}(f))$, for all $\beta \in {\rm Con}(B)$, $\varphi _f(\beta )=f^*(\beta )/{\rm Ker}(f)$. Here is a generalization of Proposition \ref{carloinj}:

\begin{lemma} If $f$ is admissible, then the restriction $\varphi _f\mid _{{\rm Spec}(B)}:{\rm Spec}(B)\rightarrow {\rm Spec}(A/{\rm Ker}(f))$ is well defined and the following are equivalent:\begin{enumerate}
\item\label{6.4lo1} $f$ fulfills LO;
\item\label{6.4lo2} the map $\varphi _f\mid _{{\rm Spec}(B)}:{\rm Spec}(B)\rightarrow {\rm Spec}(A/{\rm Ker}(f))$ is surjective.\end{enumerate}\label{6.4lo}\end{lemma}

\begin{proof} By Remark \ref{admv} and Lemma \ref{speccat}, (\ref{speccat1}), $f^*({\rm Spec}(B))\subseteq V_A({\rm Ker}(f))$, hence $\varphi _f({\rm Spec}(B))\subseteq {\rm Spec}(A/{\rm Ker}(f))$, so the restriction $\varphi _f\mid _{{\rm Spec}(B)}:{\rm Spec}(B)\rightarrow {\rm Spec}(A/{\rm Ker}(f))$ is well defined. By Lemma \ref{cargulo}, (\ref{cargulo2}), and again Lemma \ref{speccat}, (\ref{speccat1}), $f$ fulfills LO iff $f^*({\rm Spec}(B))=V_A({\rm Ker}(f))$ iff $\varphi _f({\rm Spec}(B))={\rm Spec}(A/{\rm Ker}(f))$ iff the map $\varphi _f\mid _{{\rm Spec}(B)}:{\rm Spec}(B)\rightarrow {\rm Spec}(A/{\rm Ker}(f))$ is surjective.\end{proof}

For every $\theta \in {\rm Con}(A)$, we shall denote by $\displaystyle \rho (\theta )=\bigcap _{\phi \in V_A(\theta )}\phi $, that is the intersection of the prime congruences of $A$ which include $\theta $; $\rho (\theta )$ is called the {\em radical} of $\theta $. Clearly, if $\theta \in {\rm Spec}(A)$, then $\rho (\theta )=\theta $. Actually, $\rho (\theta )=\theta $ iff $\theta $ is an intersection of prime congruences.

\begin{lemma} If $f$ is admissible, $\nabla _B$ is finitely generated and $\phi \in {\rm Spec}(A)$, then the following are equivalent:\begin{enumerate}
\item\label{lcharlo1} there exists a $\psi \in {\rm Spec}(B)$ such that $f^*(\psi )=\phi $;
\item\label{lcharlo2} $f^*(Cg_B(f(\phi )))=\phi $.\end{enumerate}\label{lcharlo}\end{lemma}

\begin{proof} (\ref{lcharlo1})$\Rightarrow $(\ref{lcharlo2}): Since $\phi =f^*(\psi )$, it follows that $f(\phi )=f(f^*(\psi ))=\psi \cap f(A^2)\subseteq \psi $, hence $Cg_B(f(\phi ))\subseteq \psi $, thus $f^*(Cg_B(f(\phi )))\subseteq f^*(\psi )=\phi $. We also have $\phi \subseteq f^*(f(\phi ))\subseteq f^*(Cg_B(f(\phi )))$. Therefore $f^*(Cg_B(f(\phi )))=\phi $.

\noindent (\ref{lcharlo2})$\Rightarrow $(\ref{lcharlo1}): By Lemma \ref{1.6} and Lemma \ref{fmsist}, (\ref{fmsist1}), $\nabla _A\setminus \phi $ is an m--system in $A$, hence $f(\nabla _A\setminus \phi )=f(A^2\setminus \phi )$ is an m--system in $B$. Let us notice that $f(A^2\setminus \phi )\cap Cg_B(f(\phi ))=\emptyset $. Indeed, assume by absurdum that there exists a $(u,v)\in f(A^2\setminus \phi )\cap Cg_B(f(\phi ))$, that is there exists an $(x,y)\in A^2\setminus \phi $ such that $(f(x),f(y))\in Cg_B(f(\phi ))$, which means that $(x,y)\in (A^2\setminus \phi )\cap f^*(Cg_B(f(\phi )))=(A^2\setminus \phi )\cap \phi =\emptyset $; we have a contradiction. Now let $\psi \in {\rm Con}(B)$ be a maximal element of the set of congruences $\theta $ of $B$ which fulfill $Cg_B(f(\phi ))\subseteq \theta $ and $\theta \cap f(A^2\setminus \phi )=\emptyset $. Then $\psi \in {\rm Spec}(B)$ by Lemma \ref{1.7}. Let us prove that $f^*(\psi )=\theta $. Since $Cg_B(f(\phi ))\subseteq \psi $, we have: $\phi =f^*(Cg_B(f(\phi )))\subseteq f^*(\psi )$. Now let $(x,y)\in f^*(\psi )$, so that $(f(x),f(y))\in \psi $. Since $\psi \cap f(A^2\setminus \phi )=\emptyset $, it follows that $(f(x),f(y))\notin f(A^2\setminus \phi )$, thus $(x,y)\notin A^2\setminus \phi $, which means that $(x,y)\in \phi $. Hence we also have $f^*(\psi )\subseteq \phi $, therefore $f^*(\psi )=\phi $.\end{proof}

\begin{proposition} If $f$ is admissible and $\nabla _B$ is finitely generated, then the following are equivalent:\begin{enumerate}
\item\label{propcharlo1} $f$ fulfills LO;
\item\label{propcharlo2} for all $\phi \in {\rm Spec}(A)$ such that ${\rm Ker}(f)\subseteq \phi $, $f^*(Cg_B(f(\phi )))=\phi $;
\item\label{propcharlo3} for all $\theta \in {\rm Con}(A)$ such that ${\rm Ker}(f)\subseteq \theta $, $f^*(\rho (Cg_B(f(\theta ))))=\rho (\theta )$.\end{enumerate}\label{propcharlo}\end{proposition}

\begin{proof} (\ref{propcharlo1})$\Leftrightarrow $(\ref{propcharlo2}): By Lemma \ref{lcharlo} and the definition of LO. 

\noindent (\ref{propcharlo2})$\Rightarrow $(\ref{propcharlo3}): Let $\theta \in {\rm Con}(A)$ such that ${\rm Ker}(f)\subseteq \theta $. We have: $\displaystyle f^*(\rho (Cg_B(f(\theta ))))=f^*(\bigcap _{\beta \in V_B(Cg_B(f(\theta )))}\beta )=\bigcap _{\beta \in V_B(Cg_B(f(\theta )))}f^*(\beta )$. Let $\beta \in V_B(Cg_B(f(\theta )))$. Then $\beta \in {\rm Spec}(B)$, so $f^*(\beta )\in {\rm Spec}(A)$, and $f(\theta )\subseteq \beta $, thus $\theta \subseteq f^*(f(\theta ))\subseteq f^*(\beta )$, so $\rho (\theta )\subseteq f^*(\beta )$. Hence $\displaystyle \rho (\theta )\subseteq \bigcap _{\beta \in V_B(Cg_B(f(\theta )))}f^*(\beta )=f^*(\rho (Cg_B(f(\theta ))))$. By the equivalence (\ref{propcharlo1})$\Leftrightarrow $(\ref{propcharlo2}) proven above, since ${\rm Ker}(f)\subseteq \theta $, we have: $\displaystyle \rho (\theta )=\bigcap _{\alpha \in V_A(\theta )}\alpha =\bigcap _{\alpha \in V_A(\theta )}f^*(Cg_B(f(\alpha )))$. Let $\displaystyle (a,b)\in f^*(\rho (Cg_B(f(\theta ))))=\bigcap _{\beta \in V_B(Cg_B(f(\theta )))}f^*(\beta )$, so that, for all $\beta \in V_B(Cg_B(f(\theta )))$, $(a,b)\in f^*(\beta )$. Since $f$ fulfills LO, for every $\alpha \in {\rm Spec}(A)$ such that $\alpha \supseteq \theta \supseteq {\rm Ker}(f)$, there exists a $\beta \in {\rm Spec}(B)$ such that $f^*(\beta )=\alpha $, so that $f(\theta )\subseteq f(\alpha )=f(f^*(\beta ))=\beta \cap f(A^2)\subseteq \beta $. Hence $(a,b)\in f^*(\beta )=\alpha $, thus $(f(a),f(b))\in f(f^*(\beta ))=f(\alpha )\subseteq Cg_B(f(\alpha ))$, so $(a,b)\in f^*(Cg_B(f(\alpha )))$. Therefore $\displaystyle (a,b)\in \bigcap _{\alpha \in V_A(\theta )}f^*(Cg_B(f(\alpha )))=\rho (\theta )$, hence $f^*(\rho (Cg_B(f(\theta ))))\subseteq \rho (\theta )$. Therefore $f^*(\rho (Cg_B(f(\theta ))))=\rho (\theta )$.

\noindent (\ref{propcharlo3})$\Rightarrow $(\ref{propcharlo2}): Let $\phi \in {\rm Spec}(A)$ such that ${\rm Ker}(f)\subseteq \phi $. Then $\rho (\phi )=\phi $, so we have $f^*(\rho (Cg_B(f(\phi ))))=\phi $. Since $Cg_B(f(\phi ))\subseteq \rho (Cg_B(f(\phi )))$ it follows that $f^*(Cg_B(f(\phi )))\subseteq f^*(\rho (Cg_B(f(\phi ))))=\phi $. But $f(\phi )\subseteq Cg_B(f(\phi ))$, thus $\phi \subseteq f^*(f(\phi ))\subseteq f^*(Cg_B(f(\phi )))$. Therefore $f^*(Cg_B(f(\phi )))=\phi $, hence $f$ fulfills LO by the equivalence (\ref{propcharlo1})$\Leftrightarrow $(\ref{propcharlo2}) proven above.\end{proof}

For any $\theta \in {\rm Con}(A)$, we shall denote by $f_{[\theta ]}:A/\theta \rightarrow B/Cg_B(f(\theta ))$, for all $a\in A$, $f_{[\theta ]}(a/\theta )=f(a)/Cg_B(f(\theta ))$.

\begin{remark} For any $\theta \in {\rm Con}(A)$ and any $a,b\in A$, if $a/\theta =b/\theta $, which means that $(a,b)\in \theta $, then $(f(a),f(b))\in f(\theta )\subseteq Cg_B(f(\theta ))$, thus $f(a)/Cg_B(f(\theta ))=f(b)/Cg_B(f(\theta ))$, so $f_{[\theta ]}$ is well defined. Clearly, $f_{[\theta ]}$ is a morphism and the following diagram is commutative:\vspace*{-22pt}

\begin{center}\begin{picture}(180,60)(0,0)
\put(31,35){$A$}
\put(24,5){$A/\theta $}
\put(143,35){$B$}
\put(122,5){$B/Cg_B(f(\theta ))$}
\put(23,22){$p_{\theta }$}
\put(34,33){\vector(0,-1){19}}
\put(150,22){$p_{Cg_B(f(\theta ))}$}
\put(147,33){\vector(0,-1){19}}
\put(86,42){$f$}
\put(40,38){\vector(1,0){100}}
\put(73,12){$f_{[\theta ]}$}
\put(43,8){\vector(1,0){77}}\end{picture}\end{center}\vspace*{-7pt}\label{fteta}

Clearly, if $f^*(Cg_B(f(\theta )))=\theta $, then $f_{[\theta ]}=f_{Cg_B(f(\theta ))}$ (see the definition of $f_{\beta }$ at the beginning of this section), so, by Proposition \ref{propcharlo}, if $f$ fulfills LO and $\nabla _B$ is finitely generated, then, for every $\phi \in V_A({\rm Ker}(f))$, $f_{[\phi ]}=f_{Cg_B(f(\phi ))}$. Since ${\rm Ker}(f)=\Delta _A$ if $f$ is injective, we obtain: if $f$ is injective and fulfills LO and $\nabla _B$ is finitely generated, then, for every $\phi \in {\rm Spec}(A)$, $f_{[\phi ]}=f_{Cg_B(f(\phi ))}$.\label{fthetafbeta}\end{remark}

\begin{lemma} Let $\theta \in {\rm Con}(A)$ and $\lambda \in {\rm Con}(B)$. Then:\begin{itemize}
\item $\theta \subseteq f^*(\lambda )$ iff $Cg_B(f(\theta ))\subseteq \lambda $; \item if $\theta \subseteq f^*(\lambda )$, then $f_{[\theta ]}^*(\lambda /Cg_B(f(\theta )))=f^*(\lambda )/\theta $.\end{itemize}\label{f(theta)}\end{lemma}

\begin{proof} If $\theta \subseteq f^*(\lambda )$, then $f(\theta )\subseteq f(f^*(\lambda ))\subseteq \lambda $, hence $Cg_B(f(\theta ))\subseteq \lambda $. If $Cg_B(f(\theta ))\subseteq \lambda $, then $f(\theta )\subseteq \lambda $, thus $\theta \subseteq f^*(f(\theta ))\subseteq f^*(\lambda )$.

Now assume that $\theta \subseteq f^*(\lambda )$, so that $Cg_B(f(\theta ))\subseteq \lambda $ by the above. Then, for every $a,b\in A$, the following equivalences hold: $(a/\theta ,b/\theta )\in f_{[\theta ]}^*(\lambda /Cg_B(f(\theta )))$ iff $(f_{[\theta ]}(a/\theta ),f_{[\theta ]}(b/\theta ))\in \lambda /Cg_B(f(\theta ))$ iff $(f(a)/Cg_B(f(\theta )),$\linebreak $f(b)/Cg_B(f(\theta )))\in \lambda /Cg_B(f(\theta ))$ iff $(f(a),f(b))\in \lambda $ iff $(a,b)\in f^*(\lambda )$ iff $(a/\theta ,b/\theta )\in f^*(\lambda )/\theta $. Hence the equality in the enunciation.\end{proof}

\begin{lemma} The following are equivalent:\begin{enumerate}
\item\label{f(theta)adm1} $f$ is admissible;
\item\label{f(theta)adm2} for any $\theta \in {\rm Con}(A)$, $f_{[\theta ]}$ is admissible;
\item\label{f(theta)adm3} $f_{[{\rm Ker}(f)]}$ is admissible.\end{enumerate}\label{f(theta)adm}\end{lemma}

\begin{proof} (\ref{f(theta)adm1})$\Leftrightarrow $(\ref{f(theta)adm2}) is a result in \cite{euadm}, but we provide a proof for it, for the sake of completeness.

\noindent (\ref{f(theta)adm2})$\Rightarrow $(\ref{f(theta)adm1}): Take $\theta =\Delta _A$, so that $f(\theta )=f(\Delta _A)\subseteq \Delta _B$, thus $Cg_B(f(\theta ))=\Delta _B$, and so $p_{\theta }=p_{\Delta _A}:A\rightarrow A/\Delta _A$ and $p_{Cg_B(f(\theta ))}=p_{\Delta _B}:B\rightarrow B/\Delta _B$ are isomorphisms, and $p_{\Delta _B}\circ f=f_{[\Delta _A]}\circ p_{\Delta _A}$. By Remark \ref{admcompizom}, since $f_{[\Delta _A]}$ is admissible, it follows that $f$ is admissible.

\noindent (\ref{f(theta)adm1})$\Rightarrow $(\ref{f(theta)adm2}): Let $\theta \in {\rm Con}(A)$ and $\phi \in {\rm Spec}(B/Cg_B(f(\theta )))$, so that $\phi =\psi /Cg_B(f(\theta ))$ for some $\psi \in V_B(Cg_B(f(\theta )))$. Then $\psi \in {\rm Spec}(B)$ and $Cg_B(f(\theta ))\subseteq \psi $, thus $\theta \subseteq f^*(\psi )$ by Lemma \ref{f(theta)}, thus $f^*(\psi )\in {\rm Spec}(A)\cap [\theta )=V_A(\theta )$ since $f$ is admissible. Then, by Lemmas \ref{f(theta)} and \ref{speccat}, (\ref{speccat1}), $f_{[\theta ]}^*(\phi )=f_{[\theta ]}^*(\psi /Cg_B(f(\theta )))=f^*(\psi )/\theta \in {\rm Spec}(A/\theta )$. Thus $f_{[\theta ]}$ is admissible.

\noindent (\ref{f(theta)adm2})$\Rightarrow $(\ref{f(theta)adm3}): Trivial.

\noindent (\ref{f(theta)adm3})$\Rightarrow $(\ref{f(theta)adm1}): $f({\rm Ker}(f))=f(f^*(\Delta _B))\subseteq \Delta _B$, so $Cg_B(f({\rm Ker}(f)))=\Delta _B$, thus $p_{Cg_B(f({\rm Ker}(f)))}=p_{\Delta _B}:B\rightarrow B/\Delta _B$ is an isomorphism. By Lemma \ref{surjadm} and Remarks \ref{compadm} and \ref{admcompizom}, $p_{{\rm Ker}(f)}$ is surjective and thus admissible, so $f_{[{\rm Ker}(f)]}\circ p_{{\rm Ker}(f)}=p_{\Delta _B}\circ f$ is admissible, hence $f$ is admissible.\end{proof}

\begin{proposition} If $f$ is admissible, then the following are equivalent:\begin{enumerate}
\item\label{f(theta)gu1} $f$ fulfills GU;
\item\label{f(theta)gu2} for all $\theta \in {\rm Con}(A)$, $f_{[\theta ]}$ fulfills GU;
\item\label{f(theta)gu3} for all $\theta \in {\rm Spec}(A)$, $f_{[\theta ]}$ fulfills GU;
\item\label{f(theta)gu4} for all $\theta \in {\rm Min}(A)$, $f_{[\theta ]}$ fulfills GU;
\item\label{f(theta)gu5} $f_{[{\rm Ker}(f)]}$ fulfills GU.\end{enumerate}\label{f(theta)gu}\end{proposition}

\begin{proof} By Lemma \ref{f(theta)adm}, for all $\theta \in {\rm Con}(A)$, $f_{[\theta ]}$ is admissible.

\noindent (\ref{f(theta)gu2})$\Rightarrow $(\ref{f(theta)gu1}): As in the proof of Lemma \ref{f(theta)adm}, take $\theta =\Delta _A$ and obtain that $p_{\Delta _B}\circ f=f_{[\Delta _A]}\circ p_{\Delta _A}$, where $p_{\Delta _A}$ and $p_{\Delta _B}$ are isomorphisms. Now, from Remark \ref{compizom}, since $f_{[\Delta _A]}$ fulfills GU, it follows that $f$ fulfills GU.

\noindent (\ref{f(theta)gu1})$\Rightarrow $(\ref{f(theta)gu2}): Let $\theta \in {\rm Con}(A)$, $\phi ,\psi \in {\rm Spec}(A/\theta )$ and $\phi _1\in {\rm Spec}(B/Cg_B(f(\theta )))$ such that $\phi \subseteq \psi $ and $f_{[\theta ]}^*(\phi _1)=\phi $. Then, by Lemma \ref{speccat}, (\ref{speccat1}), we have $\phi =\alpha /\theta $, $\psi =\beta /\theta $ and $\phi _1=\alpha _1/Cg_B(f(\theta ))$ for some $\alpha ,\beta \in V_A(\theta )$ such that $\alpha \subseteq \beta $ and some $\alpha _1\in V_B(Cg_B(f(\theta )))$. Then $Cg_B(f(\theta ))\subseteq \alpha _1$, so, by Lemma \ref{f(theta)}, $\alpha /\theta =\phi =f_{[\theta ]}^*(\phi _1)=f_{[\theta ]}^*(\alpha _1/Cg_B(f(\theta )))=f^*(\alpha _1)/\theta $, hence $\alpha =f^*(\alpha _1)$, therefore, since $f$ fulfills GU, there exists a $\beta _1\in {\rm Spec}(B)$ such that $f^*(\beta _1)=\beta $ and $\beta _1\supseteq \alpha _1\supseteq Cg_B(f(\theta ))$. Denote $\psi _1=\beta _1/Cg_B(f(\theta ))\in {\rm Spec}(B/Cg_B(f(\theta )))$ by Lemma \ref{speccat}, (\ref{speccat1}). Then $\phi _1\subseteq \psi _1$ and, again by Lemma \ref{f(theta)}, $f_{[\theta ]}^*(\psi _1)=f_{[\theta ]}^*(\beta _1/Cg_B(f(\theta )))=f^*(\beta _1)/\theta =\beta /\theta =\psi $. Therefore $f_{[\theta ]}$ fulfills GU.

\noindent (\ref{f(theta)gu2})$\Rightarrow $(\ref{f(theta)gu3})$\Rightarrow $(\ref{f(theta)gu4}) and (\ref{f(theta)gu2})$\Rightarrow $(\ref{f(theta)gu5}): Trivial. 

\noindent (\ref{f(theta)gu4})$\Rightarrow $(\ref{f(theta)gu1}): Let $\alpha ,\beta \in {\rm Spec}(A)$ and $\alpha _1\in {\rm Spec}(B)$ such that $\alpha \subseteq \beta $ and $f^*(\alpha _1)=\alpha $. It is easy to derive from Zorn`s Lemma that there exists a $\theta \in {\rm Min}(A)$ such that $\theta \subseteq \alpha \subseteq \beta $, so that $\alpha /\theta ,\beta /\theta \in {\rm Spec}(A/\theta )$ and $\alpha /\theta \subseteq \beta /\theta $ by Lemma \ref{speccat}, (\ref{speccat1}). We have: $f(\theta )\subseteq f(\alpha )=f(f^*(\alpha _1))=\alpha _1\cap f(A^2)\subseteq \alpha _1$, thus $Cg_B(f(\theta ))\subseteq \alpha _1$, hence $\alpha _1/Cg_B(f(\theta ))\in {\rm Spec}(B/Cg_B(f(\theta )))$ by Lemma \ref{speccat}, (\ref{speccat1}), and, by Lemma \ref{f(theta)}, $f_{[\theta ]}^*(\alpha _1/Cg_B(f(\theta )))=f^*(\alpha _1)/\theta =\alpha /\theta $. Since $f_{[\theta ]}$ fulfills GU, it follows that there exists a $\psi _1\in {\rm Spec}(B/Cg_B(f(\theta )))$ such that $\alpha _1/Cg_B(f(\theta ))\subseteq \psi _1$ and $f_{[\theta ]}^*(\psi _1)=\beta /\theta $. Again by Lemma \ref{speccat}, (\ref{speccat1}), $\psi _1=\beta _1/Cg_B(f(\theta ))$ for some $\beta _1\in V_B(Cg_B(f(\theta )))$, so that $\alpha  _1/Cg_B(f(\theta ))\subseteq \beta _1/Cg_B(f(\theta ))$, hence $\alpha _1\subseteq \beta _1$. And, again by Lemma \ref{f(theta)}, $f^*(\beta _1)/\theta =f_{[\theta ]}^*(\beta _1/Cg_B(f(\theta )))=f_{[\theta ]}^*(\psi _1)=\beta /\theta $, thus $f^*(\beta _1)=\beta $. Therefore $f$ fulfills GU.

\noindent (\ref{f(theta)gu5})$\Rightarrow $(\ref{f(theta)gu1}): By the proof of Lemma \ref{f(theta)adm}, $f_{[{\rm Ker}(f)]}\circ p_{{\rm Ker}(f)}=p_{\Delta _B}\circ f$, with $p_{\Delta _B}$ an isomorphism and $p_{{\rm Ker}(f)}$ surjective, thus admissible and with GU by Proposition \ref{surjgulo}. Since $f_{[{\rm Ker}(f)]}$ fulfills GU, by Lemma \ref{gulocomp}, (\ref{gulocomp1}), and Remark \ref{compizom} it follows that $f_{[{\rm Ker}(f)]}\circ p_{{\rm Ker}(f)}$ fulfills GU, thus $f$ fulfills GU.\end{proof}

\begin{proposition} If $f$ is admissible, then the following are equivalent:\begin{enumerate}
\item\label{charlo1} $f$ fulfills LO;
\item\label{charlo2} for all $\theta \in {\rm Con}(A)$, $f_{[\theta ]}$ fulfills LO;
\item\label{charlo3} $f_{[{\rm Ker}(f)]}$ fulfills LO.\end{enumerate}\label{charlo}\end{proposition}

\begin{proof} (\ref{charlo2})$\Rightarrow $(\ref{charlo1}): As in the proof of Proposition \ref{f(theta)gu}, take $\theta =\Delta _A$ and apply Remark \ref{compizom}.

\noindent (\ref{charlo1})$\Rightarrow $(\ref{charlo2}): Let $\theta \in {\rm Con}(A)$ and $\phi \in {\rm Spec}(A/\theta )$ such that ${\rm Ker}(f_{[\theta ]})\subseteq \phi $. By Lemma \ref{speccat}, (\ref{speccat1}), and Lemma \ref{f(theta)}, $\phi =\alpha /\theta $ for some $\alpha \in V_A(\theta )$, and $\alpha /\theta =\phi \supseteq {\rm Ker}(f_{[\theta ]})=f_{[\theta ]}^*(\Delta _{B/Cg_B(f(\theta ))})=f_{[\theta ]}^*(Cg_B(f(\theta ))/Cg_B(f(\theta )))=f^*(Cg_B(f(\theta )))/\theta $, so that $\alpha =f^*(Cg_B(f(\theta )))\supseteq {\rm Ker}(f)$ by Remark \ref{ker}. Since $f$ fulfills LO, it follows that $\alpha =f^*(\beta )$ for some $\beta \in {\rm Spec}(B)$, so that, by Lemmas \ref{f(theta)} and \ref{speccat}, (\ref{speccat1}), $\theta \subseteq \alpha =f^*(\beta )$, hence $Cg_B(f(\theta ))\subseteq \beta $, thus $\beta \in V_B(Cg_B(f(\theta )))$, hence $\beta /Cg_B(f(\theta ))\in {\rm Spec}(B/Cg_B(f(\theta )))$, and $\phi =\alpha /\theta =f^*(\beta )/\theta =f_{[\theta ]}^*(\beta /Cg_B(f(\theta )))$. Therefore $f_{[\theta ]}$ fulfills LO.

\noindent (\ref{charlo2})$\Rightarrow $(\ref{charlo3}): Trivial.

\noindent (\ref{charlo3})$\Rightarrow $(\ref{charlo1}): By the proof of Lemma \ref{f(theta)adm}, $f_{[{\rm Ker}(f)]}\circ p_{{\rm Ker}(f)}=p_{\Delta _B}\circ f$, with $p_{\Delta _B}$ an isomorphism. Since $f_{[{\rm Ker}(f)]}$ fulfills LO and $p_{{\rm Ker}(f)}$ surjective and thus admissible according to Lemma \ref{surjadm}, by Lemma \ref{gulocomp}, (\ref{gulocomp2}), it follows that $f_{[{\rm Ker}(f)]}\circ p_{{\rm Ker}(f)}$ fulfills LO, thus $f$ fulfills LO by Remark \ref{compizom}.\end{proof}

\begin{corollary} Assume that $f$ fulfills LO and $\nabla _B$ is finitely generated. Then:\begin{enumerate}
\item\label{logu1} if $\{Cg_B(f(\phi ))\ |\ \phi \in V_A({\rm Ker}(f))\}\supseteq {\rm Spec}(B)$, then $f$ fulfills GU;
\item\label{logu2} if $f$ is injective and $\{Cg_B(f(\phi ))\ |\ \phi \in {\rm Spec}(A)\}\supseteq {\rm Spec}(B)$, then $f$ fulfills GU.\end{enumerate}\label{logu}\end{corollary}

\begin{proof} (\ref{logu1}) By Proposition \ref{charlo}, $f_{[\theta ]}$ fulfills LO for all $\theta \in {\rm Con}(A)$, thus for all $\theta \in {\rm Spec}(A)$, hence for all $\theta \in V_A({\rm Ker}(f))$. By Remark \ref{fthetafbeta}, this means that, for all $\theta \in V_A({\rm Ker}(f))$, $f_{Cg_B(f(\theta ))}$ fulfills LO, hence, for all $\beta \in {\rm Spec}(B)$, $f_{\beta }$ fulfills LO, therefore $f$ fulfills GU by Proposition \ref{fbeta}.

\noindent (\ref{logu2}) By (\ref{logu1}) and the fact that, if $f$ is injective, then ${\rm Ker}(f)=\Delta _A$, so $V_A({\rm Ker}(f))={\rm Spec}(A)$.\end{proof}

\section{Admissibility, Going Up and Lying Over in Different Kinds of Congruence--modular Equational Classes}
\label{varieties}

In this section, we point out certain kinds of congruence--modular equational classes in which all morphisms are admissible, and others in which all admissible morphisms fulfill GU and LO. From these results we obtain some classes in which all morphisms are admissible and fulfill GU and LO. Throughout this section, $A,B$ shall be members of ${\cal C}$ and $f:A\rightarrow B$ shall be a morphism in ${\cal C}$.

\begin{lemma} If ${\cal C}$ is congruence--modular and $\theta \in {\rm Con}(A)$, then the following are equivalent:\begin{enumerate}
\item\label{1.5(1)} $\theta \in {\rm Spec}(A)$;
\item\label{1.5(2)} for all $a,b,c,d\in A$, $[Cg_A(a,b),Cg_A(c,d)]_A\subseteq \theta $ implies $Cg_A(a,b)\subseteq \theta $ or $Cg_A(c,d)\subseteq \theta $.\end{enumerate}\label{1.5}\end{lemma}

\begin{proof} (\ref{1.5(1)})$\Rightarrow $(\ref{1.5(2)}): Trivial.

\noindent (\ref{1.5(2)})$\Rightarrow $(\ref{1.5(1)}): Assume by absurdum that (\ref{1.5(2)}) is satisfied, but there exist $\alpha ,\beta \in {\rm Con}(A)$ that $[\alpha ,\beta ]\subseteq \theta $, $\alpha \nsubseteq \theta $ and $\beta \nsubseteq \theta $. Then there exist $(a,b)\in \alpha $ and $(c,d)\in \beta $ with $(a,b)\notin \theta $ and $(c,d)\notin \theta $, so that $Cg_A(a,b)\subseteq \alpha $, $Cg_A(c,d)\subseteq \beta $, $Cg_A(a,b)\nsubseteq \theta $ and $Cg_A(c,d)\nsubseteq \theta $, which imply $[Cg_A(a,b),Cg_A(c,d)]_A\subseteq [\alpha ,\beta ]_A\subseteq \theta $ by Proposition \ref{1.3}, and $[Cg_A(a,b),Cg_A(c,d)]_A\nsubseteq \theta $ by the hypothesis of this implication; we have a contradiction. So $\theta \in {\rm Spec}(A)$.\end{proof}

Let $I$ be a non--empty set and, for each $i\in I$, let $p_i$ and $q_i$ be terms of arity $4$ from ${\bf L}_{\tau }$. Following \cite[Section 2]{aglbak}, we call $\{(p_i,q_i)\ |\ i\in I\}$ a {\em system of congruence intersection terms without parameters} for ${\cal C}$ iff, for any member $M$ of ${\cal C}$ and all $a,b,c,d\in M$, $\displaystyle Cg_M(a,b)\cap Cg_M(c,d)=\bigvee _{i\in I}Cg_M(p_i^M(a,b,c,d),q_i^M(a,b,c,d))$.

\begin{theorem}{\rm \cite[Theorem 2.4]{aglbak}} If ${\cal C}$ has a system of congruence intersection terms without parameters, then ${\cal C}$ is congruence--distributive.\label{syscondistrib}\end{theorem}

\begin{proposition} If ${\cal C}$ has a system of congruence intersection terms without parameters, then any morphism in ${\cal C}$ is admissible.\label{syscontpar}\end{proposition}

\begin{proof} Let $\{(p_i,q_i)\ |\ i\in I\}$ be a system of congruence intersection terms without parameters for ${\cal C}$. Let $\psi \in {\rm Spec}(B)$ and $a,b,c,d\in A$ such that $[Cg_A(a,b),Cg_A(c,d)]_A\subseteq f^*(\psi )$. By Theorems \ref{syscondistrib} and \ref{distrib}, this means that $Cg_A(a,b)\cap Cg_A(c,d)\subseteq f^*(\psi )$, thus $\displaystyle \bigvee _{i\in I}Cg_A(p_i^A(a,b,c,d),q_i^A(a,b,c,d))\subseteq f^*(\psi )$, so that, for all $i\in I$, $Cg_A(p_i^A(a,b,c,d),q_i^A(a,b,c,d))\subseteq f^*(\psi )$, that is, for all $i\in I$, $(p_i^A(a,b,c,d),q_i^A(a,b,c,d))\in f^*(\psi )$, thus, for all $i\in I$, $(p_i^B(f(a),f(b),f(c),f(d)),q_i^B(f(a),f(b),f(c),f(d))))=(f(p_i^A(a,b,c,d)),f(q_i^A(a,b,c,d)))\in \psi $, hence, for all $i\in I$, $Cg_B(p_i^B(f(a),f(b),f(c),f(d)),q_i^B(f(a),f(b),f(c),f(d))))=(f(p_i^A(a,b,c,d)),f(q_i^A(a,b,c,d)))\subseteq \psi $, hence $\displaystyle \bigvee _{i\in I}Cg_B(p_i^B(f(a),f(b),f(c),f(d)),q_i^B(f(a),f(b),f(c),f(d))))=(f(p_i^A(a,b,c,d)),f(q_i^A(a,b,c,d)))\subseteq \psi $, therefore $[Cg_B(f(a),f(b)),Cg_B(f(c),f(d))]_B=Cg_B(f(a),f(b))\cap Cg_B(f(c),f(d))\subseteq \psi $, hence $Cg_B(f(a),f(b))\subseteq \psi $ or $Cg_B(f(c),f(d))\subseteq \psi $, that is $(f(a),f(b))\in \psi $ or $(f(c),f(d))\in \psi $, so $(a,b)\in f^*(\psi )$ or $(c,d)\in f^*(\psi )$, thus $Cg_A(a,b)\subseteq f^*(\psi )$ or $Cg_A(c,d)\subseteq f^*(\psi )$. By Lemma \ref{1.5}, it follows that $f^*(\psi )\in {\rm Spec}(A)$, thus $f$ is admissible.\end{proof}

We recall that the compact elements of the lattice ${\rm Con}(A)$ are exactly the finitely generated congruences of $A$. We shall denote by ${\rm Con}_w(A)$ the set of the finitely generated congruences of $A$. Clearly, $({\rm Con}_w(A),\vee ,\Delta _A)$ is a lower bounded join--semilattice.

We say that ${\cal C}$ has {\em compact intersection property} (abbreviated {\em CIP}) iff, for any algebra $M$ from ${\cal C}$, the intersection of every two compact congruences of $M$ is a compact congruence of $M$. We say that ${\cal C}$ has {\em principal intersection property} (abbreviated {\em PIP}) iff, for any algebra $M$ from ${\cal C}$, the intersection of every two principal congruences of $M$ is a principal congruence of $M$.

\begin{remark} If ${\cal C}$ is congruence--distributive and has the PIP, then ${\cal C}$ has the CIP, because, if $A$ is congruence--distributive, then, for any $n,k\in \N ^*$ and any $a_1,\ldots ,a_n,b_1,\ldots ,b_n,c_1,\ldots ,c_k,d_1,\ldots ,d_k\in A$, $Cg_A(\{(a_1,b_1),\ldots ,$\linebreak $\displaystyle (a_n,b_n)\}\cap Cg_A(\{(c_1,d_1),\ldots ,(c_k,d_k)\})=(\bigvee _{i=1}^nCg_A(a_i,b_i))\cap (\bigvee _{j=1}^kCg_A(c_j,d_j))=\bigvee _{i=1}^n\bigvee _{j=1}^k(Cg_A(a_i,b_i))\cap Cg_A(c_j,d_j))$. See also \cite[p. 109]{aglbak}.\label{pipcip}\end{remark}

\begin{proposition}{\rm \cite{aglbak}} Any congruence--distributive equational class with CIP has a system of congruence intersection terms without parameters.\label{syscip}\end{proposition}

\begin{corollary} Any congruence--distributive equational class with PIP has a system of congruence intersection terms without parameters.\label{syspip}\end{corollary}

\begin{corollary}\begin{enumerate}
\item\label{cippip1} If ${\cal C}$ is congruence--distributive and has the CIP, then every morphism in ${\cal C}$ is admissible. 
\item\label{cippip2} If ${\cal C}$ is congruence--distributive and has the PIP, then every morphism in ${\cal C}$ is admissible.\end{enumerate}\label{cippip}\end{corollary}

Following \cite[Chapter $4$]{bj} and \cite[Chapter IV, Section $9$]{bur}, we call ${\cal C}$ a {\em discriminator variety} iff there exists a ternary term $t$ from ${\bf L}_{\tau }$ such that, for every subdirectly irreducible algebra $M$ in ${\cal C}$ and all $a,b,c\in M$: $t^M(a,b,c)=\begin{cases}a, & a\neq b,\\ c, & a=b.\end{cases}$

\begin{lemma}\begin{itemize}
\item {\rm \cite[Theorem 9.4, p. 166]{bur}} If ${\cal C}$ is a discriminator variety, then ${\cal C}$ is congruence--distributive, and there exists a term $s$ of arity $4$ in ${\bf L}_{\tau }$ such that, for any member $M$ of ${\cal C}$ and any $a,b,c,d\in M$, $Cg_M(a,b)\cap Cg_M(c,d)=Cg_M(s^M(a,b,c,d),c)$.
\item {\rm \cite[Corollary 2.7]{aglbak}} If ${\cal C}$ is congruence--distributive, then: ${\cal C}$ has the PIP iff there exist terms ${\cal P}$ and $q$ of arity $4$ from ${\bf L}_{\tau }$ such that, for any algebra $M$ from ${\cal C}$ and any $a,b,c,d\in M$, $Cg_M(a,b)\cap Cg_M(c,d)=Cg_M(p^M(a,b,c,d),q^M(a,b,c,d))$.\end{itemize}\label{termpip}\end{lemma}

Following \cite{aglbak}, we call ${\cal C}$ a {\em filtral variety} iff, for any up--directed set $(I,\leq )$ and any family $(M_i)_{i\in I}$ of subdirectly irreducible algebras from ${\cal C}$, if $S$ is a subdirect product of the family $(M_i)_{i\in I}$, then every congruence of $S$ is of the form $\{((a_i)_{i\in I},(b_i)_{i\in I})\in S\ |\ \{j\in I\ |\ a_j=b_j\}\subseteq F\}$ for some filter $F$ of $(I,\leq )$.

We recall that a {\em (commutative) residuated lattice} is an algebra $(R,\vee ,\wedge ,\odot ,\rightarrow ,0,1)$ of type $(2,2,2,2,0,0)$, in which $(R,\vee ,\wedge ,0,1)$ is a bounded lattice, $(R,\odot ,0,1)$ is a commutative monoid, and each $a,b,c\in R$ fulfill the {\em law of residuation}: $a\leq b\rightarrow c$ iff $a\odot b\leq c$, where $\leq $ is the partial order of the underlying lattice of $R$. For the results on residuated lattices that we use in what follows, we refer the reader to \cite{gal}, \cite{jits}, \cite{kow}. Residuated lattices form a semi--degenerate congruence--distributive equational class, which includes BL--algebras and MV--algebras.

Throughout the rest of this section, $R$ shall be a residuated lattice. For all $a,b\in R$ and all $n\in \N $, we denote by $a\leftrightarrow b=(a\rightarrow b)\wedge (b\rightarrow a)$, $a^0=1$ and $a^{n+1}=a^n\odot a$. We shall denote by ${\rm Filt}(R)$ the set of the {\em filters} of $R$, that is the non--empty subsets of $R$ which are closed with respect to $\odot $ and to upper bounds. Then $({\rm Filt}(R),\vee ,\cap ,\{1\},R)$ is a complete distributive lattice, with $\vee $ defined as in the case of bounded lattices. The map $F\mapsto \sim _F=\{(a,b)\in R^2\ |\ a\leftrightarrow b\in F\}$ is a bounded lattice isomorphism from ${\rm Filt}(R)$ to ${\rm Con}(R)$. For any $a\in R$, we shall denote by $[a)$ the principal filter of $R$ generated by $a$: $[a)=\{x\in R\ |\ (\exists \, n\in \N )\, (a^n\leq x)\}$. For all $a,b\in R$: $[a)\cap [b)=[a\vee b)$, $\sim _{[a)}=Cg_A(a,1)$ and $Cg_A(a,b)=Cg_A(a\leftrightarrow b,1)$, hence, for all $a,b,x,y\in R$, $Cg_A(a,b)\cap Cg_A(x,y)=Cg_A(a\leftrightarrow b,1)\cap Cg_A(x\leftrightarrow y,1)=\sim _{[a\leftrightarrow b)}\cap \sim _{[x\leftrightarrow y)}=\sim _{[a\leftrightarrow b)\cap [x\leftrightarrow y)}=\sim _{[(a\leftrightarrow b)\vee (x\leftrightarrow y))}=Cg_A((a\leftrightarrow b)\vee (x\leftrightarrow y),1)$.
 
\begin{example}\begin{itemize}
\item By the above, the class of residuated lattices is congruence--distributive and has the PIP.
\item The class of bounded distributive lattices is congruence--distributive and, by \cite{aglbak}, it has the PIP.
\item By \cite[Example 2.11]{aglbak}, any filtral variety has the CIP.
\item By Lemma \ref{termpip}, any discriminator variety is congruence--distributive and has the PIP.\end{itemize}\label{pip}\end{example}

Following \cite[p. 382]{bj}, we say that ${\cal C}$ has {\em equationally definable principal congruences} (abbreviated, {\em EDPC}) iff, for any algebra $M$ from ${\cal C}$, there exist an $n\in \N ^*$ and terms $p_1,q_1,\ldots ,p_n,q_n$ of arity $4$ from ${\bf L}_{\tau }$ such that, for all $a,b\in M$, $Cg_M(a,b)=\{(c,d)\in M^2\ |\ (\forall \, i\in \overline{1,n})\, (p_i(a,b,c,d)=q_i(a,b,c,d))\}$.

\begin{theorem}{\rm \cite{blkpgz}} If ${\cal C}$ has EDPC, then ${\cal C}$ is is congruence--distributive.\label{edpcdistrib}\end{theorem}

\begin{example} Here are some examples of varieties with EDPC, from \cite{blkpgz}, \cite[Theorem 2.8]{bj} and \cite{lpt}:\begin{itemize}
\item distributive lattices, residuated lattices;
\item discriminator varieties, which include: Boolean algebras, $n$--valued Post algebras, $n$--valued \L ukasiewicz algebras, $n$--valued MV--algebras, relation algebras, monadic algebras, $n$--dimensional cylindric algebras, ${\rm G\ddot{o}del}$ residuated lattices;
\item dual discriminator varieties;
\item filtral varieties;
\item implication algebras, de Morgan algebras, Hilbert algebras, Brouwerian semilattices, Heyting algebras, modal algebras.\end{itemize}\label{edpc}\end{example}

Let $(L,\vee ,0)$ be a lower bounded join--semilattice. $L$ is said to be {\em dually Browerian} iff it has a binary derivative operation $\dot{-}$ such that, for all $x\in L$, $a\dot{-}b\leq x$ iff $a\leq b\vee x$.

\begin{proposition}{\rm \cite{kp}} ${\cal C}$ has EDPC iff, for any algebra $M$ from ${\cal C}$, ${\rm Con}_w(M)$ is dually Browerian. In this case, if $M$ is a member of ${\cal C}$ and $n\in \N ^*$ and $p_1,q_1,\ldots ,p_n,q_n$ are the terms of arity $4$ which define the principal congruences of $M$ as above, then, for any $a,b,c,d\in M$, $\displaystyle Cg_M(c,d)\dot{-}Cg_M(a,b)=\bigvee _{i=1}^nCg_M(p_i(a,b,c,d),q_i(a,b,c,d))$.\label{charedpc}\end{proposition}

\begin{proposition} If ${\cal C}$ is semi--degenerate and has EDPC, then every admissible morphism in ${\cal C}$ fulfills GU.\label{edpcgu}\end{proposition}

\begin{proof} By Theorem \ref{edpcdistrib}, ${\cal C}$ is congruence--distributive. By Corollary \ref{sufembed}, (\ref{sufembed2}), it is sufficient to prove that every admissible canonical embedding in ${\cal C}$ fulfills GU. Let $B$ be a member of ${\cal C}$, $A$ be a subalgebra of $B$ and $i:A\rightarrow B$ be the canonical embedding. Let $\phi \in {\rm Spec}(A)$ and $\psi $ be a maximal element of the set $\{\theta \in {\rm Con}(B)\setminus \{\nabla _B\}\ |\ \theta \cap (\nabla _A\setminus \phi )=\emptyset \}$. Then $\emptyset =\psi \cap (A^2\setminus \phi )=(\psi \cap A^2)\setminus (\psi \cap \phi )$, thus $\psi \cap A^2\subseteq \psi \cap \phi \subseteq \phi $. Now assume by absurdum that $\phi \nsubseteq \psi \cap A^2$, thus there exists $(x,y)\in \phi \setminus (\psi \cap A^2)$, so that $(x,y)\in \phi\subseteq A^2$ and $(x,y)\notin \psi \cap A^2$, hence $(x,y)\notin \psi $, thus $\psi \nsubseteq \psi \vee Cg_B(x,y)$, therefore $(\psi \vee Cg_B(x,y))\cap (A^2\setminus \phi )\neq \emptyset $ or $\psi \vee Cg_B(x,y)=\nabla _B$, by the maximality of $\psi $. Since $\phi \in {\rm Spec}(A)$, so $\phi \subsetneq \nabla _A$ and thus $\nabla _B\cap (A^2\setminus \phi )=A^2\setminus \phi \neq \emptyset $, it follows that $(\psi \vee Cg_B(x,y))\cap (A^2\setminus \phi )\neq \emptyset $. Let $(s,t)\in (\psi \vee Cg_B(x,y))\cap (A^2\setminus \phi )$, so that there exist an $n\in \N $ and $(a_1,b_1),\ldots ,(a_n,b_n)\in \psi $ such that $(s,t)\in Cg_B((a_1,b_1),\ldots ,(a_n,b_n))\vee Cg_B(x,y)$, hence $Cg_B(s,t)\subseteq Cg_B((a_1,b_1),\ldots ,(a_n,b_n))\vee Cg_B(x,y)$. Since $Cg_B(s,t),Cg_B((a_1,b_1),\ldots ,(a_n,b_n)),Cg_B(x,y)\in {\rm Con}_w(A)$ and, by Proposition \ref{charedpc}, ${\rm Con}_w(A)$ is a dually Browerian join--semilattice, it follows that $Cg_B(s,t)\dot{-}Cg_B(x,y)\subseteq Cg_B((a_1,b_1),\ldots ,(a_n,b_n))\subseteq \psi $, hence $\nabla _A\cap (Cg_B(s,t)\dot{-}Cg_B(x,y))\subseteq \nabla _A\cap \psi \subseteq \phi $. Let $n\in \N $ and $p_1,q_1,\ldots ,p_n,q_n$ be the terms in the equations which define the principal congruences of $A$, as written above. $x,y,s,t\in A$, thus, for every $i\in \overline{1,n}$, we may write $Cg_A(p_i(x,y,s,t),q_i(x,y,s,t))\subseteq Cg_B(p_i(x,y,s,t),q_i(x,y,s,t))$, so, by Proposition \ref{charedpc}, $\displaystyle Cg_A(s,t)\dot{-}Cg_A(x,y)=\bigvee _{i=1}^nCg_A(p_i(x,y,s,t),q_i(x,y,s,t))\subseteq \bigvee _{i=1}^nCg_B(p_i(x,y,s,t),q_i(x,y,s,t))=Cg_B(s,t)\dot{-}Cg_B(x,y)$, hence $Cg_A(s,t)\dot{-}Cg_A(x,y)\subseteq \nabla_A\cap (Cg_B(s,t)\dot{-}Cg_B(x,y))\subseteq \phi $, so $\displaystyle \bigvee _{i=1}^nCg_A(p_i(x,y,s,t),$\linebreak $q_i(x,y,s,t))\subseteq \phi $, thus, for all $i\in \overline{1,n}$, $Cg_A(p_i(x,y,s,t),q_i(x,y,s,t))\subseteq \phi $, so, for all $i\in \overline{1,n}$, $(p_i(x,y,s,t),$\linebreak $q_i(x,y,s,t))\in \phi $, hence, for all $i\in \overline{1,n}$, $p_i(x/\phi ,y/\phi ,s/\phi ,t/\phi )=p_i(x,y,s,t)/\phi =q_i(x,y,s,t)/\phi =q_i(x/\phi ,y/\phi ,$\linebreak $s/\phi ,t/\phi )$, which means that $(s/\phi ,t/\phi )\in Cg_{A/\phi }(x/\phi ,y/\phi )$. But $(x,y)\in \phi $, so that $x/\phi =y/\phi $, hence $Cg_{A/\phi }(x/\phi ,y/\phi )=Cg_{A/\phi }(x/\phi ,x/\phi )=\Delta _{A/\phi }$. Thus $(s/\phi ,t/\phi )\in \Delta _{A/\phi }$, that is $s/\phi =t/\phi $, so $(s,t)\in \phi $. Therefore $(s,t)\in \phi \cap (\Delta _A\setminus \phi )=\emptyset $; we have a contradiction. Hence $\phi \subseteq \psi \cap A^2$, therefore $\psi \cap A^2=\phi $. By Proposition \ref{altacargu}, it follows that $i$ fulfills GU, which concludes the proof.\end{proof}

\begin{corollary} If ${\cal C}$ is semi--degenerate and has EDPC, then every admissible morphism in ${\cal C}$ fulfills LO.\label{edpclo}\end{corollary}

\begin{proof} By Proposition \ref{edpcgu} and Corollary \ref{classgulo}.\end{proof}

\begin{corollary}\begin{itemize}
\item If ${\cal C}$ is semi--degenerate and has EDPC and CIP, then any morphism in ${\cal C}$ is admissible and fulfills GU and LO.
\item If ${\cal C}$ is semi--degenerate and has EDPC and PIP, then any morphism in ${\cal C}$ is admissible and fulfills GU and LO.\end{itemize}\label{pipcipedpc}\end{corollary}

\begin{proof} By Proposition \ref{charedpc}, Theorem \ref{edpcdistrib} and Corollary \ref{classgulo}.\end{proof}

\begin{corollary}\begin{itemize}
\item Any morphism in the class of residuated lattices is admissible and fulfills GU and LO.
\item Any morphism in the class of bounded distributive lattices is admissible and fulfills GU and LO.
\item Any morphism in a semi--degenerate filtral variety is admissible and fulfills GU and LO.
\item Any morphism in a semi--degenerate discriminator variety is admissible and fulfills GU and LO.\end{itemize}\end{corollary}

\begin{proof} By Corollary \ref{pipcipedpc} and Examples \ref{pip} and \ref{edpc}.\end{proof}

The previous corollary implies the results from \cite{bel} and \cite{rada} which say that any morphism of MV--algebras or BL--algebras is admissible and fulfills GU and LO.

\end{document}